\documentclass{amsart}

\usepackage{amsmath}
\usepackage{amssymb}
\usepackage{xcolor}
\usepackage{tikz}
\usetikzlibrary{cd}
\usetikzlibrary{knots}
\usetikzlibrary{calc}

\tikzset{string/.style={ultra thick}}
\tikzset{smallstring/.style={thick,scale=0.75,every node/.style={transform shape}}}
\tikzset{
    triple/.style args={[#1] in [#2] in [#3]}{
        #1,preaction={preaction={draw,#3},draw,#2}
    }
}
\tikzset{
    quadruple/.style args={[#1] in [#2] in [#3] in [#4]}{
        #1,preaction={preaction={preaction={draw,#4},draw,#3}, draw,#2}
    }
} 
\tikzset{
	super thick/.style={line width=3pt},
	more thick/.style={line width=1pt},
}

\usepackage{graphicx}
\usepackage{fullpage}
\usepackage[T1]{fontenc}
\usepackage[final]{microtype}
\usepackage{libertine}
\usepackage[libertine]{newtxmath}
\usepackage{ifthen}
\usepackage[all]{xy}

\usepackage{xcolor}
\definecolor{dark-red}{rgb}{0.7,0.25,0.25}
\definecolor{dark-blue}{rgb}{0.15,0.15,0.55}
\definecolor{medium-blue}{rgb}{0,0,0.65}
\definecolor{DarkGreen}{RGB}{0,150,0}

\usepackage[pdftex,plainpages=false,hypertexnames=false,pdfpagelabels,breaklinks]{hyperref}
\hypersetup{
   colorlinks, linkcolor={purple},
   citecolor={medium-blue}, urlcolor={medium-blue}
}
\newcommand{\doi}[1]{\href{https://dx.doi.org/#1}{{\small DOI:#1}}}

\newcommand{\googlebooks}[1]{(preview at \href{https://books.google.com/books?id=#1}{google books})}

\newcommand{\numdam}[1]{}

\theoremstyle{plain}
\newtheorem{prop}{Proposition}[section]

\newtheorem{thm}[prop]{Theorem}
\newtheorem{lem}[prop]{Lemma}

\numberwithin{equation}{section}

\theoremstyle{remark}
\newtheorem{example}[prop]{Example}

\newtheorem{remark}[prop]{Remark}

\theoremstyle{definition}
\newtheorem{defn}[prop]{Definition}         

\newtheorem{nota}[prop]{Notation}

\newcommand{\sslash}{\mathbin{/\mkern-6mu/}}

\DeclareMathOperator{\mate}{mate}

\DeclareMathOperator{\Tr}{Tr}

\newcommand{\id}{\mathbf{1}}

\renewcommand{\Vec}{{\mathsf {Vec}}}
\newcommand{\Rep}{{\mathsf {Rep}}}

\def\semicolon{;}
\def\applytolist#1{
    \expandafter\def\csname multi#1\endcsname##1{
        \def\multiack{##1}\ifx\multiack\semicolon
            \def\next{\relax}
        \else
            \csname #1\endcsname{##1}
            \def\next{\csname multi#1\endcsname}
        \fi
        \next}
    \csname multi#1\endcsname}

\def\calc#1{\expandafter\def\csname c#1\endcsname{{\mathcal #1}}}
\applytolist{calc}QWERTYUIOPLKJHGFDSAZXCVBNM;
\def\bbc#1{\expandafter\def\csname bb#1\endcsname{{\mathbb #1}}}
\applytolist{bbc}QWERTYUIOPLKJHGFDSAZXCVBNM;
\def\bfc#1{\expandafter\def\csname bf#1\endcsname{{\mathbf #1}}}
\applytolist{bfc}QWERTYUIOPLKJHGFDSAZXCVBNM;

\makeatletter
\newlength{\L@UnitsRaiseDisplaystyle}
\newlength{\L@UnitsRaiseTextstyle}
\newlength{\L@UnitsRaiseScriptstyle}
\DeclareRobustCommand*{\@UnitsNiceFrac}[3][]{%
  \ifthenelse{\boolean{mmode}}{%
    \settoheight{\L@UnitsRaiseDisplaystyle}{%
      \ensuremath{\displaystyle#1{M}}%
    }%
    \settoheight{\L@UnitsRaiseTextstyle}{%
      \ensuremath{\textstyle#1{M}}%
    }%
    \settoheight{\L@UnitsRaiseScriptstyle}{%
      \ensuremath{\scriptstyle#1{M}}%
    }%
    \settoheight{\@tempdima}{%
      \ensuremath{\scriptscriptstyle#1{M}}%
    }%
    \addtolength{\L@UnitsRaiseDisplaystyle}{%
      -\L@UnitsRaiseScriptstyle%
    }%
    \addtolength{\L@UnitsRaiseTextstyle}{%
      -\L@UnitsRaiseScriptstyle%
    }%
    \addtolength{\L@UnitsRaiseScriptstyle}{-\@tempdima}%
    \mathchoice
      {%
        \raisebox{\L@UnitsRaiseDisplaystyle}{%
          \ensuremath{\scriptstyle#1{#2}}%
        }%
      }%
      {%
        \raisebox{\L@UnitsRaiseTextstyle}{%
          \ensuremath{\scriptstyle#1{#2}}%
        }%
      }%
      {%
        \raisebox{\L@UnitsRaiseScriptstyle}{%
          \ensuremath{\scriptscriptstyle#1{#2}}%
        }%
      }%
      {%
        \raisebox{\L@UnitsRaiseScriptstyle}{%
          \ensuremath{\scriptscriptstyle#1{#2}}%
        }%
      }%
    \mkern-2mu{\sslash}\mkern-1mu%
    \bgroup
      \mathchoice
        {\scriptstyle}%
        {\scriptstyle}%
        {\scriptscriptstyle}%
        {\scriptscriptstyle}%
      #1{#3}%
    \egroup
  }%
  {%
    \settoheight{\L@UnitsRaiseTextstyle}{#1{M}}%
    \settoheight{\@tempdima}{%
      \ensuremath{%
        \mbox{\fontsize\sf@size\z@\selectfont#1{M}}%
      }%
    }%
    \addtolength{\L@UnitsRaiseTextstyle}{-\@tempdima}%
    \raisebox{\L@UnitsRaiseTextstyle}{%
      \ensuremath{%
        \mbox{\fontsize\sf@size\z@\selectfont#1{#2}}%
      }%
    }%
    \ensuremath{\mkern-2mu}{\sslash}\ensuremath{\mkern-1mu}%
    \ensuremath{%
      \mbox{\fontsize\sf@size\z@\selectfont#1{#3}}%
    }%
  }%
}

\DeclareRobustCommand*{\nicefrac}{\@UnitsNiceFrac}%
\makeatother

\usepackage{amsmath}
\usepackage{amsfonts}
\usepackage{amssymb}
\usepackage{mathrsfs}

\usetikzlibrary{calc, knots}

\tikzset{smallstring/.style={thick,scale=0.75,every node/.style={transform shape}}}
\newcommand{\Cat}{\textbf{Cat}}
\newcommand{\vmoncat}{\boldsymbol{\mathcal V}\textbf{MonCat}}
\newcommand{\ggrvmoncat}{\boldsymbol{\mathcal V}\textbf{MonCat}_G}
\newcommand{\gextvmoncat}{\boldsymbol{\mathcal V}\textbf{MonCat}_G^{\mathcal A}}
\newcommand{\vmodtens}{\boldsymbol{\mathcal V}\textbf{ModTens}}
\newcommand{\ggrvmodtens}{\boldsymbol{\mathcal V}\textbf{ModTens}_G}
\newcommand{\gextvmodtens}{\boldsymbol{\mathcal V}\textbf{ModTens}_G^{\mathcal A}}
\newcommand{\commentout}[1]{}
\newcommand{\comph}{\circ}
\newcommand{\compv}{}

\newcommand{\E}{\left(}
\newcommand{\R}{\right)}
\newcommand{\DF}{\left[}
  \newcommand{\FD}{\right]}

\newcommand{\blank}{{-}}
\newcommand{\mor}{\xrightarrow}
\newcommand{\qtand}{\quad \text{and} \quad}
\newcommand{\tand}{\text{ and }}

\newcommand{\restrict}{|}
\DeclareMathOperator{\Forget}{Forget}
\let\id\relax
\DeclareMathOperator{\id}{id}

\tikzset{knot gap = 10}

\begin{document}
\title{A characterization of braided enriched monoidal categories}
\author{Zachary Dell}
\address[Zachary Dell]{
  Department of Mathematics, 
  Ohio State University
}
\email{\href{mailto:dell.31@osu.edu}{dell.31@osu.edu}}
\date{}

\begin{abstract}
  We construct an equivalence between the 2-categories $\boldsymbol{\mathcal V}\textbf{MonCat}_G^{\mathcal A}$ of rigid $\mathcal V$-monoidal categories for a braided monoidal category $\mathcal V$ and $\boldsymbol{\mathcal V}\textbf{ModTens}_G^{\mathcal A}$ of oplax braided functors from $\mathcal V$ into the Drinfeld centers of ordinary rigid monoidal categories.  The 1-cells in each are the respective lax monoidal functors, and the 2-cells are the respective monoidal natural transformations.  Our proof also gives an equivalence in the case that we consider only strong monoidal 1-cells on both sides.  The 2-categories $\boldsymbol{\mathcal V}\textbf{MonCat}_G^{\mathcal A}$ and $\boldsymbol{\mathcal V}\textbf{ModTens}_G^{\mathcal A}$ have $G$-graded analogues. We also get an equivalence of 2-categories between $G$-extensions of some fixed $\mathcal V$-monoidal category $\mathcal A$, and $G$-extensions of some fixed $\mathcal V$-module tensor category $(A , \mathcal F_A^Z)$.
\end{abstract}
\maketitle{}

\section{Introduction}
The theory of monoidal categories enriched in a symmetric monoidal category $\mathcal V$ is well established \cite{MR749468, MR2177301, MR2219705, MR3775482}.  The articles (\cite{MR3961709, 1809.09782, 1910.03178}) have developed the theory when $\mathcal V$ need only be braided, without symmetry.  Remark 5.2 of \cite{MR1250465} outlines the idea of braided enriched categories and mentions that they form a 2-category.  Braided enriched monoidal categories were classified in \cite{MR3961709} in terms of strongly unital oplax braided monoidal functors from the enriching category $\mathcal V$ into the Drinfeld center of an ordinary monoidal category:
\begin{thm}[\cite{MR3961709}]
  \label{Main MP Theorem}
  Let $\mathcal V$ be a braided monoidal category.
  There is a bijective correspondence 
  \[
  \left\{\, 
  \parbox{5cm}{\rm Rigid $\mathcal V$-monoidal categories $\mathcal A$, such that $x\mapsto \mathcal A (1_{\mathscr C} \to x)$ admits a left adjoint}\,\left\}
  \,\,\,\,\cong\,\,
  \left\{\,\parbox{8cm}{\rm Pairs $\E A, \mathcal F^{\scriptscriptstyle Z} \R$ with $A$ a rigid monoidal category and $\mathcal F_A^{\scriptscriptstyle Z}: \mathcal V \to Z \E A \R$ braided oplax monoidal, such that $\mathcal F_A \coloneqq \mathcal F_A^{\scriptscriptstyle Z} \circ R$ admits a right adjoint}\,\right\},
  \right.\right.
  \]
  where $R: Z \E A \R \to A$ is the forgetful functor.
\end{thm}

Note that we are writing composition from left to right, as in \cite{MR3961709, 1809.09782}, but reversed to the typical convention. From a $\mathcal V$-monoidal category $\mathcal A$, we can take the underlying category $A$, which has the same objects as $\mathcal A$ and hom sets $A \E a \to b \R \coloneqq \mathcal V \E 1_{\mathcal V} \to \mathcal A \E a \to b \R \R$.  The underlying category of $\mathcal A$ is denoted by $\mathcal A^{\mathcal V}$ in \cite{MR3961709}, but here we merely change the font; there are no ambiguous situations, and our focus is not on the distinction between $\mathcal A$ and its underlying category $A$.  As shown in \cite{MR3961709}, $\mathcal F_A$ is taken to be the left adjoint of $A \E 1_A \to \blank \R$ and shown to lift to $\mathcal F_A^Z :\mathcal V \to Z \E A \R$.
In the other direction, given $\E A , \mathcal F_A^Z \R$ we can construct a $\mathcal V$-monoidal category $\mathcal A$ using the same objects as $A$, and hom objects defined via the Yoneda lemma by the natural isomorphisms
\begin{align*}
  \mathcal V \E v \to \mathcal A \E a \to b \R \R \cong A \E a \mathcal F_A \E v \R \to b \R,
\end{align*}
which themselves follow from rigidity and the right adjoint of $\mathcal F$.
The article \cite{1809.09782} characterizes when the functor $\mathcal F^Z$ is in fact \emph{strong} monoidal, and also develops a notion of completeness for $\mathcal V$-monoidal categories.  In the absence of monoidal structures, \cite{MR3961709} states an equivalence of 2-categories; we go one categorical level higher to the case of $\mathcal V$-monoidal categories:
\begin{thm}
  \label{Main theorem}
  Theorem \ref{Main MP Theorem} extends to an equivalence of 2-categories.
\end{thm}

This theorem was announced in \cite{1910.03178}. A recent article, \cite{2104.03121}, contains similar results with independent proofs. To start, we construct 2-categories with the sets from Theorem \ref{Main MP Theorem} as 0-cells as follows:
\begin{itemize}
\item The 1-cells on the left hand side are lax $\mathcal V$-monoidal functors.
\item The 1-cells on the right are lax monoidal functors equipped with an extra \emph{action-coherence} natural transformation, introduced in Section \ref{Define bicategories}. The action-coherence natural transformation encodes the interaction between the functor $\mathcal F^Z$, which is \emph{oplax}, and the 1-cells, which are \emph{lax}.  
\item On both sides, 2-cells are the corresponding monoidal natural transformations satisfying the appropriate coherences with the action coherence natural transformations.
\end{itemize}
We then construct a 2-functor between the 2-categories, proceeding as in \cite{MR3961709} on 0-cells.  The map on 2-cells turns out to be straightforward, so the bulk of our work goes into constructing the map on 1-cells between the two 2-categories, verifying that the maps together form a 2-functor, and showing that this 2-functor is a 2-equivalence.  Starting with a lax $\mathcal V$-monoidal functor $\mathcal R : \mathcal A \to \mathcal B$ between $\mathcal V$-monoidal categories, we get an ordinary lax monoidal functor by taking the underlying functor $R = \mathcal R^{\mathcal V}$, and we construct our action-coherence natural transformation $r : F_B \to R \circ F_A$ via, for $v \in \mathcal V$,
\begin{equation*}
  r_v = \mate \E
  \begin{tikzpicture}[baseline=30, smallstring]
    \node (top) at (0,3) {$\mathcal B \E 1_{\mathcal B} \to R \E \mathcal F_A \E v \R \R \R$};
    \node (bot) at (0,0) {$v$};
    \node[draw, rectangle] (eta) at (0,1) {$\eta_v^{\mathcal F_A}$};
    \node[draw, rectangle] (R) at (0,2) {$\mathcal R_{1_{\mathcal A} \to \mathcal F_A \E v \R}$};
    \draw (bot) to (eta);
    \draw (eta) to (R);
    \draw (R) to (top);
  \end{tikzpicture} \R,
\end{equation*}
where the mate is taken under the adjunction
\begin{align*}
  B \E \mathcal F_B \E v \R \to R \E \mathcal F_A \E v \R \R \R \cong \mathcal V \E v \to \mathcal B \E 1_{\mathcal B} \to R \E \mathcal F_A \E v \R \R \R \R.
\end{align*}
Starting instead with a monoidal functor $\E R , \rho \R$ between $\E A , \mathcal F_A \R$ and $\E B , \mathcal F_B \R$ in $\vmodtens$, with action-coherence $r$, we construct a $\mathcal V$-monoidal functor $\mathcal R$:
\begin{equation*}
  \mathcal R_{a \to b} \coloneqq \mate \E
  \begin{tikzpicture}[baseline=35, smallstring]
    \node (top) at (1.5,4) {$R \E b \R$};
    \node (a) at (0,0) {$R \E a \R$};
    \node (bot) at (3,0) {$\DF a \to b \FD_{\mathcal F_B}^{\mathcal A}$};
    \node[draw, rectangle] (r) at (3,0.8) {$r_{\mathcal A \E a \to b \R}$};
    \node[draw, rectangle] (rho) at (1.5,2) {$\rho_{a , \DF a \to b \FD_{\mathcal F_A}^{\mathcal \mathcal A}}$};
    \node[draw, rectangle] (rep) at (1.5,3) {$R \E \epsilon_{a \to b}^{\mathcal F_A} \R$};
    \draw (a) to[in=-90,out=90] (rho.225);
    \draw (bot) to[in=-90,out=90] (r);
    \draw (r) to[in=-90,out=90] (rho.-45);
    \draw (rho) to[in=-90,out=90] (rep);
    \draw (rep) to (top);
  \end{tikzpicture} \R,
\end{equation*}
with mate taken under the adjunction
\begin{align*}
  \mathcal V \E \mathcal A \E a \to b \R \to \mathcal B \E R \E a \R \to R \E b \R \R \R \cong B \E R \E a \R \mathcal F_B \E \mathcal A \E a \to b \R \R \to R \E b \R \R.
\end{align*}
The laxitors for the 1-cells are the same on both sides, interpreted as elements of
\begin{align*}
  B \E R \E a \R R \E b \R \to R \E a b \R \R = \mathcal V \E 1_{\mathcal V} \to \mathcal B \E \mathcal R \E a \R \mathcal R \E b \R \to \mathcal R \E a b \R \R \R.
\end{align*}

In \cite{Lin76, MR3961709}, Theorem \ref{Main Theorem} is addressed without any monoidal structure on the enriched categories by combining the $\E \id r \R \circ \rho$ from the diagram above into a single coherence morphism to be included in the data of a 1-cell. Lastly, in \cite{1910.03178} the authors introduce the notions of a $G$-grading and a $G$-extension of a $\mathcal V$-monoidal category for a finite group $G$.  In particular, they studied $G$-extensions of $\mathcal V$-fusion categories, which were also used by Kong and Zheng to present a unified theory of gapped and gapless edges for 2D topological orders \cite{MR3763324, 1905.04924, 1912.01760}.  In Section \ref{G-gradings} we introduce the necessary definitions and lift Theorem \ref{Main theorem} to the $G$-graded setting.  Then in Section \ref{G-extensions} we go one step further to the case of $G$-extensions of $\mathcal V$-monoidal categories, and again lift Theorem \ref{Main theorem}.

\subsubsection*{Acknowledgments}
We would like to thank Peter Huston for helpful discussions.  The work was partially supported by NSF grants DMS 1654159 and 1927098.


\section{Background}
Let $\mathcal V$ be a braided monoidal category with braiding $\beta$.  We follow the notational conventions set in \cite{MR3961709}; specifically, we suppress $\otimes$ (in $\mathcal V$) and all associators and unitors in $\mathcal V$, with compositions written explicitly as $\circ$, and composition written from left to right.

\subsection{$\mathcal V$-categories}
A $\mathcal V$-category, or $\mathcal V$-enriched category, $\mathcal A$ consists of a collection of objects along with:
\begin{itemize}
\item For each pair of objects $a , b \in \mathcal A$, an associated hom object $\mathcal A \E a \to b \R \in \mathcal V$,
\item For each object $a \in \mathcal A$, an identity element $j_a \in \mathcal V \E 1_{\mathcal V} \to \mathcal A \E a \to a \R \R$, and
\item For each triple of objects $a , b , c \in \mathcal A$, a distinguished composition morphism
  \begin{align*}
    \blank \circ_{\mathcal A} \blank \in \mathcal V \E \mathcal A \E a \to b \R \mathcal A \E b \to c \R \to \mathcal A \E a \to c \R \R
  \end{align*}
\end{itemize}
satisfying the axioms
\begin{itemize}
\item Unitality: For all $a , b \in \mathcal A$,
  \begin{align*}
    \begin{tikzpicture}[baseline=30, smallstring]
      \node (top) at (1,3) {$\mathcal A \E a \to b \R$};
      \node (bot) at (2,0) {$\mathcal A \E a \to b \R$};
      \node[draw, rectangle] (j) at (0,1) {$j_a$};
      \node[draw, rectangle] (circ) at (1,2) {$\blank \circ_{\mathcal A} \blank$};
      \draw (j) to[in=-90,out=90] (circ.225);
      \draw (bot) to[in=-90,out=90] (circ.-45);
      \draw (circ) to[in=-90,out=90] (top);
    \end{tikzpicture}
    =
    \begin{tikzpicture}[baseline=30, smallstring]
      \node(bot) at (0,0) {$\mathcal A \E a \to b \R$};
      \node (top) at (0,3) {$\mathcal A \E a \to b \R$};
      \draw (top) to (bot);
    \end{tikzpicture}
    =
    \begin{tikzpicture}[baseline=30, smallstring]
      \node (top) at (1,3) {$\mathcal A \E a \to b \R$};
      \node (bot) at (0,0) {$\mathcal A \E a \to b \R$};
      \node[draw, rectangle] (j) at (2,1) {$j_b$};
      \node[draw, rectangle] (circ) at (1,2) {$\blank \circ_{\mathcal A} \blank$};
      \draw (j) to[in=-90,out=90] (circ.-45);
      \draw (bot) to[in=-90,out=90] (circ.225);
      \draw (circ) to[in=-90,out=90] (top);
    \end{tikzpicture}
  \end{align*}
\item Associativity of $\E \blank \circ \blank \R$: For all $a , b , c , d$, 
  \begin{align*}
    \begin{tikzpicture}[baseline=30, smallstring]
      \node (ab) at (0,0) {$\mathcal A \E a \to b \R$};
      \node (bc) at (2,0) {$\mathcal A \E b \to c \R$};
      \node (cd) at (4,0) {$\mathcal A \E c \to d \R$};
      \node (top) at (2,3) {$\mathcal A \E a \to d \R$};
      \node[draw, rectangle] (circ1) at (1,1) {$\blank \circ_{\mathcal A} \blank$};
      \node[draw, rectangle] (circ2) at (2,2) {$\blank \circ_{\mathcal A} \blank$};
      \draw (ab) to[in=-90,out=90] (circ1.225);
      \draw (bc) to[in=-90,out=90] (circ1.-45);
      \draw (circ1) to[in=-90,out=90] (circ2.225);
      \draw (cd) to[in=-90,out=90] (circ2.-45);
      \draw (circ2) to[in=-90,out=90] (top);
    \end{tikzpicture}
    =
    \begin{tikzpicture}[baseline=30, smallstring]
      \node (ab) at (0,0) {$\mathcal A \E a \to b \R$};
      \node (bc) at (2,0) {$\mathcal A \E b \to c \R$};
      \node (cd) at (4,0) {$\mathcal A \E c \to d \R$};
      \node (top) at (2,3) {$\mathcal A \E a \to d \R$};
      \node[draw, rectangle] (circ1) at (3,1) {$\blank \circ_{\mathcal A} \blank$};
      \node[draw, rectangle] (circ2) at (2,2) {$\blank \circ_{\mathcal A} \blank$};
      \draw (ab) to[in=-90,out=90] (circ2.225);
      \draw (bc) to[in=-90,out=90] (circ1.225);
      \draw (circ1) to[in=-90,out=90] (circ2.-45);
      \draw (cd) to[in=-90,out=90] (circ1.-45);
      \draw (circ2) to[in=-90,out=90] (top);
    \end{tikzpicture}.
  \end{align*}
\end{itemize}
For objects $a , b$ in a $\mathcal V$-enriched category $\mathcal A$ and $v \in \mathcal V$, a \emph{$v$-graded morphism} from $a$ to $b$ is a morphism $f \in \mathcal V \E v \to \mathcal A \E a \to b \R \R$.  We will give particular attention to the case where $v = 1_{\mathcal V}$, the \emph{$1_{\mathcal V}$-graded morphisms}.

\example[Self-enrichment]
Given a closed monoidal category $\mathcal V$, we can construct a $\mathcal V$-category $\widehat{\mathcal V}$ by starting with the objects of $\mathcal V$, then taking hom objects $\widehat{\mathcal V} \E u \to v \R = \DF u , v \FD$, the internal hom in $\mathcal V$ defined via
\begin{align*}
  \mathcal V \E u w \to v \R = \mathcal V \E w \to \widehat{\mathcal V} \E u \to v \R \R.
\end{align*}
Under this identity adjunction, we also define $j_v \coloneqq \mate \E \id_v \R$ and
\begin{align*}
  \blank \circ_{\widehat{\mathcal V}} \blank \coloneqq
  \mate \E \begin{tikzpicture}[smallstring, baseline=30]
    \node (u) at (0,0) {$u$};
    \node (uv) at (2,0) {$\widehat{\mathcal V} \E u \to v \R$};
    \node (vw) at (4,0) {$\widehat{\mathcal V} \E v \to w \R$};
    \node (top) at (2,3) {$w$};
    \node[draw, rectangle] (ep1) at (1,1) {$\epsilon_{u \to v}^{\mathcal F}$};
    \node[draw, rectangle] (ep2) at (2,2) {$\epsilon_{v \to w}^{\mathcal F}$};
    \draw (u) to[in=-90,out=90] (ep1.225);
    \draw (uv) to[in=-90,out=90] (ep1.-45);
    \draw (vw) to[in=-90,out=90] (ep2.-45);
    \draw (ep1) to[in=-90,out=90] (ep2.225);
    \draw (ep2) to[in=-90,out=90] (top);
  \end{tikzpicture} \R.
\end{align*}
\begin{defn} \label{Underlying category}
  Given a $\mathcal V$-category $\mathcal A$, we can construct an ordinary category $\mathcal A^{\mathcal V}$, the \emph{underlying category} of $\mathcal A$, by giving $\mathcal A^{\mathcal V}$ the same objects as $\mathcal A$, defining $\mathcal A^{\mathcal V} \E a \to b \R \coloneqq \mathcal V \E 1_{\mathcal V} \to \mathcal A \E a \to b \R \R$ for each $a , b \in \mathcal A^{\mathcal V}$, and composition via $f \circ g = \E f g \R \circ \E \blank \circ_{\mathcal A} \blank \R$, for $f \in \mathcal A^{\mathcal V} \E a \to b \R$ and $g \in \mathcal A^{\mathcal V} \E b \to c \R$.  In the following we will often denote the underlying category of a $\mathcal V$-category $\mathcal A$ by a standard font $A$.  
\end{defn}

\subsection{$\mathcal V$-functors and natural transformations}
A \emph{$\mathcal V$-functor} $\mathcal R : \mathcal A \to \mathcal B$ between $\mathcal V$-categories consists of a function on objects and for each $a , b \in \mathcal A$, a morphism $\mathcal R_{a \to b} \in \mathcal V \E \mathcal A \E a \to b \R \to \mathcal B \E \mathcal R \E a \R \to \mathcal R \E b \R \R \R$, satisfying the following axioms:
\begin{itemize}
\item Functoriality: For all $a , b \in \mathcal A$, \commentout{$\E \blank \circ_{\mathcal A} \blank \R \circ \mathcal R_{a \to c} = \E \mathcal R_{a \to b} \mathcal R_{b \to c} \R \circ \E \blank \circ_{\mathcal B} \blank \R$:}
  \begin{align*}
    \begin{tikzpicture}[baseline=30, smallstring]
      \node (ab) at (0,0) {$\mathcal A \E a \to b \R$};
      \node (bc) at (2,0) {$\mathcal A \E b \to c \R$};
      \node (top) at (1,3) {$\mathcal B \E \mathcal R \E a \R \to \mathcal R \E c \R \R$};
      \node[draw, rectangle] (circ) at (1,1) {$\blank \circ_{\mathcal A} \blank$};
      \node[draw, rectangle] (r) at (1,2) {$\mathcal R_{a \to c}$};
      \draw (ab) to[in=-90,out=90] (circ.225);
      \draw (bc) to[in=-90,out=90] (circ.-45);
      \draw (circ) to[in=-90,out=90] (r);
      \draw (r) to[in=-90,out=90] (top);
    \end{tikzpicture}
    =
    \begin{tikzpicture}[baseline=30, smallstring]
      \node (ab) at (0,0) {$\mathcal A \E a \to b \R$};
      \node (bc) at (2,0) {$\mathcal A \E b \to c \R$};
      \node (top) at (1,3) {$\mathcal B \E \mathcal R \E a \R \to \mathcal R \E b \R \R$};
      \node[draw, rectangle] (rab) at (0,1) {$\mathcal R_{a \to b}$};
      \node[draw, rectangle] (rbc) at (2,1) {$\mathcal R_{b \to c}$};
      \node[draw, rectangle] (circ) at (1,2) {$\blank \circ_{\mathcal B} \blank$};
      \draw (ab) to[in=-90,out=90] (rab);
      \draw (bc) to[in=-90,out=90] (rbc);
      \draw (rab) to[in=-90,out=90] (circ.225);
      \draw (rbc) to[in=-90,out=90] (circ.-45);
      \draw (circ) to[in=-90,out=90] (top);
    \end{tikzpicture}
  \end{align*}

\item Unit preserving: For all $a \in \mathcal A$, \commentout{$j_a^{\mathcal A} \circ \mathcal R_{a \to a} = j_{\mathcal R \E a \R}^{\mathcal B}$:}
  \begin{align*}
    \begin{tikzpicture}[baseline=15, smallstring]
      \node (top) at (0,2) {$\mathcal B \E \mathcal R \E a \R \to \mathcal R \E a \R \R$};
      \node[draw, rectangle] (j) at (0,0) {$j_a^{\mathcal A}$};
      \node[draw, rectangle] (r) at (0,1) {$\mathcal R_{a \to a}$};
      \draw (j) to (r);
      \draw (r) to (top);
    \end{tikzpicture}
    =
    \begin{tikzpicture}[baseline=5, smallstring]
      \node (top) at (0,1) {$\mathcal B \E \mathcal R \E a \R \to \mathcal R \E a \R \R$};
      \node[draw, rectangle] (j) at (0,0) {$j_{\mathcal R \E a \R}^{\mathcal B}$};
      \draw (j) to (top);
    \end{tikzpicture}
  \end{align*}

\end{itemize}
We can compose two $\mathcal V$-functors $\mathcal A \mor{\mathcal R} \mathcal B \mor{\mathscr S} \mathscr C$ via $\E \mathcal R \circ \mathscr S \R \E a \R = \mathscr S \E \mathcal R \E a \R \R$ and $\E \mathcal R \circ \mathscr S \R_{a \to b} = \mathcal R_{a \to b} \circ \mathscr S_{\mathcal R \E a \R \to \mathcal R \E b \R}$ for all objects $a , b \in \mathcal A$.
For $\mathcal V$-functors $\mathcal R , \mathscr S : \mathcal A \to \mathcal B$, we define a \emph{$1_{\mathcal V}$-graded natural transformation} $\theta : \mathcal R \Rightarrow \mathscr S$ by assigning to each object $a \in \mathcal A$ a $1_{\mathcal V}$-graded morphism $\theta_a \in \mathcal V \E 1_{\mathcal V} \to \mathcal B \E \mathcal R \E a \R \to \mathscr S \E a \R \R \R$ such that that for all $a , b \in \mathcal A$,
\begin{equation}\label{vtrans nat}
  \begin{tikzpicture}[baseline=35, smallstring]
    \node (ab) at (2,0) {$\mathcal A \E a \to b \R$};
    \node (top) at (1,3) {$\mathcal B \E \mathcal R \E a \R \to \mathscr S \E b \R \R$};
    \node[draw, rectangle] (theta) at (0,1) {$\theta_a$};
    \node[draw, rectangle] (s) at (2,1) {$\mathscr S_{a \to b}$};
    \node[draw, rectangle] (circ) at (1,2) {$\blank \circ_{\mathcal B} \blank$};
    \draw (ab) to[in=-90,out=90] (s);
    \draw (theta) to[in=-90,out=90] (circ.225);
    \draw (s) to[in=-90,out=90] (circ.-45);
    \draw (circ) to[in=-90,out=90] (top);
  \end{tikzpicture}
  =
  \begin{tikzpicture}[baseline=35, smallstring]
    \node (ab) at (0,0) {$\mathcal A \E a \to b \R$};
    \node (top) at (1,3) {$\mathcal B \E \mathcal R \E a \R \to \mathscr S \E b \R \R$};
    \node[draw, rectangle] (theta) at (2,1) {$\theta_b$};
    \node[draw, rectangle] (r) at (0,1) {$\mathcal R_{a \to b}$};
    \node[draw, rectangle] (circ) at (1,2) {$\blank \circ_{\mathcal B} \blank$};
    \draw (ab) to[in=-90,out=90] (r);
    \draw (r) to[in=-90,out=90] (circ.225);
    \draw (theta) to[in=-90,out=90] (circ.-45);
    \draw (circ) to[in=-90,out=90] (top);
  \end{tikzpicture}.
\end{equation}
This condition is the usual naturality square, but written in terms of composition of $1_{\mathcal V}$-graded morphisms.

\begin{defn} \label{Underlying functor}
  Given a $\mathcal V$-functor $\mathcal R : \mathcal A \to \mathcal B$ between $\mathcal V$-categories, we construct the \emph{underlying functor} of $\mathcal R$, an ordinary functor $\mathcal R^{\mathcal V} : A \to B$ between the underlying categories.  We take $\mathcal R^{\mathcal V} \E a \R = \mathcal R \E a \R$ for each $a \in A$, and on morphisms $f \in A \E a \to b \R = \mathcal V \E 1_{\mathcal V} \to \mathcal A \E a \to b \R \R$, we define $\mathcal R \E f \R \coloneqq f \circ \mathcal R_{a \to b} \in B \E \mathcal R \E a \R \to \mathcal R \E b \R \R = \mathcal V \E 1_{\mathcal V} \to \mathcal B \E \mathcal R \E a \R \to \mathcal R \E b \R \R \R$.  As with the underlying category, we will often denote $\mathcal R^{\mathcal V}$ by simply $R$.
\end{defn}

\subsection{\texorpdfstring{$\mathcal V$-Monoidal Categories}{}}
Given a $\mathcal V$-category, we can add a monoidal structure to get a (strict) $\mathcal V$-monoidal category.  We need the additional data of
\begin{itemize}
\item A unit object $1_{\mathcal A}$,
\item For each pair $a , b \in \mathcal A$, a tensor product object $ab \in \mathcal A$, and
\item For each $a , b , c , d \in \mathcal A$, a tensor product morphism $\blank \otimes_{\mathcal A} \blank \in \mathcal V \E \mathcal A \E a \to c \R \mathcal A \E b \to d \R \to \mathcal A \E a b \to c d \R \R$
\end{itemize}

such that the following axioms hold:
\begin{itemize}
\item Tensor unit: For all $a \in \mathcal A$, $1_{\mathcal A} a = a = a 1_{\mathcal A}$,
\item Associative on objects: For all $a , b , c \in \mathcal A$, $\E a b \R c = a \E b c \R$,
\item Unitality: For all $a , b \in \mathcal A$,
  \begin{align*}
    \begin{tikzpicture}[baseline=30, smallstring]
      \node (top) at (1,3) {$\mathcal A \E a \to b \R$};
      \node (bot) at (2,0) {$\mathcal A \E a \to b \R$};
      \node[draw, rectangle] (j) at (0,1) {$j_{1_{\mathcal A}}$};
      \node[draw, rectangle] (circ) at (1,2) {$\blank \otimes_{\mathcal A} \blank$};
      \draw (j) to[in=-90,out=90] (circ.225);
      \draw (bot) to[in=-90,out=90] (circ.-45);
      \draw (circ) to[in=-90,out=90] (top);
    \end{tikzpicture}
    =
    \begin{tikzpicture}[baseline=30, smallstring]
      \node (bot) at (0,0) {$\mathcal A \E a \to b \R$};
      \node (top) at (0,3) {$\mathcal A \E a \to b \R$};
      \draw (top) to (bot);
    \end{tikzpicture}
    =
    \begin{tikzpicture}[baseline=30, smallstring]
      \node (top) at (1,3) {$\mathcal A \E a \to b \R$};
      \node (bot) at (0,0) {$\mathcal A \E a \to b \R$};
      \node[draw, rectangle] (j) at (2,1) {$j_{1_{\mathcal A}}$};
      \node[draw, rectangle] (circ) at (1,2) {$\blank \otimes_{\mathcal A} \blank$};
      \draw (j) to[in=-90,out=90] (circ.-45);
      \draw (bot) to[in=-90,out=90] (circ.225);
      \draw (circ) to[in=-90,out=90] (top);
    \end{tikzpicture}
  \end{align*}
\item Associativity of $\E \blank \otimes \blank \R$: For all $a , b , c , d , e , f \in \mathcal A$,
  \begin{align*}
    \begin{tikzpicture}[baseline=30, smallstring]
      \node (ad) at (0,0) {$\mathcal A \E a \to d \R$};
      \node (be) at (2,0) {$\mathcal A \E b \to e \R$};
      \node (cf) at (4,0) {$\mathcal A \E c \to f \R$};
      \node (top) at (2,3) {$\mathcal A \E a b c \to d e f \R$};
      \node[draw, rectangle] (circ1) at (1,1) {$\blank \otimes_{\mathcal A} \blank$};
      \node[draw, rectangle] (circ2) at (2,2) {$\blank \otimes_{\mathcal A} \blank$};
      \draw (ad) to[in=-90,out=90] (circ1.225);
      \draw (be) to[in=-90,out=90] (circ1.-45);
      \draw (circ1) to[in=-90,out=90] (circ2.225);
      \draw (cf) to[in=-90,out=90] (circ2.-45);
      \draw (circ2) to[in=-90,out=90] (top);
    \end{tikzpicture}
    =
    \begin{tikzpicture}[baseline=30, smallstring]
      \node (ad) at (0,0) {$\mathcal A \E a \to d \R$};
      \node (be) at (2,0) {$\mathcal A \E b \to e \R$};
      \node (cf) at (4,0) {$\mathcal A \E c \to f \R$};
      \node (top) at (2,3) {$\mathcal A \E a b c \to d e f \R$};
      \node[draw, rectangle] (circ1) at (3,1) {$\blank \otimes_{\mathcal A} \blank$};
      \node[draw, rectangle] (circ2) at (2,2) {$\blank \otimes_{\mathcal A} \blank$};
      \draw (ad) to[in=-90,out=90] (circ2.225);
      \draw (be) to[in=-90,out=90] (circ1.225);
      \draw (circ1) to[in=-90,out=90] (circ2.-45);
      \draw (cf) to[in=-90,out=90] (circ1.-45);
      \draw (circ2) to[in=-90,out=90] (top);
    \end{tikzpicture}
  \end{align*}
\item Braided interchange: For all $a , b , c , d , e , f \in \mathcal A$,
  \begin{equation*}
    \label{eq:BraidedInterchange}
    \begin{tikzpicture}[baseline=50, smallstring]
      \node (a-b) at (0,0) {$\mathcal A(a \to b)$};
      \node (d-e) at (2,0) {$\mathcal A(d \to e)$};
      \node (b-c) at (4,0) {$\mathcal A(b \to c)$};
      \node (e-f) at (6,0) {$\mathcal A(e \to f)$};
      \node[draw,rectangle] (t1) at (1,2) {$\quad -\otimes-\quad $};
      \draw (a-b) to[in=-90,out=90] (t1.-135);
      \draw (d-e) to[in=-90,out=90] (t1.-45);
      \node[draw,rectangle] (t2) at (5,2) {$\quad -\otimes-\quad $};
      \draw (b-c) to[in=-90,out=90] (t2.-135);
      \draw (e-f) to[in=-90,out=90] (t2.-45);
      \node[draw,rectangle] (c) at (3,4) {$\quad - \circ - \quad $};
      \draw (t1) to[in=-90,out=90] node[left=13pt] {$\mathcal A(ad\to be)$} (c.-135);
      \draw (t2) to[in=-90,out=90] node[right=13pt] {$\mathcal A(be\to cf)$} (c.-45);
      \node (r) at (3,5.5) {$\mathcal A(ad \to cf)$};
      \draw (c) -- (r);
    \end{tikzpicture}
    =
    \begin{tikzpicture}[baseline=50, smallstring]
      \node (a-b) at (0,0) {$\mathcal A(a \to b)$};
      \node (d-e) at (2,0) {$\mathcal A(d \to e)$};
      \node (b-c) at (4,0) {$\mathcal A(b \to c)$};
      \node (e-f) at (6,0) {$\mathcal A(e \to f)$};
      \node[draw,rectangle] (t1) at (1,2) {$\quad -\circ-\quad $};
      \node[draw,rectangle] (t2) at (5,2) {$\quad -\circ-\quad $};
      \draw (a-b) to[in=-90,out=90] (t1.-135);
      \draw[knot] (b-c) to[in=-90,out=90] (t1.-45);
      \draw[knot] (d-e) to[in=-90,out=90] (t2.-135);
      \draw (e-f) to[in=-90,out=90] (t2.-45);
      \node[draw,rectangle] (c) at (3,4) {$\quad - \otimes - \quad $};
      \draw (t1) to[in=-90,out=90] node[left=13pt] {$\mathcal A(a\to c)$} (c.-135);
      \draw (t2) to[in=-90,out=90] node[right=13pt] {$\mathcal A(d\to f)$} (c.-45);
      \node (r) at (3,5.5) {$\mathcal A(ad \to cf)$};
      \draw (c) -- (r);
    \end{tikzpicture}
  \end{equation*}
\end{itemize}

\begin{defn}
  Given a $\mathcal V$-monoidal category $\mathcal A$, take the underlying category $A = \mathcal A^{\mathcal V}$.  Then $A$ inherits the monoidal structure of $\mathcal A$ via the same tensor product on objects, and for $f \in A \E a \to c \R$ and $g \in A \E b \to d \R$, $f g \coloneqq \E f g \R \circ \E \blank \otimes \blank \R$, where on the right hand side $f$ and $g$ are viewed as morphisms in $\mathcal V$.  This is the \emph{underlying monoidal category} of $\mathcal A$.
\end{defn}

\example[Monoidal self-enrichment]
As before, when $\mathcal V$ is closed we can form the $\mathcal V$-category $\widehat{\mathcal V}$. We now additionally define, for $u , v , w , x \in \mathcal V$,
\begin{align*}
  \blank \otimes_{\widehat{\mathcal V}} \blank \coloneqq
  \mate \E \begin{tikzpicture}[smallstring, baseline=30]
    \node (u) at (0,0) {$u$};
    \node (w) at (2,0) {$w$};
    \node (uv) at (4,0) {$\widehat{\mathcal V} \E u \to v \R$};
    \node (wx) at (6,0) {$\widehat{\mathcal V} \E w \to x \R$};
    \node (v) at (1.5,3) {$v$};
    \node (x) at (4.5,3) {$x$};
    \node[draw, rectangle] (ep1) at (1.5,2) {$\epsilon_{u \to v}^{\mathcal F}$};
    \node[draw, rectangle] (ep2) at (4.5,2) {$\epsilon_{w \to x}^{\mathcal F}$};
    \draw (u) to[in=-90,out=90] (ep1.225);
    \draw (uv) to[in=-90,out=90] (ep1.-45);
    \draw[knot] (w) to[in=-90,out=90] (ep2.225);
    \draw (wx) to[in=-90,out=90] (ep2.-45);
    \draw (ep1) to[in=-90,out=90] (v);
    \draw (ep2) to[in=-90,out=90] (x);
  \end{tikzpicture} \R
\end{align*}
to give $\widehat{\mathcal V}$ the structure of a $\mathcal V$-monoidal category.

\subsection{$\mathcal V$-monoidal functors and natural transformations}
We can similarly extend $\mathcal V$-functors to a monoidal setting: a strictly unital $\mathcal V$-monoidal functor $\E \mathcal R , \rho^R \R: \mathcal A \to \mathcal B$ consists of an underlying $\mathcal V$-functor $\mathcal R : \mathcal A \to \mathcal B$ with $\mathcal R \E 1_{\mathcal A} \R = 1_{\mathcal B}$ and $j_{1_{\mathcal A}}^{\mathcal A} \circ \mathcal R_{1_{\mathcal A} \to 1_{\mathcal A}} = j_{1_{\mathcal B}}^{\mathcal B}$, and a family of $1_{\mathcal V}$-graded isomorphisms $\rho_{a , b}^R \in \mathcal V \E 1_{\mathcal V} \to \mathcal B \E \mathcal R \E a \R \mathcal R \E b \R \to \mathcal R \E a b \R \R \R$ satisfying:
\begin{itemize}
\item Naturality:
  \begin{align*}
    \begin{tikzpicture}[baseline=40, smallstring]
      \node (ac) at (2,0) {$\mathcal A \E a \to c \R$};
      \node (bd) at (4,0) {$\mathcal A \E b \to d \R$};
      \node (top) at (2,4) {$\mathcal B \E \mathcal R \E a \R \mathcal R \E b \R \to \mathcal R \E c d \R \R$};
      \node[draw, rectangle] (otimes) at (3,1) {$\blank \otimes_{\mathcal A} \blank$};
      \node[draw, rectangle] (r) at (3,2) {$\mathcal R_{a b \to c d}$};
      \node[draw, rectangle] (rho) at (1,2) {$\rho_{a , b}^R$};
      \node[draw, rectangle] (circ) at (2,3) {$\blank \circ_{\mathcal B} \blank$};
      \draw (ac) to[in=-90,out=90] (otimes.225);
      \draw (bd) to[in=-90,out=90] (otimes.-45);
      \draw (otimes) to[in=-90,out=90] (r);
      \draw (r) to[in=-90,out=90] (circ.-45);
      \draw (rho) to[in=-90,out=90] (circ.225);
      \draw (circ) to[in=-90,out=90] (top);
    \end{tikzpicture}
    =
    \begin{tikzpicture}[baseline=40, smallstring]
      \node (ac) at (0,0) {$\mathcal A \E a \to c \R$};
      \node (bd) at (2,0) {$\mathcal A \E b \to d \R$};
      \node (top) at (2,4) {$\mathcal B \E \mathcal R \E a \R \mathcal R \E b \R \to \mathcal R \E c d \R \R$};
      \node[draw, rectangle] (rac) at (0,1) {$\mathcal R_{a \to c}$};
      \node[draw, rectangle] (rbd) at (2,1) {$\mathcal R_{b \to d}$};
      \node[draw, rectangle] (otimes) at (1,2) {$\blank \otimes_{\mathcal B} \blank$};
      \node[draw, rectangle] (rho) at (3,2) {$\rho_{c , d}^R$};
      \node[draw, rectangle] (circ) at (2,3) {$\blank \circ_{\mathcal B} \blank$};
      \draw (ac) to[in=-90,out=90] (rac);
      \draw (bd) to[in=-90,out=90] (rbd);
      \draw (rac) to[in=-90,out=90] (otimes.225);
      \draw (rbd) to[in=-90,out=90] (otimes.-45);
      \draw (otimes) to[in=-90,out=90] (circ.225);
      \draw (rho) to[in=-90,out=90] (circ.-45);
      \draw (circ) to[in=-90,out=90] (top);
    \end{tikzpicture}
  \end{align*}
\item Associativity:
  \begin{align*}
    \begin{tikzpicture}[baseline=40, smallstring]
      \node (top) at (2,3) {$\mathcal B \E \mathcal R \E a \R \mathcal R \E b \R \mathcal R \E c \R \to \mathcal R \E a b c \R \R$};
      \node[draw, rectangle] (j) at (0,0) {$j_{\mathcal R \E a \R}^{\mathcal B}$};
      \node[draw, rectangle] (rho1) at (2,0) {$\rho_{b , c}^R$};
      \node[draw, rectangle] (otimes) at (1,1) {$\blank \otimes_{\mathcal B} \blank$};
      \node[draw, rectangle] (rho2) at (3,1) {$\rho_{a , b c}^R$};
      \node[draw, rectangle] (circ) at (2,2) {$\blank \circ_{\mathcal B} \blank$};
      \draw (j) to[in=-90,out=90] (otimes.225);
      \draw (rho1) to[in=-90,out=90] (otimes.-45);
      \draw (otimes) to[in=-90,out=90] (circ.225);
      \draw (rho2) to[in=-90,out=90] (circ.-45);
      \draw (circ) to[in=-90,out=90] (top);
    \end{tikzpicture}
    =
    \begin{tikzpicture}[baseline=40, smallstring]
      \node (top) at (2,3) {$\mathcal B \E \mathcal R \E a \R \mathcal R \E b \R \mathcal R \E c \R \to \mathcal R \E a b c \R \R$};
      \node[draw, rectangle] (j) at (2,0) {$j_{\mathcal R \E c \R}^{\mathcal B}$};
      \node[draw, rectangle] (rho1) at (0,0) {$\rho_{a , b}^R$};
      \node[draw, rectangle] (otimes) at (1,1) {$\blank \otimes_{\mathcal B} \blank$};
      \node[draw, rectangle] (rho2) at (3,1) {$\rho_{a b , c}^R$};
      \node[draw, rectangle] (circ) at (2,2) {$\blank \circ_{\mathcal B} \blank$};
      \draw (j) to[in=-90,out=90] (otimes.-45);
      \draw (rho1) to[in=-90,out=90] (otimes.225);
      \draw (otimes) to[in=-90,out=90] (circ.225);
      \draw (rho2) to[in=-90,out=90] (circ.-45);
      \draw (circ) to[in=-90,out=90] (top);
    \end{tikzpicture}
  \end{align*}
\end{itemize}
We call $\mathcal R$ \emph{lax} if the $\rho_{a , b}^R$ are not necessarily isomorphisms, and \emph{oplax} if we have $\overline \rho_{a , b}^R : 1_{\mathcal V} \to \mathcal B \E \mathcal R \E a b \R \to \mathcal R \E a \R \mathcal R \E b \R \R$ instead, again not necessarily isomorphisms, along with the corresponding axioms in place of the above.

We can compose $\E \mathcal R , \rho^R \R : \mathcal A \to \mathcal B$ with $\E \mathscr S , \sigma^S \R : \mathcal B\to \mathscr C$ to get a $\mathcal V$-monoidal functor via the usual composition on $\mathcal V$-functors and laxitor
\begin{align*}
  \E \rho_{a , b}^R \sigma_{\mathcal R \E a \R , \mathcal R \E b \R}^S \R \circ \E \mathscr S_{\mathcal R \E a \R \mathcal R \E b \R \to \mathcal R \E a b \R} \id \R \circ \E \blank \circ_{\mathscr C} \blank \R
  =
  \begin{tikzpicture}[baseline=30, smallstring]
    \node (top) at (1,3) {$\mathscr C \E \mathcal T \E a \R \mathcal T \E b \R \to \mathcal T \E a b \R \R$};
    \node[draw, rectangle] (rho) at (2,0) {$\rho_{a , b}^R$};
    \node[draw, rectangle] (sigma) at (0,0) {$\sigma_{\mathcal R \E a \R , \mathcal R \E b \R}^S$};
    \node[draw, rectangle] (s) at (2,1) {$\mathscr S_{\mathcal R \E a \R \mathcal R \E b \R \to \mathcal R \E a b \R}$};
    \node[draw, rectangle] (circ) at (1,2) {$\blank \circ_{\mathscr C} \blank$};
    \draw (rho) to[in=-90,out=90] (s);
    \draw (sigma) to[in=-90,out=90] (circ.225);
    \draw (s) to[in=-90,out=90] (circ.-45);
    \draw (circ) to[in=-90,out=90] (top);
  \end{tikzpicture}.
\end{align*}
A $1_{\mathcal V}$-graded monoidal natural transformation $\theta : \mathcal R \to \mathscr S$ is a $1_{\mathcal V}$-graded natural transformation such that
\begin{equation}\label{vtrans mon}
  \begin{tikzpicture}[baseline=15, smallstring]
    \node (top) at (1,2) {$\mathcal B \E \mathcal R \E a \R \mathcal R \E b \R \to \mathscr S \E a b \R \R$};
    \node[draw, rectangle] (rho) at (0,0) {$\rho_{a , b}^R$};
    \node[draw, rectangle] (theta) at (2,0) {$\theta_{a b}$};
    \node[draw, rectangle] (circ) at (1,1) {$\blank \circ_{\mathcal B} \blank$};
    \draw (rho) to[in=-90,out=90] (circ.225);
    \draw (theta) to[in=-90,out=90] (circ.-45);
    \draw (circ) to[in=-90,out=90] (top);
  \end{tikzpicture}
  =
  \begin{tikzpicture}[baseline=30, smallstring]
    \node(top) at (2,3) {$\mathcal B \E \mathcal R \E a \R \mathcal R \E b \R \to \mathscr S \E a b \R \R$};
    \node[draw, rectangle] (eta) at (0,0) {$\theta_a$};
    \node[draw, rectangle] (etb) at (2,0) {$\theta_b$};
    \node[draw, rectangle] (tens) at (1,1) {$\blank \otimes_{\mathcal B} \blank$};
    \node[draw, rectangle] (sigma) at (3,1) {$\sigma_{a , b}^S$};
    \node[draw, rectangle] (circ) at (2,2) {$\blank \circ_{\mathcal B} \blank$};
    \draw (eta) to[in=-90,out=90] (tens.225);
    \draw (etb) to[in=-90,out=90] (tens.-45);
    \draw (tens) to[in=-90,out=90] (circ.225);
    \draw (sigma) to[in=-90,out=90] (circ.-45);
    \draw (circ) to[in=-90,out=90] (top);
  \end{tikzpicture}.
\end{equation}

\begin{defn}
  Given a (lax) $\mathcal V$-monoidal functor $\E \mathcal R , \rho^{\mathcal R} \R$ between $\mathcal V$-monoidal categories $\mathcal A$ and $\mathcal B$, there is a canonical ordinary (lax) monoidal functor $\E R , \rho \R$ between the underlying categories $A$ and $B$.  We take $R$ to be the underlying functor, and we take $\rho \coloneqq \rho^{\mathcal R}$; notice that for $a , b \in \mathcal A$,
  \begin{align*}
    \rho_{a , b}^{\mathcal R} \in \mathcal V \E 1_{\mathcal V} \to \mathcal B \E \mathcal R \E a \R \mathcal R \E b \R \to \mathcal R \E a b \R \R \R = B \E R \E a \R R \E b \R \to R \E a b \R \R.
  \end{align*}
  We call $\E R , \rho \R$ the \emph{underlying monoidal functor} of $\E \mathcal R , \rho^{\mathcal R} \R$; verifying the axioms of a (lax) monoidal functor is a useful exercise.
\end{defn}

\subsection{The categorified trace}
With a minor change in the notation of \cite{MR3961709}, given a $\mathcal V$-monoidal category $\mathcal A$, we define a functor $\Tr_{\mathcal A} : \mathcal A^{\mathcal V} \to \mathcal V$, where $\mathcal A^{\mathcal V}$ is the underlying category of $\mathcal A$.  On objects, we set $\Tr_{\mathcal A} \E a \R = \mathcal A \E 1_{\mathcal A} \to a \R$, and on morphisms $f \in \mathcal A^{\mathcal V} \E a \to b \R = \mathcal V \E 1_{\mathcal V} \to \mathcal A \E a \to b \R \R$ we define
\begin{align*}
  \Tr_{\mathcal A} \E f \R%
  =
  \begin{tikzpicture}[baseline=30, smallstring]
    \node (bot) at (0,0) {$\mathcal A \E 1_{\mathcal A} \to a \R$};
    \node (top) at (1,3) {$\mathcal A \E 1_{\mathcal A} \to b \R$};
    \node[draw, rectangle] (f) at (2,1) {$f$};
    \node[draw, rectangle] (circ) at (1,2) {$\blank \circ_{\mathcal A} \blank$};
    \draw (bot) to[in=-90,out=90] (circ.225);
    \draw (f) to[in=-90,out=90] (circ.-45);
    \draw (circ) to[in=-90,out=90] (top);
  \end{tikzpicture}.
\end{align*}
In order for the name ``trace'' to make sense, we would expect an isomorphism $\Tr_{\mathcal A} \E a b \R \to \Tr_{\mathcal A} \E b a \R$.  When $\mathcal A$ is rigid, we have
\begin{align*}
  \Tr_{\mathcal A} \E a b \R%
  &= \mathcal A \E 1_{\mathcal A} \to a b \R\\
  &\cong \mathcal A \E b^* \to a \R\\
  &\cong \mathcal A \E 1_{\mathcal A} \to b^{**} a \R,
\end{align*}
so we get a natural isomorphism $\Tr_{\mathcal A} \E a b \R \cong \Tr_{\mathcal A} \E b^{**} a \R$.  In the case that there is a $1_{\mathcal V}$-graded natural isomorphism $1_{\mathcal A} \cong **$ ($\mathcal A$ is $1_{\mathcal V}$-graded pivotal), this is then also isomorphic to $\mathcal A \E 1_{\mathcal A} \to b a \R = \Tr_{\mathcal A} \E b a \R$.  With this in mind, we continue the convention in \cite{MR3961709} and \cite{MR3578212} to call the map a trace.

\subsection{2-categories, 2-functors, and 2-equivalences}

\begin{defn}
  A \emph{2-category} $\mathscr A$ is a category enriched in $\Cat$. Explicitly, $\mathscr A$ consists of the data
\begin{itemize}
\item A collection of objects.  In the 2-categorical context, we will refer to objects as \emph{0-cells} to avoid ambiguity.
\item For each $A , B \in \mathscr A$, a \emph{hom-category} $\mathscr A \E A \to B \R$.  We avoid the term ``morphism'' here to avoid confusion, and refer to objects of such a hom-category simply as 1-cells.  The morphisms of a hom-category are called \emph{2-cells}.  We write $R : A \to B$ for 1-cells, and $\theta : R \Rightarrow S$ for 2-cells.  Vertical composition of 2-cells is suppressed in our notation.
\item For each $A \in \mathscr A$, a \emph{unit 1-cell} $\id_A \in \mathscr A \E A \to A \R$.
\item For each $A , B , C \in \mathscr A$, a composition bifunctor $\comph : \mathscr A \E A \to B \R \times \mathscr A \E B \to C \R$.  On 1-cells $\comph$ is the expected composition, and it is horizontal composition on 2-cells.
\end{itemize}
satisfying the following axioms:
\begin{itemize}
\item Unitality: For all $A , B \in \mathscr A$ and each $f \in \mathscr A \E A \to B \R$, $\id_A \circ f = f \circ \id_B = f$.
\item Associativity: For each $A , B , C , D \in \mathscr A$,
  \begin{equation*}
    \begin{tikzcd}
      \mathscr A \E A \to B \R \times \mathscr A \E B \to C \R \times \mathscr A \E C \to D \R \dar["\comph \times \id"] \rar ["\id \times \comph"] & \mathscr A \E A \to B \R \times \mathscr A \E B \to D \R \dar["\comph"]\\
      \mathscr A \E A \to C \R \times \mathscr A \E C \to D \R \rar["\comph"] & \mathscr A \E A \to D \R
    \end{tikzcd}
  \end{equation*}
  commutes.
\end{itemize}
\end{defn}
\begin{defn} \label{2-functors}
A \emph{2-functor} $P : \mathscr A \to \mathscr B$ between 2-categories $\mathscr A$ and $\mathscr B$ is a $\Cat$-functor between $\Cat$-categories. Explicitly, $P$ consists of:
\begin{itemize}
\item For each 0-cell $A \in \mathscr A$, a 0-cell $P \E A \R$ in $\mathscr B$.
\item For each pair of 0-cells $A , B$ in $\mathscr A$, a functor
  \begin{align*}
    P_{A \to B} : \mathscr A \E A \to B \R \to \mathscr B \E P \E A \R \to P \E B \R \R
  \end{align*}
\end{itemize}
such that
\begin{itemize}
\item For each 0-cell $\mathcal A$ in $\mathscr A$, $P_{\mathcal A \to \mathcal A} \E \id_{\mathcal A} \R = \id_{P \E \mathcal A \R}$.
\item the composition functor is preserved, i.e., given 0-cells $\mathcal A , \mathcal B , \mathcal C$ in $\mathscr A$:
  \begin{itemize}
  \item For 1-cells $\mathcal R : \mathcal A \to \mathcal B$ and $\mathcal S : \mathcal B \to \mathcal C$,
    \begin{align*}
      P_{\mathcal A \to \mathcal B} \E \mathcal R \R \circ P_{\mathcal B \to \mathcal C} \E \mathcal S\R = P_{\mathcal A \to \mathcal C} \E \mathcal R \circ \mathcal S \R.
    \end{align*}
  \item For 1-cells $\mathcal R_1 , \mathcal R_2 : \mathcal A \to \mathcal B$ and $\mathcal S_1 , \mathcal S_2 : B \to \mathcal C$, and 2-cells $\theta : \mathcal R_1 \Rightarrow \mathcal R_2$ and $\psi : \mathcal S_1 \Rightarrow \mathcal S_2$,
    \begin{align*}
      P_{\mathcal A \to \mathcal C} \E \theta \psi \R = P_{\mathcal A \to \mathcal B} \E \theta \R P_{\mathcal B \to \mathcal C} \E \psi \R.
    \end{align*}
  \end{itemize}
\end{itemize}
\end{defn}

\begin{defn}[\cite{2002.06055}]
  We say that a 2-functor $P: \mathscr A \to \mathscr B$ is a \emph{2-equivalence} if it is:
\begin{itemize}
\item Essentially surjective: For every 0-cell $B \in \mathscr B$, there is a 0-cell $A \in \mathscr A$ and an equivalence of categories $A \simeq P \E A \R$.
\item Fully faithful: For each pair of 0-cells $A , B \in \mathscr A$, $P_{A \to B}$ is an equivalence of categories
\end{itemize}
\end{defn}


\section{\texorpdfstring{The 2-categories $\vmoncat$ and $\vmodtens$}{}} \label{Define bicategories}
In this section we will construct the 2-categories $\vmoncat$ and $\vmodtens$, based on the 0-cells studied in \cite{MR3961709}.  First, however, we record some terminology from \cite{1809.09782, 1910.03178}.
\begin{defn} [\cite{MR3961709, 1809.09782, 1910.03178}]
  We call a $\mathcal V$-category $\mathcal A$
  \begin{itemize}
  \item \emph{weakly tensored} if for each $a \in \mathcal A$, the representable functor $\mathcal A \E a \to \blank \R : A \to \mathcal V$ admits a left adjoint.
  \item \emph{tensored} if for each $a \in \mathcal A$, the $\mathcal V$-representable functor $\mathcal A \E a \to \blank \R : \mathcal A \to \widehat{\mathcal V}$ admits a left $\mathcal V$-adjoint.
  \end{itemize}
  Note that we must be able to form the self-enrichment $\widehat{\mathcal V}$ for the definition of a tensored $\mathcal V$-category to make sense; this requires $\mathcal V$ to be \emph{closed} (see \cite{1809.09782}).  In this article we assume $\mathcal V$ is rigid, and being rigid is a stronger property of a $\mathcal V$-category than being closed.
\end{defn}
\begin{defn} [\cite{1910.03178} defn 2.6, HPT]
  A \emph{$\mathcal V$-module tensor category} is a pair $( A , \mathcal F_A^Z )$ where $A$ is a monoidal category and $\mathcal F_A^Z : \mathcal V \to Z \E A \R$ is an oplax, strongly unital braided monoidal functor. We call the $\mathcal V$-module tensor category \emph{oplax} or \emph{strongly unital} if $\mathcal F_A^Z$ is oplax or strongly unital, respectively.\\
  We say that a $\mathcal V$-module tensor category $( A , \mathcal F_A^Z )$ is
  \begin{itemize}
  \item \emph{rigid} if $A$ is rigid,
  \item \emph{weakly tensored} if $\mathcal F \coloneqq \mathcal F^Z \circ \Forget_Z$ admits a right adjoint, with $\Forget_Z : Z \E A \R \to A$ the forgetful functor,
  \item \emph{tensored} if it is weakly tensored and $\mathcal F$ is strong monoidal.
  \end{itemize}
\end{defn}
\begin{defn} \label{Define vmoncat}
  The following data defines a 2-category, which we call $\vmoncat$:
  \begin{itemize}
  \item 0-Cells are rigid $\mathcal V$-monoidal categories ($\mathcal A$, $\mathcal B$, etc.) such that the functor $x \mapsto \mathcal A \E 1_{\mathcal A} \to x \R$ admits a left adjoint.
  \item For each pair of 0-cells $\mathcal A , \mathcal B$, the hom-category $\vmoncat \E \mathcal A \to \mathcal B \R$ consists of
    \begin{itemize}
    \item 1-cells: Strictly unital lax $\mathcal V$-monoidal functors $\E \mathcal R , \rho^R \R : \mathcal A \to \mathcal B$, where $\rho^R$ is the laxitor for $\mathcal R$.
    \item 2-cells: $1_{\mathcal V}$-graded monoidal natural transformations $\theta : \E \mathcal R , \rho^R \R \to \E \mathcal S , \sigma^S \R$.
    \item Composition is the usual vertical composition of $1_{\mathcal V}$-graded natural transformations:
      \begin{align*}
        \E \theta \circ \varphi \R_a =
        \begin{tikzpicture}[baseline=15, smallstring]
          \node (top) at (1,2) {$\mathcal B \E \mathcal R \E a \R \to \mathcal S \E a \R \R$};
          \node[draw, rectangle] (theta) at (0,0) {$\theta_a$};
          \node[draw, rectangle] (ed) at (2,0) {$\varphi_a$};
          \node[draw, rectangle] (circ) at (1,1) {$\blank \circ_{\mathcal B} \blank$};
          \draw (theta) to[in=-90,out=90] (circ.225);
          \draw (ed) to[in=-90,out=90] (circ.-45);
          \draw (circ) to[in=-90,out=90] (top);
        \end{tikzpicture}
      \end{align*}
    \item Unit 2-cells are the identity $1_{\mathcal V}$-graded natural transformations.
    \end{itemize}
  \item Composition functors:
    \begin{itemize}
    \item Two 1-cells $\E \mathcal R , \rho^R \R : \mathcal A \to \mathcal B$ and $\E \mathcal S , \sigma^S \R : \mathcal B \to \mathcal C$ have composition defined via $\E \mathcal R \circ \mathcal S \R \E a \R \coloneqq \mathcal S \E \mathcal R \E a \R \R$, $\E \mathcal R \circ \mathcal S \R_{a \to b} \coloneqq \mathcal R_{a \to b} \circ \mathcal S_{R \E a \R \to \mathcal R \E b \R}$, and composite laxitor
      \begin{align*}
        \begin{tikzpicture}[baseline=15, smallstring]
          \node (top) at (1,2) {$\mathcal C \E R \E a \R \to S \E a \R \R$};
          \node[draw, rectangle] (sigma) at (0,0) {$\sigma_{R \E a \R , R \E b \R}^S$};
          \node[draw, rectangle] (rho) at (2,0) {$\mathcal S \E \rho_{a , b}^R \R$};
          \node[draw, rectangle] (circ) at (1,1) {$\blank \circ_{\mathcal C} \blank$};
          \draw (sigma) to[in=-90,out=90] (circ.225);
          \draw (rho) to[in=-90,out=90] (circ.-45);
          \draw (circ) to[in=-90,out=90] (top);
        \end{tikzpicture}
      \end{align*}
    \item If $\mathcal A , \mathcal B , \mathcal C \in \vmoncat$, $\mathcal R_1 , \mathcal R_2 \in \vmoncat \E \mathcal A \to \mathcal B \R$, and $\mathcal S_1 , \mathcal S_2 \in \vmoncat \E \mathcal B \to \mathcal C \R$, the horizontal composition of $\theta : \mathcal R_1 \Rightarrow \mathcal R_2$ and $\varphi : \mathcal S_1 \Rightarrow \mathcal S_2$ is given by
      \begin{align*}
        \E \theta \varphi \R_a = 
        \begin{tikzpicture}[baseline=25, smallstring]
          \node (top) at (1,3) {$\mathcal C \E S_1 \E R_1 \E a \R \R \to S_2 \E R_2 \E a \R \R \R$};
          \node[draw, rectangle] (theta) at (0,0) {$\theta_a$};
          \node[draw, rectangle] (ed) at (2,0) {$\varphi_{R_1 \E a \R}$};
          \node[draw, rectangle] (S) at (0,1) {$\E \mathcal S_1 \R_{\mathcal R_1 \E a \R \to \mathcal R_2 \E a \R}$};
          \node[draw, rectangle] (circ) at (1,2) {$\blank \circ_{\mathcal C} \blank$};
          \draw (ed) to[in=-90,out=90] (circ.-45);
          \draw (theta) to[in=-90,out=90] (S);
          \draw (S) to[in=-90,out=90] (circ.225);
          \draw (circ) to[in=-90,out=90] (top);
        \end{tikzpicture}
      \end{align*}
    \end{itemize}
  \item Lastly, unit 1-cells are the usual identity $\mathcal V$-functor $\id_{\mathcal A}: \mathcal A \to \mathcal A$ along with the identity natural transformations as the tensorator.
  \end{itemize}
\end{defn}

\begin{remark}
  The definition above of $\vmoncat$ does in fact define a 2-category.
  \begin{itemize}
  \item First, $\vmoncat \E \mathcal A \to \mathcal B \R$ is a category:
    \begin{itemize}
    \item Unit 2-cells are the usual identity natural transformations, and thus composition preserves units.
    \item Vertical composition of 2-cells is associative by Lemma 2.2 of \cite{MR3961709}.
    \end{itemize}
  \item Composition functors are unital because unit 1-cells are identity $\mathcal V$-monoidal functors, with their usual composition.
  \item Composition functors are associative: Horizontal composition of 2-cells is associative by Lemma 2.3 of \cite{MR3961709}, and composition of 1-cells is just ordinary $\mathcal V$-monoidal functor composition, which is associative.
  \end{itemize}
\end{remark}


\begin{defn} \label{Define vmodtens}
  We define a second 2-category, which we call $\vmodtens$, from the following data:
  \begin{itemize}
  \item 0-Cells are weakly tensored $\mathcal V$-module tensor categories $( A , \mathcal F_A^Z )$.
  \item For each pair of 0-cells $( A , \mathcal F_A^Z ) , ( B , \mathcal F_B^Z )$, we define the hom-category $\vmodtens \E ( A , \mathcal F_A^Z ) \to ( B , \mathcal F_B^Z ) \R$ as follows:
    \begin{itemize}
    \item 1-cells are triples $\E R , \rho , r \R : ( A , \mathcal F_A^Z ) \to ( B , \mathcal F_B^Z )$, where $\E R , \rho \R$ is a strictly unital lax monoidal functor, and $r: \mathcal F_B \Rightarrow \mathcal F_A \circ R$ is a strictly unital monoidal natural transformation, i.e., $R \E 1_A \R = 1_B$ and $r_{1_{\mathcal V}} = \id_{1_B}$; they must also satisfy the half-braiding coherence and action coherence conditions:
      \begin{align*}
        \begin{tikzpicture}[baseline=40, smallstring]
          \node (top) at (1,4) {$R \E \mathcal F_A \E v \R a \R$};
          \node (ra) at (0,0) {$R \E a \R$};
          \node (v) at (2,0) {$\mathcal F_B \E v \R$};
          \node[draw, rectangle] (r) at (0,2) {$r_v$};
          \node[draw, rectangle] (rho) at (1,3) {$\rho_{\mathcal F_A \E v \R , a}$};
          \node[draw, rectangle] (e) at (1,1) {$e_{R \E a \R , \mathcal F_B \E v \R}$};
          \draw (v) to[in=-90,out=90] (e.-45);
          \draw (ra) to[in=-90,out=90] (e.225);
          \draw (e.135) to[in=-90,out=90] (r);
          \draw (e.45) to[in=-90,out=90] (rho.-45);
          \draw (r) to[in=-90,out=90] (rho.225);
          \draw (rho) to[in=-90,out=90] (top);
        \end{tikzpicture}
        =
        \begin{tikzpicture}[baseline=40, smallstring]
          \node (top) at (1,4) {$R \E \mathcal F_A \E v \R a \R$};
          \node (ra) at (0,0) {$R \E a \R$};
          \node (v) at (2,0) {$\mathcal F_B \E v \R$};
          \node[draw, rectangle] (r) at (2,1) {$r_v$};
          \node[draw, rectangle] (rho) at (1,2) {$\rho_{a , \mathcal F_A \E v \R}$};
          \node[draw, rectangle] (e) at (1,3) {$R \E e_{a , \mathcal F_A \E v \R} \R$};
          \draw (v) to[in=-90,out=90] (r);
          \draw (ra) to[in=-90,out=90] (rho.225);
          \draw (r) to[in=-90,out=90] (rho.-45);
          \draw (rho) to[in=-90,out=90] (e);
          \draw (e) to[in=-90,out=90] (top);
        \end{tikzpicture}
        \hspace{3em}
        \tand
        \hspace{3em}
        \begin{tikzpicture}[baseline=40, smallstring]
          \node (top) at (0,3) {$R \E \mathcal F_A \E u \R \mathcal F_A \E v \R \R$};
          \node (bot) at (0,0) {$\mathcal F_B \E u v \R$};
          \node[draw, rectangle] (r) at (0,1) {$r_{uv}$};
          \node[draw, rectangle] (mu) at (0,2) {$R \E \mu_{u , v}^{\mathcal F_A} \R$};
          \draw (bot) to[in=-90,out=90] (r);
          \draw (r) to[in=-90,out=90] (mu);
          \draw (mu) to[in=-90,out=90] (top);
        \end{tikzpicture}
        =
        \begin{tikzpicture}[baseline=45, smallstring]
          \node (top) at (1,4) {$R \E \mathcal F_A \E u \R \mathcal F_A \E v \R \R$};
          \node (bot) at (1,0) {$\mathcal F_B \E u v \R$};
          \node[draw, rectangle] (mu) at (1,1) {$\mu_{u , v}^{\mathcal F_B}$};
          \node[draw, rectangle] (r1) at (0,2) {$r_u$};
          \node[draw, rectangle] (r2) at (2,2) {$r_v$};
          \node[draw, rectangle] (rho) at (1,3) {$\rho_{\mathcal F_A \E u \R , \mathcal F_A \E v \R}$};
          \draw (bot) to[in=-90,out=90] (mu);
          \draw (mu.135) to[in=-90,out=90] (r1);
          \draw (mu.45) to[in=-90,out=90] (r2);
          \draw (r1) to[in=-90,out=90] (rho.225);
          \draw (r2) to[in=-90,out=90] (rho.-45);
          \draw (rho) to[in=-90,out=90] (top);
        \end{tikzpicture}
      \end{align*}
      where $e_{a,\mathcal F_A \E v \R} \in \mathcal B \E a \mathcal F_A \E v \R \to \mathcal F_A \E v \R a \R$ is the half-braiding on $\mathcal F_A \E v \R$ (from the Drinfeld center), and similarly for $e_{R \E a \R , \mathcal F_B \E v \R}$.
    \item 2-cells $\E R , \rho , r \R \Rightarrow \E S , \sigma , s \R$ are monoidal natural transformations $\Theta : \E R , \rho \R \to \E S , \sigma \R$ such that $r_v \circ \Theta_{\mathcal F_A \E v \R} = s_v$.
    \item Composition is the usual vertical composition of natural transformations: $\E \Theta \compv \varphi \R_v = \Theta_v \circ \varphi_v$.
    \item Unit 2-cells are the identity natural transformations.
    \end{itemize}
  \item Composition functors:
    \begin{itemize}
    \item Given $\E R , \rho , r \R : A \to B$ and $\E S , \sigma , s \R :  B \to C$, we form their composite via functor $R \circ S$, laxitor $\sigma_{R \E a \R , R \E b \R} \circ S \E \rho_{a , b} \R$ and coherence $r_v \circ S \E s_v \R$.
    \item If $A , B , C \in \vmodtens$, $R_1 , R_2 \in \vmodtens \E A \to B \R$, $S_1 , S_2 \in \vmodtens \E B \to C \R$, and $\Theta : R_1 \Rightarrow R_2$ and $\varphi : S_1 \Rightarrow S_2$ are 2-cells, the horizontal composite $\Theta \varphi$ is defined via
      \begin{align*}
        \E \Theta \varphi \R_a = \varphi_{R_1 \E a \R} \circ S_2 \E \Theta_a \R.
      \end{align*}
    \end{itemize}
  \item Unit 1-cells:
    For each 0-cell $( A , \mathcal F_A^Z ) \in \vmodtens$ we define the unit 1-cell to be the triple consisting of the identity functor on $A$, the identity laxitor, and the identity natural transformation on $\mathcal F_A$. 
    
  \end{itemize}

\end{defn}

\begin{remark}
  With the above definition, $\vmodtens$ is a 2-category:
  \begin{itemize}
  \item For each pair $\E A , \mathcal F_A^Z \R , \E B , \mathcal F_B^Z \R$ of 0-cells in $\vmodtens$, $\vmodtens \E \E A , \mathcal F_A^Z \R \to \E B , \mathcal F_B^Z \R \R$ is a category:
    \begin{itemize}
    \item Unit 2-cells are units under composition, since they are identity natural transformations with the usual composition of natural transformations.
    \item Composition of natural transformations is associative.
    \end{itemize}
  \item Composition functors are unital because the unit 1-cells are identity functors with identity laxitors and the identity natural transformation for action-coherence.
  \item Composition functors are associative because horizontal composition of natural transformations, composition of monoidal functors, and vertical composition of natural transformations are all associative.
  \end{itemize}
\end{remark}


\section{Calculations using mates}\label{mate calc}
Given an adjunction $\mathcal A \E \mathcal L \E x \R \to b \R \cong \mathcal B \E x \to \mathcal R \E b \R \R$ and a morphism $f \in \mathcal A \E \mathcal L \E x \R \to b \R$, the \emph{mate} of $f$ is the corresponding morphism $\mate \E f \R \in \mathcal B \E x \to \mathcal R \E b \R \R$.

\begin{remark}
  A large number of our discussions focus on the following adjunction (Adjunction 4.2 in \cite{MR3961709}): Let $\mathcal A$ be a 0-cell in $\vmoncat$. Then for each $a , b \in \mathcal A$ and each $v \in \mathcal V$, we have the adjunction
  
  \begin{equation}\label{Main Adjunction} 
    A \E a \mathcal F_A \E v \R \to b \R \cong \mathcal V \E v \to \mathcal A \E a \to b \R \R
  \end{equation}

  where $A = \mathcal A^{\mathcal V}$ is the underlying category of $\mathcal A$.  In particular, note that when $v = 1_{\mathcal V}$ the adjunction becomes $\mathcal A^{\mathcal V} \E a \to b \R \cong \mathcal V \E 1_{\mathcal V} \to \mathcal A \E a \to b \R \R$.  In this case, we actually have equality by definition of $\mathcal A^{\mathcal V}$, and, although slightly redundant, we often use a superscript to keep track of which side of this equality adjunction a morphism lives on. There is no ambiguity even without a notational difference, as a quick check of source and target will clear up any confusion.
\end{remark}
\begin{defn}
  We define the \emph{unit} of this adjunction for $v \in \mathcal V$ as
  \begin{align*}
    \eta_v^{\mathcal F_A} \coloneqq \mate \E \id_{\mathcal F_A \E v \R} \R
  \end{align*}
  and the \emph{counit} of this adjunction for $a , b \in A$ as
  \begin{align*}
    \epsilon_{a \to b}^{\mathcal F_A} \coloneqq \mate \E \id_{\mathcal A \E a \to b \R} \R,
  \end{align*}
  both under the corresponding version of Adjunction \ref{Main Adjunction}.
\end{defn}
\begin{nota}
  We extend a useful notation from \cite{MR3961709}:
  \begin{align*}
    \DF a \to b ; c \to d ; \cdots \FD_{\mathcal F}^{\mathcal A} \coloneqq \mathcal F \E \mathcal A \E a \to b \R \mathcal A \E c \to d \R \cdots \R,
  \end{align*}
  where $\mathcal A$ is a $\mathcal V$-monoidal category, and $\mathcal F$ is any functor from $\mathcal V$.  Note that $\mathcal F$ need not be specifically $\mathcal F_A$, e.g., we use
  \begin{align*}
    \DF a \to b ; c \to d ; \cdots \FD_{\mathcal F_B}^{\mathcal A} \coloneqq \mathcal F_B \E \mathcal A \E a \to b \R \mathcal A \E c \to d \R \cdots \R.
  \end{align*}
\end{nota}
\begin{remark}
  Under Adjunction \ref{Main Adjunction}, the mate of $\blank \circ_{\mathcal A} \blank \in \mathcal V \E \mathcal A \E a \to b \R \mathcal A \E b \to c \R \to \mathcal A \E a \to c \R \R$ is
  
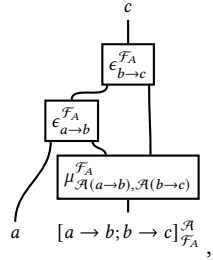
\begin{align*}
    \begin{tikzpicture}[smallstring]
      \node (a) at (0,0) {$a$};
      \node (bot) at (2,0) {$\DF a \to b ; b \to c \FD_{\mathcal F_A}^{\mathcal A}$};
      \node (top) at (2,4) {$c$};
      \node[draw, rectangle] (mu) at (2,1) {$\mu_{\mathcal A \E a \to b \R , \mathcal A \E b \to c \R}^{\mathcal F_A}$};
      \node[draw, rectangle] (ep1) at (1,2) {$\epsilon_{a \to b}^{\mathcal F_A}$};
      \node[draw, rectangle] (ep2) at (2,3) {$\epsilon_{b \to c}^{\mathcal F_A}$};
      \draw (a) to[in=-90,out=90] (ep1.225);
      \draw (bot) to[in=-90,out=90] (mu);
      \draw (mu.135) to[in=-90,out=90] (ep1.-45);
      \draw (mu.45) to[in=-90,out=90] (ep2.-45);
      \draw (ep1) to[in=-90,out=90] (ep2.225);
      \draw (ep2) to[in=-90,out=90] (top);
    \end{tikzpicture},
  \end{align*}
  and the mate of $\E \blank \otimes_{\mathcal A} \blank \R : \mathcal A \E a \to b \R \mathcal A \E c \to d \R \to \mathcal A \E a c \to b d \R$ is
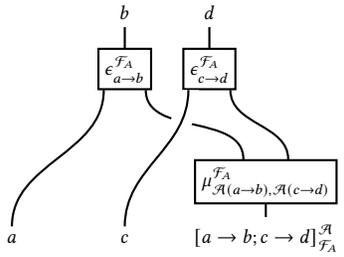
  \begin{align*}
    \begin{tikzpicture}[smallstring]
      \node (a) at (0,0) {$a$};
      \node (c) at (2,0) {$c$};
      \node (bot) at (4.5,0) {$\DF a \to b ; c \to d \FD_{\mathcal F_A}^{\mathcal A}$};
      \node (b) at (2,4) {$b$};
      \node (d) at (3.5,4) {$d$};
      \node[draw, rectangle] (mu) at (4.5,1) {$\mu_{\mathcal A \E a \to b \R , \mathcal A \E c \to d \R}^{\mathcal F_A}$};
      \node[draw, rectangle] (ep1) at (2,3) {$\epsilon_{a \to b}^{\mathcal F_A}$};
      \node[draw, rectangle] (ep2) at (3.5,3) {$\epsilon_{c \to d}^{\mathcal F_A}$};
      \draw (a) to[in=-90,out=90] (ep1.225);
      \draw (mu.135) to[in=-90,out=90] (ep1.-45);
      \draw[knot] (c) to[in=-90,out=90] (ep2.225);
      \draw(mu.45) to[in=-90,out=90] (ep2.-45);
      \draw (bot) to[in=-90,out=90] (mu);
      \draw (ep1) to[in=-90,out=90] (b);
      \draw (ep2) to[in=-90,out=90] (d);
    \end{tikzpicture}.
  \end{align*}
  Note that when going from $\vmodtens$ to $\vmoncat$ this is a definition, and in the other direction it is a proposition requiring proof (Proposition 4.9 in \cite{MR3961709} and Proposition 5.3 in \cite{MR3961709}).
\end{remark}

We now record several useful lemmas in computing mates; we start with a consequence of naturality.
\begin{lem}
  \label{Composition Lemma}
  Let $\mathcal A \E \mathscr L \E x \R \to b \R \cong \mathcal B \E x \to \mathscr R \E b \R \R$ be an adjunction.  If $f_1 \in \mathcal A \E \mathscr L \E x \R \to b \R$ and $f_2 \in \mathcal A \E b \to c \R$, then $\mate \E f_1 \circ f_2 \R = \mate \E f_1 \R \circ \mathscr R \E f_2 \R$.  Similarly, if $g_1 \in \mathcal B \E x \to y \R$ and $g_2 \in \mathcal B \E y \to \mathscr R \E b \R \R$, then $\mate \E g_1 \circ g_2 \R = \mathscr L \E g_1 \R \circ \mate \E g_2 \R$. In diagrams,
  \begin{align*}
    \mate \E
    \begin{tikzpicture}[baseline=25, smallstring]
      \node (ly) at (0,0) {$\mathscr L \E y \R$};
      \node (c) at (0,3) {$c$};
      \node[draw, rectangle] (f1) at (0,1) {$f_1$};
      \node[draw, rectangle] (f2) at (0,2) {$f_2$};
      \draw (ly) to (f1);
      \draw (f1) to (f2);
      \draw (f2) to (c);
    \end{tikzpicture} \R
    =
    \begin{tikzpicture}[baseline=25, smallstring]
      \node (y) at (0,0) {$y$};
      \node (rc) at (0,3) {$\mathscr R \E c \R$};
      \node[draw, rectangle] (f1) at (0,1) {$\mate \E f_1 \R$};
      \node[draw, rectangle] (f2) at (0,2) {$\mathscr R \E f_2 \R$};
      \draw (ly) to (f1);
      \draw (f1) to (f2);
      \draw (f2) to (c);
    \end{tikzpicture}
    \quad \qtand \quad
    \mate \E
    \begin{tikzpicture}[baseline=25, smallstring]
      \node (x) at (0,0) {$x$};
      \node (rb) at (0,3) {$\mathscr R \E b \R$};
      \node[draw, rectangle] (g1) at (0,1) {$g_1$};
      \node[draw, rectangle] (g2) at (0,2) {$g_2$};
      \draw (x) to (g1);
      \draw (g1) to (g2);
      \draw (g2) to (rb);
    \end{tikzpicture} \R
    =
    \begin{tikzpicture}[baseline=25, smallstring]
      \node (lx) at (0,0) {$\mathscr L \E x \R$};
      \node (b) at (0,3) {$b$};
      \node[draw, rectangle] (g1) at (0,1) {$\mathscr L \E g_1 \R$};
      \node[draw, rectangle] (g2) at (0,2) {$\mate \E g_2 \R$};
      \draw (lx) to (g1);
      \draw (g1) to (g2);
      \draw (g2) to (b);
    \end{tikzpicture}
  \end{align*}.
\end{lem}
An easy consequence of Lemma \ref{Composition Lemma} also turns out to be quite useful:
\begin{lem} \label{unit-counit mates}
  Take $a , b \in \mathcal A$, $v \in \mathcal V$, and $f \in \mathcal V \E v \to \mathcal A \E a \to b \R \R$.  Then under Adjunction \ref{Main Adjunction} we have
  \begin{align*}
    \mate \E
    \begin{tikzpicture}[baseline=20, smallstring]
      \node (bot) at (0,0) {$v$};
      \node (top) at (0,2) {$\mathcal A \E a \to b \R$};
      \node[draw, rectangle] (f) at (0,1) {$f$};
      \draw (bot) to (f);
      \draw (f) to (top);
    \end{tikzpicture} \R
    =
    \begin{tikzpicture}[baseline=25, smallstring]
      \node (a) at (0,0) {$a$};
      \node (v) at (2,0) {$\mathcal F_A \E v \R$};
      \node (b) at (1,3) {$b$};
      \node[draw, rectangle] (f) at (2,1) {$\mathcal F_A \E f \R$};
      \node[draw, rectangle] (epsilon) at (1,2) {$\epsilon_{a \to b}^{\mathcal F_A}$};
      \draw (a) to[in=-90,out=90] (epsilon.225);
      \draw (v) to[in=-90,out=90] (f);
      \draw (f) to[in=-90,out=90] (epsilon.-45);
      \draw (epsilon) to[in=-90,out=90] (b);
    \end{tikzpicture}.
  \end{align*}
  In the case that $v = 1_{\mathcal V}$, the expressions above are equal to $f$, viewed as a morphism in $A \E a \to b \R$. Similarly, if $v \in \mathcal V$, $a \in A$, and $g \in A \E \mathcal F_A \E v \R \to a \R = \mathcal V \E 1_{\mathcal V} \to \mathcal A \E \mathcal F_A \E v \R \to a \R \R$, then
  \begin{align*}
    \mate \E
    \begin{tikzpicture}[baseline=20, smallstring]
      \node (bot) at (0,0) {$\mathcal F \E v \R$};
      \node (top) at (0,2) {$a$};
      \node[draw, rectangle] (g) at (0,1) {$g$};
      \draw (bot) to (g);
      \draw (g) to (top);
    \end{tikzpicture} \R
    =
    \begin{tikzpicture}[baseline=25, smallstring]
      \node (v) at (0,0) {$v$};
      \node (a) at (1,3) {$\mathcal A \E 1_{\mathcal A} \to a \R$};
      \node[draw, rectangle] (g) at (2,1) {$g$};
      \node[draw, rectangle] (circ) at (1,2) {$\blank \circ_{\mathcal A} \blank$};
      \node[draw, rectangle] (eta) at (0,1) {$\eta_v^{\mathcal F_A}$};
      \draw (v) to[in=-90,out=90] (eta);
      \draw (eta) to[in=-90,out=90] (circ.225);
      \draw (g) to[in=-90,out=90] (circ.-45);
      \draw (circ) to[in=-90,out=90] (a);
    \end{tikzpicture}.
  \end{align*}
\end{lem}
\begin{lem}\label{mate of fgcirc}
  (Lemma 4.10 in \cite{MR3961709}) Suppose $f \in \mathcal V \E u \to \mathcal A \E a \to b \R \R$ and $g \in \mathcal V \E v \to \mathcal A \E b \to c \R \R$.  Then
  \begin{align*}
    \mate \E
    \begin{tikzpicture}[baseline=30, smallstring]
      \node (u) at (0,0) {$u$};
      \node (v) at (2,0) {$v$};
      \node (top) at (1,3) {$\mathcal A \E a \to c \R$};
      \node[draw, rectangle] (f) at (0,1) {$f$};
      \node[draw, rectangle] (g) at (2,1) {$g$};
      \node[draw, rectangle] (circ) at (1,2) {$\blank \circ_{\mathcal A} \blank$};
      \draw (u) to[in=-90,out=90] (f);
      \draw (v) to[in=-90,out=90] (g);
      \draw (f) to[in=-90,out=90] (circ.225);
      \draw (g) to[in=-90,out=90] (circ.-45);
      \draw (circ) to[in=-90,out=90] (top);
    \end{tikzpicture} \R
    =
    \begin{tikzpicture}[baseline=30, smallstring]
      \node (a) at (0,0) {$a$};
      \node (uv) at (2,0) {$\mathcal F_A \E u v \R$};
      \node (c) at (2,4) {$c$};
      \node[draw, rectangle] (mu) at (2,1) {$\mu_{u , v}$};
      \node[draw, rectangle] (f) at (1,2) {$\mate \E f \R$};
      \node[draw, rectangle] (g) at (2,3) {$\mate \E g \R$};
      \draw (uv) to[in=-90,out=90] (mu);
      \draw (a) to[in=-90,out=90] (f.225);
      \draw (mu.135) to[in=-90,out=90] (f.-45);
      \draw (mu.45) to[in=-90,out=90] (g.-45);
      \draw (f) to[in=-90,out=90] (g.225);
      \draw (g) to[in=-90,out=90] (c);
    \end{tikzpicture},
  \end{align*}
  where all mates are under the corresponding version of Adjunction \ref{Main Adjunction}.  If $u = 1_{\mathcal V}$ or $v = 1_{\mathcal V}$, then the oplaxitor $\mu$ is an identity.
\end{lem}
We have an analogous result for $\blank \otimes \blank$ as well:
\begin{lem} \label{mate of fgtens}
  Suppose $f \in \mathcal V \E u \to \mathcal A \E a \to c \R \R$ and $g \in \mathcal V \E v \to \mathcal A \E b \to d \R \R$, then we have
  \begin{align*}
    \mate \E
    \begin{tikzpicture}[baseline=30, smallstring]
      \node (u) at (0,0) {$u$};
      \node (v) at (2,0) {$v$};
      \node (top) at (1,3) {$\mathcal A \E a b \to c d \R$};
      \node[draw, rectangle] (f) at (0,1) {$f$};
      \node[draw, rectangle] (g) at (2,1) {$g$};
      \node[draw, rectangle] (tens) at (1,2) {$\blank \otimes_{\mathcal A} \blank$};
      \draw (u) to[in=-90,out=90] (f);
      \draw (v) to[in=-90,out=90] (g);
      \draw (f) to[in=-90,out=90] (tens.225);
      \draw (g) to[in=-90,out=90] (tens.-45);
      \draw (tens) to[in=-90,out=90] (top);
    \end{tikzpicture} \R
    =
    \begin{tikzpicture}[baseline=30, smallstring]
      \node (a) at (0,0) {$a$};
      \node (b) at (1,0) {$b$};
      \node (uv) at (3,0) {$\mathcal F_A \E u v \R$};
      \node (c) at (1,4) {$c$};
      \node (d) at (3,4) {$d$};
      \node[draw, rectangle] (mu) at (3,1) {$\mu_{u , v}$};
      \node[draw, rectangle] (f) at (1,3) {$\mate \E f \R$};
      \node[draw, rectangle] (g) at (3,3) {$\mate \E g \R$};
      \draw (uv) to[in=-90,out=90] (mu);
      \draw (a) to[in=-90,out=90] (f.225);
      \draw (mu.135) to[in=-90,out=90] (f.-45);
      \draw (mu.45) to[in=-90,out=90] (g.-45);
      \draw[knot] (b) to[in=-90,out=90] (g.225);
      \draw (f) to[in=-90,out=90] (c);
      \draw (g) to[in=-90,out=90] (d);
    \end{tikzpicture},
  \end{align*}
  where the crossing is the half-braiding on $\mathcal F_A \E u \R$ from the lift of $\mathcal F_A$ to $Z \E A \R$.  In the case where $u = 1_{\mathcal V}$ or $v = 1_{\mathcal V}$, the oplaxitor $\mu$ is again an identity.  If $u = 1_{\mathcal V}$ or $b = 1_{\mathcal A}$, then the half-braiding is an identity.
\end{lem}
\begin{proof}
  From the composition lemma, the mate of $\blank \otimes_{\mathcal A} \blank$, and naturality of $\mu$, we have
  \begin{align*}
    \mate \E
    \begin{tikzpicture}[baseline=30, smallstring]
      \node (u) at (0,0) {$u$};
      \node (v) at (2,0) {$v$};
      \node (top) at (1,3) {$\mathcal A \E a b \to c d \R$};
      \node[draw, rectangle] (f) at (0,1) {$f$};
      \node[draw, rectangle] (g) at (2,1) {$g$};
      \node[draw, rectangle] (tens) at (1,2) {$\blank \otimes_{\mathcal A} \blank$};
      \draw (u) to[in=-90,out=90] (f);
      \draw (v) to[in=-90,out=90] (g);
      \draw (f) to[in=-90,out=90] (tens.225);
      \draw (g) to[in=-90,out=90] (tens.-45);
      \draw (tens) to[in=-90,out=90] (top);
    \end{tikzpicture} \R
    =
    \begin{tikzpicture}[baseline=30, smallstring]
      \node (a) at (0,0) {$a$};
      \node (b) at (1,0) {$b$};
      \node (uv) at (4,0) {$\mathcal F_A \E u v \R$};
      \node (c) at (1.5,5) {$c$};
      \node (d) at (3,5) {$d$};
      \node[draw, rectangle] (fg) at (4,1) {$\mathcal F_A \E f g \R$};
      \node[draw, rectangle] (mu) at (4,2) {$\mu_{\mathcal A \E a \to c \R , \mathcal A \E b \to d \R}^{\mathcal F_A}$};
      \node[draw, rectangle] (epac) at (1.5,4) {$\epsilon_{a \to c}^{\mathcal F_A}$};
      \node[draw, rectangle] (epbd) at (3,4) {$\epsilon_{b \to d}^{\mathcal F_A}$};
      \draw (uv) to[in=-90,out=90] (fg);
      \draw (fg) to[in=-90,out=90] (mu);
      \draw (a) to[in=-90,out=90] (epac.225);
      \draw (mu.135) to[in=-90,out=90] (epac.-45);
      \draw (mu.45) to[in=-90,out=90] (epbd.-45);
      \draw[knot] (b) to[in=-90,out=90] (epbd.225);
      \draw (epac) to[in=-90,out=90] (c);
      \draw (epbd) to[in=-90,out=90] (d);
    \end{tikzpicture}
    =
    \begin{tikzpicture}[baseline=30, smallstring]
      \node (a) at (0,0) {$a$};
      \node (b) at (1,0) {$b$};
      \node (uv) at (4,0) {$\mathcal F_A \E u v \R$};
      \node (c) at (1.5,5) {$c$};
      \node (d) at (3,5) {$d$};
      \node[draw, rectangle] (mu) at (4,1) {$\mu_{u , v}^{\mathcal F_A}$};
      \node[draw, rectangle] (f) at (3,2) {$\mathcal F_A \E f \R$};
      \node[draw, rectangle] (g) at (5,2) {$\mathcal F_A \E g \R$};
      \node[draw, rectangle] (epac) at (1.5,4) {$\epsilon_{a \to c}^{\mathcal F_A}$};
      \node[draw, rectangle] (epbd) at (3,4) {$\epsilon_{b \to d}^{\mathcal F_A}$};
      \draw (uv) to[in=-90,out=90] (mu);
      \draw (mu.135) to[in=-90,out=90] (f);
      \draw (mu.45) to[in=-90,out=90] (g);
      \draw (a) to[in=-90,out=90] (epac.225);
      \draw (f) to[in=-90,out=90] (epac.-45);
      \draw (g) to[in=-90,out=90] (epbd.-45);
      \draw[knot] (b) to[in=-90,out=90] (epbd.225);
      \draw (epac) to[in=-90,out=90] (c);
      \draw (epbd) to[in=-90,out=90] (d);
    \end{tikzpicture}.
  \end{align*}
  Next, $\mathcal F_A$ lifts to $Z \E A \R$ and so $\mathcal F_A \E f \R$ can go under the crossing, and using lemma \ref{unit-counit mates}, we have
  \begin{align*}
    \begin{tikzpicture}[baseline=45, smallstring]
      \node (a) at (0,0) {$a$};
      \node (b) at (1,0) {$b$};
      \node (uv) at (3,0) {$u v$};
      \node (c) at (1.5,5) {$c$};
      \node (d) at (3.5,5) {$d$};
      \node[draw, rectangle] (mu) at (3,1) {$\mu_{u , v}^{\mathcal F_A}$};
      \node[draw, rectangle] (f) at (2,3) {$\mathcal F_A \E f \R$};
      \node[draw, rectangle] (g) at (4,2) {$\mathcal F_A \E g \R$};
      \node[draw, rectangle] (epac) at (1.5,4) {$\epsilon_{a \to c}^{\mathcal F_A}$};
      \node[draw, rectangle] (epbd) at (3.5,4) {$\epsilon_{b \to d}^{\mathcal F_A}$};
      \draw (uv) to[in=-90,out=90] (mu);
      \draw (mu.45) to[in=-90,out=90] (g);
      \draw (mu.135) to[in=-90,out=90] (f);
      \draw (a) to[in=-90,out=90] (epac.225);
      \draw (f) to[in=-90,out=90] (epac.-45);
      \draw (g) to[in=-90,out=90] (epbd.-45);
      \draw[knot] (b) to[in=-90,out=90] (epbd.225);
      \draw (epac) to[in=-90,out=90] (c);
      \draw (epbd) to[in=-90,out=90] (d);
    \end{tikzpicture}
    =
    \begin{tikzpicture}[baseline=45, smallstring]
      \node (a) at (0,0) {$a$};
      \node (b) at (1,0) {$b$};
      \node (uv) at (3,0) {$\mathcal F_A \E u v \R$};
      \node (c) at (1,4) {$c$};
      \node (d) at (3,4) {$d$};
      \node[draw, rectangle] (mu) at (3,1) {$\mu_{u , v}$};
      \node[draw, rectangle] (f) at (1,3) {$\mate \E f \R$};
      \node[draw, rectangle] (g) at (3,3) {$\mate \E g \R$};
      \draw (uv) to[in=-90,out=90] (mu);
      \draw (a) to[in=-90,out=90] (f.225);
      \draw (mu.135) to[in=-90,out=90] (f.-45);
      \draw (mu.45) to[in=-90,out=90] (g.-45);
      \draw[knot] (b) to[in=-90,out=90] (g.225);
      \draw (f) to[in=-90,out=90] (c);
      \draw (g) to[in=-90,out=90] (d);
    \end{tikzpicture}.
  \end{align*}
\end{proof}
To end the section we notice one more straightforward lemma that will be useful:
\begin{lem}\label{Mates of Faf}
  Given $f \in \mathcal V \E u \to v \R$, we can consider $\mathcal F_A \E f \R$ as a morphism in $\mathcal V \E 1_{\mathcal V} \to \mathcal A \E \mathcal F_A \E u \R \to \mathcal F_A \E v \R \R \R = A \E \mathcal F_A \E u \R \to \mathcal F_A \E v \R \R$. Then using Lemmas \ref{mate of fgcirc} and \ref{Composition Lemma}, respectively, we note $\mathcal F_A \E f \R$ is the mate of each of the following expressions, so they are equal:
  \begin{align*}
    \begin{tikzpicture}[baseline=30, smallstring]
      \node (top) at (0.5,3) {$\mathcal A \E 1_{\mathcal A} \to \mathcal F_A \E v \R \R$};
      \node (bot) at (0,0) {$u$};
      \node[draw, rectangle] (eta) at (0,1) {$\eta_u$};
      \node[draw, rectangle] (f) at (1,1) {$\mathcal F_A \E f \R$};
      \node[draw, rectangle] (circ) at (0.5,2) {$\blank \circ \blank$};
      \draw (bot) to[in=-90,out=90] (eta);
      \draw (eta) to[in=-90,out=90] (circ.225);
      \draw (f) to[in=-90,out=90] (circ.-45);
      \draw (circ) to[in=-90,out=90] (top);
    \end{tikzpicture}
    =
    \begin{tikzpicture}[baseline=30, smallstring]
      \node (top) at (0,3) {$\mathcal A \E 1_{\mathcal A} \to \mathcal F_A \E v \R \R$};
      \node (bot) at (0,0) {$u$};
      \node[draw, rectangle] (eta) at (0,2) {$\eta_v$};
      \node[draw, rectangle] (f) at (0,1) {$f$};
      \draw (bot) to[in=-90,out=90] (f);
      \draw (f) to[in=-90,out=90] (eta);
      \draw (eta) to[in=-90,out=90] (top);
    \end{tikzpicture}.
  \end{align*}
\end{lem}


\section{The 2-functor} \label{Define 2-functor}
In this section, we construct a map $P : \vmoncat \to \vmodtens$ and prove that $P$ is a 2-functor. In the following section we show that $P$ is in fact a 2-equivalence.
\subsection{Map on 0-cells} \label{Map on 0-cells}
Our map on 0-cells is the bijective correspondence from \cite{MR3961709}: we set $P \E \mathcal A \R = \E A , \mathcal F_A^{\mathcal Z} \R$, with $A \coloneqq \mathcal A^{\mathcal V}$ the underlying category of $\mathcal A$ and $\mathcal F_A^Z$ the left adjoint of $x \mapsto \mathcal A \E 1_{\mathcal A} \to x \R$ lifted to the center via the half-braidings
\begin{equation}\label{Define half-braidings}
  e_{a , \mathcal F_A \E v \R} \coloneqq
  \mate \E
  \begin{tikzpicture}[baseline=30, smallstring]
    \node (top) at (1,3) {$\mathcal A \E 1_{\mathcal A} \to a \R$};
    \node (bot) at (0,0) {$v$};
    \node[draw, rectangle] (eta) at (0,1) {$\eta_v^{\mathcal F_A}$};
    \node[draw, rectangle] (j) at (2,1) {$j_a^{\mathcal A}$};
    \node[draw, rectangle] (tens) at (1,2) {$\blank \otimes \blank$};
    \draw (bot) to[in=-90,out=90] (eta);
    \draw (eta) to[in=-90,out=90] (tens.225);
    \draw (j) to[in=-90,out=90] (tens.-45);
    \draw (tens) to[in=-90,out=90] (top);
  \end{tikzpicture} \R.
\end{equation}
\begin{remark}
  To lift $\mathcal F_A$ to $\mathcal F_A^Z$, it is also necessary to show that the images of morphisms in $\mathcal V$ are morphisms in $Z \E A \R$.  This was omitted from \cite{MR3961709}, so we include the short proof here.
  \begin{proof}
    To prove the remark, we need to show that we can pull morphisms in the image of $\mathcal F_A$ under the half-braiding, i.e.,
    \begin{align*}
      \begin{tikzpicture}[baseline=30, smallstring]
        \node (a-top) at (2,3) {$a$};
        \node (a-bot) at (0,0) {$a$};
        \node (u) at (2,0) {$\mathcal F_A \E u \R$};
        \node (v) at (0,3) {$\mathcal F_A \E v \R$};
        \node (empty) at (0,1) {};
        \node[draw, rectangle] (f) at (2,1) {$\mathcal F_A \E f \R$};
        \draw (u) to (f);
        \draw (f) to[in=-90,out=90] (v);
        \draw (a) to (empty.90);
        \draw[knot] (empty) to[in=-90,out=90] (a-top);
      \end{tikzpicture}
      =
      \begin{tikzpicture}[baseline=30, smallstring]
        \node (a-top) at (2,3) {$a$};
        \node (a-bot) at (0,0) {$a$};
        \node (u) at (2,0) {$\mathcal F_A \E u \R$};
        \node (v) at (0,3) {$\mathcal F_A \E v \R$};
        \node (empty) at (2,2) {};
        \node[draw, rectangle] (f) at (0,2) {$\mathcal F_A \E f \R$};
        \draw (u) to[in=-90,out=90] (f);
        \draw (f) to[in=-90,out=90] (v);
        \draw[knot] (a-bot) to[in=-90,out=90] (empty);
        \draw (empty.-90) to (a-top);
      \end{tikzpicture}    
    \end{align*}
    for all $f \in \mathcal V \E u \to v \R$. We take the mate of the left hand side using Lemma \ref{Composition Lemma} and Lemma 4.6 from \cite{MR3961709}, then apply a braided interchange, Lemma (\ref{Mates of Faf}), and Corollary 4.7 from \cite{MR3961709} in turn:
    \begin{align*}
      \mate \E
      \begin{tikzpicture}[baseline=30, smallstring]
        \node (a-top) at (2,3) {$a$};
        \node (a-bot) at (0,0) {$a$};
        \node (u) at (2,0) {$\mathcal F_A \E u \R$};
        \node (v) at (0,3) {$\mathcal F_A \E v \R$};
        \node (empty) at (0,1) {};
        \node[draw, rectangle] (f) at (2,1) {$\mathcal F_A \E f \R$};
        \draw (u) to (f);
        \draw (f) to[in=-90,out=90] (v);
        \draw (a) to (empty.90);
        \draw[knot] (empty) to[in=-90,out=90] (a-top);
      \end{tikzpicture} \R
      =
      \begin{tikzpicture}[baseline=45, smallstring]
        \node (top) at (2.25,5) {$\mathcal A \E a \to \mathcal F \E v \R a \R$};
        \node (bot) at (1,0) {$u$};
        \node[draw, rectangle] (j1) at (0,1) {$j_a$};
        \node[draw, rectangle] (eta) at (1,1) {$\eta_u$};
        \node[draw, rectangle] (j2) at (2,1) {$j_a$};
        \node[draw, rectangle] (f) at (3,1) {$\mathcal F_A \E f \R$};
        \node[draw, rectangle] (tens1) at (0.5,2) {$\blank \otimes \blank$};
        \node[draw, rectangle] (tens2) at (2.5,2) {$\blank \otimes \blank$};
        \node[draw, rectangle] (circ1) at (1.5,3) {$\blank \circ \blank$};
        \node[draw, rectangle] (e) at (3,3) {$e_{a , \mathcal F_A \E v \R}$};
        \node[draw, rectangle] (circ2) at (2.25,4) {$\blank \circ \blank$};
        \draw (bot) to[in=-90,out=90] (eta);
        \draw (j1) to[in=-90,out=90] (tens1.225);
        \draw (eta) to[in=-90,out=90] (tens1.-45);
        \draw (j2) to[in=-90,out=90] (tens2.225);
        \draw (f) to[in=-90,out=90] (tens2.-45);
        \draw (tens1) to[in=-90,out=90] (circ1.225);
        \draw (tens2) to[in=-90,out=90] (circ1.-45);
        \draw (circ1) to[in=-90,out=90] (circ2.225);
        \draw (e) to[in=-90,out=90] (circ2.-45);
        \draw (circ2) to[in=-90,out=90] (top);
      \end{tikzpicture} 
      =
      \begin{tikzpicture}[baseline=45, smallstring]
        \node (top) at (1.5,5) {$\mathcal A \E a \to \mathcal F \E v \R a \R$};
        \node (bot) at (0.5,0) {$u$};
        \node[draw, rectangle] (eta) at (0.5,1) {$\eta_u$};
        \node[draw, rectangle] (f) at (1.5,1) {$\mathcal F_A \E f \R$};
        \node[draw, rectangle] (j) at (0,2) {$j_a$};
        \node[draw, rectangle] (circ1) at (1,2) {$\blank \circ \blank$};
        \node[draw, rectangle] (tens) at (0.5,3) {$\blank \otimes  \blank$};
        \node[draw, rectangle] (e) at (2.5,3) {$e_{a , \mathcal F_A \E v \R}$};
        \node[draw, rectangle] (circ2) at (1.5,4) {$\blank \circ \blank$};
        \draw (bot) to[in=-90,out=90] (eta);
        \draw (j) to[in=-90,out=90] (tens.225);
        \draw (eta) to[in=-90,out=90] (circ1.225);
        \draw (f) to[in=-90,out=90] (circ1.-45);
        \draw (circ1) to[in=-90,out=90] (tens.-45);
        \draw (j) to[in=-90,out=90] (tens.225);
        \draw (tens) to[in=-90,out=90] (circ2.225);
        \draw (e) to[in=-90,out=90] (circ2.-45);
        \draw (circ2) to[in=-90,out=90] (top);
      \end{tikzpicture} 
      =
      \begin{tikzpicture}[baseline=45, smallstring]
        \node (top) at (1.5,5) {$\mathcal A \E a \to \mathcal F \E v \R a \R$};
        \node (bot) at (1,0) {$u$};
        \node[draw, rectangle] (f) at (1,1) {$f$};
        \node[draw, rectangle] (eta) at (1,2) {$\eta_v$};
        \node[draw, rectangle] (j) at (0,2) {$j_a$};
        \node[draw, rectangle] (tens) at (0.5,3) {$\blank \otimes  \blank$};
        \node[draw, rectangle] (e) at (2.5,3) {$e_{a , \mathcal F_A \E v \R}$};
        \node[draw, rectangle] (circ) at (1.5,4) {$\blank \circ \blank$};
        \draw (bot) to[in=-90,out=90] (f);
        \draw (j) to[in=-90,out=90] (tens.225);
        \draw (f) to[in=-90,out=90] (eta);
        \draw (eta) to[in=-90,out=90] (tens.-45);
        \draw (j) to[in=-90,out=90] (tens.225);
        \draw (tens) to[in=-90,out=90] (circ.225);
        \draw (e) to[in=-90,out=90] (circ.-45);
        \draw (circ) to[in=-90,out=90] (top);
      \end{tikzpicture} 
      =
      \begin{tikzpicture}[baseline=25, smallstring]
        \node (top) at (0,3) {$\mathcal A \E a \to \mathcal F \E v \R a \R$};
        \node (bot) at (0,0) {$u$};
        \node[draw, rectangle] (f) at (0,1) {$f$};
        \node[draw, rectangle] (e) at (0,2) {$\mate \E e_{a , \mathcal F_A \E v \R} \R$};
        \draw (bot) to[in=-90,out=90] (f);
        \draw (f) to[in=-90,out=90] (e);
        \draw (e) to[in=-90,out=90] (top);
      \end{tikzpicture} 
    \end{align*}
    From here we use the definition of the half-braiding, apply Lemma (\ref{Mates of Faf}) again, use another braided interchange after adding an identity, and recognize the result as the mate of the right hand side, so that $\mathcal F_A \E f \R$ is in fact a morphism in $Z \E A \R$
    \begin{align*}
      \begin{tikzpicture}[baseline=40, smallstring]
        \node (top) at (0.5,4) {$\mathcal A \E a \to \mathcal F \E v \R a \R$};
        \node (bot) at (0,0) {$u$};
        \node[draw, rectangle] (f) at (0,1) {$f$};
        \node[draw, rectangle] (eta) at (0,2) {$\eta_v$};
        \node[draw, rectangle] (j) at (1,2) {$j_a$};
        \node[draw, rectangle] (tens) at (0.5,3) {$\blank \otimes  \blank$};
        \draw (bot) to[in=-90,out=90] (f);
        \draw (eta) to[in=-90,out=90] (tens.225);
        \draw (f) to[in=-90,out=90] (eta);
        \draw (j) to[in=-90,out=90] (tens.-45);
        \draw (tens) to[in=-90,out=90] (top);
      \end{tikzpicture}
      =
      \begin{tikzpicture}[baseline=40, smallstring]
        \node (top) at (1,4) {$\mathcal A \E a \to \mathcal F \E v \R a \R$};
        \node (bot) at (0,0) {$u$};
        \node[draw, rectangle] (eta) at (0,1) {$\eta_u$};
        \node[draw, rectangle] (f) at (1,1) {$\mathcal F_A \E f \R$};
        \node[draw, rectangle] (j) at (1.5,2) {$j_a$};
        \node[draw, rectangle] (circ1) at (0.5,2) {$\blank \circ \blank$};
        \node[draw, rectangle] (tens) at (1,3) {$\blank \otimes \blank$};
        \draw (bot) to[in=-90,out=90] (eta);
        \draw (j) to[in=-90,out=90] (tens.-45);
        \draw (eta) to[in=-90,out=90] (circ1.225);
        \draw (f) to[in=-90,out=90] (circ1.-45);
        \draw (circ1) to[in=-90,out=90] (tens.225);
        \draw (tens) to[in=-90,out=90] (top);
      \end{tikzpicture}
      =
      \begin{tikzpicture}[baseline=40, smallstring]
        \node (top) at (1.5,4) {$\mathcal A \E a \to \mathcal F \E v \R a \R$};
        \node (bot) at (0,0) {$u$};
        \node[draw, rectangle] (eta) at (0,1) {$\eta_u$};
        \node[draw, rectangle] (j1) at (1,1) {$j_a$};
        \node[draw, rectangle] (f) at (2,1) {$\mathcal F_A \E f \R$};
        \node[draw, rectangle] (j2) at (3,1) {$j_a$};
        \node[draw, rectangle] (tens1) at (0.5,2) {$\blank \otimes \blank$};
        \node[draw, rectangle] (tens2) at (2.5,2) {$\blank \otimes \blank$};
        \node[draw, rectangle] (circ) at (1.5,3) {$\blank \circ \blank$};
        \draw (bot) to[in=-90,out=90] (eta);
        \draw (j1) to[in=-90,out=90] (tens1.-45);
        \draw (eta) to[in=-90,out=90] (tens1.225);
        \draw (f) to[in=-90,out=90] (tens2.225);
        \draw (j2) to[in=-90,out=90] (tens2.-45);
        \draw (tens1) to[in=-90,out=90] (circ.225);
        \draw (tens2) to[in=-90,out=90] (circ.-45);
        \draw (circ) to[in=-90,out=90] (top);
      \end{tikzpicture}
      =
      \mate \E
      \begin{tikzpicture}[baseline=30, smallstring]
        \node (a-top) at (2,3) {$a$};
        \node (a-bot) at (0,0) {$a$};
        \node (u) at (2,0) {$\mathcal F_A \E u \R$};
        \node (v) at (0,3) {$\mathcal F_A \E v \R$};
        \node (empty) at (2,2) {};
        \node[draw, rectangle] (f) at (0,2) {$\mathcal F_A \E f \R$};
        \draw (u) to[in=-90,out=90] (f);
        \draw (f) to[in=-90,out=90] (v);
        \draw[knot] (a-bot) to[in=-90,out=90] (empty);
        \draw (empty.-90) to (a-top);
      \end{tikzpicture} \R
    \end{align*}
  \end{proof}
\end{remark}


\subsection{Map on 1-cells: Definition}
Given 0-cells $\mathcal A , \mathcal B \in \vmoncat$ and a 1-cell $\E \mathcal R , \rho^R \R \in \vmoncat \E \mathcal A \to \mathcal B \R$, we construct a 1-cell $\E R , \rho , r \R$ in $\vmodtens \E P \E \mathcal A \R \to P \E \mathcal B \R \R$. Define $R : A \to B$ to be the underlying functor of $\mathcal R$, and define the laxitor $\rho$ via, for $a , b \in A$, $\rho_{a , b} \coloneqq \mate \E \rho_{a , b}^R \R = \rho_{a , b}^R$ under the identity adjunction
\begin{align*}
  \mathcal B \E R \E a \R R \E b \R \to R \E a b \R \R = \mathcal V \E 1_{\mathcal V} \to B \E R \E a \R R \E b \R \to R \E a b \R \R \R.
\end{align*}
Lastly, we define the natural transformation $r$ via
\begin{equation}\label{Define r}
  r_v = \mate \E
  \begin{tikzpicture}[baseline=30, smallstring]
    \node (top) at (0,3) {$\mathcal B \E 1_{\mathcal B} \to R \E \mathcal F_A \E v \R \R \R$};
    \node (bot) at (0,0) {$v$};
    \node[draw, rectangle] (eta) at (0,1) {$\eta_v^{\mathcal F_A}$};
    \node[draw, rectangle] (R) at (0,2) {$\mathcal R_{1_{\mathcal A} \to \mathcal F_A \E v \R}$};
    \draw (bot) to (eta);
    \draw (eta) to (R);
    \draw (R) to (top);
  \end{tikzpicture} \R
\end{equation}
\begin{remark}
  Following immediately from the corresponding properties in $\vmoncat$, we have:
  \begin{itemize}
  \item The pair $\E R , \rho \R$ is a monoidal functor.
  \item The functor $R$ is strictly unital: $R \E 1_{\mathcal A} \R = 1_{\mathcal B}$.
  \item The laxitor $\rho$ is strictly unital: $\rho_{1_{\mathcal A} , a} = \id_{R \E a \R} = \rho_{a , 1_{\mathcal A}}$.
  \item The natural transformation $r$ is strictly unital: $r_{1_{\mathcal V}} = \id_{1_{\mathcal B}}$.
  \end{itemize}
\end{remark}
\begin{prop}
  The natural transformation $r$ is in fact natural, i.e., for each $f \in \mathcal V \E u \to v \R$, we have
  \begin{equation*}

  \end{equation*}
\end{prop}
\begin{proof}
  The mate of the left hand side under Adjunction \ref{Main Adjunction}, using Lemmas \ref{Composition Lemma} and \ref{Mates of Faf} along with functoriality of $\mathcal R$, is
  \begin{equation*}
    %
  \end{equation*}
  which is exactly the mate of the right hand side.
\end{proof}


\begin{prop}
  The natural transformation $r$ plays nicely with the half braidings $e_{a , \mathcal F_A \E v \R}$ defined in Equation \ref{Define half-braidings}, i.e.,
  \begin{equation*}
    %
  \end{equation*}
\end{prop}
\begin{proof}
  Taking the mate of the left hand side, then using braided interchange, simplifying, and recognizing the mate of $r_v$ via lemma \ref{unit-counit mates} gives
  \begin{equation*}
    %
  \end{equation*}
  Next, use unitality and monoidality of $\mathcal R$, then simplify to see that this is equal to
  \begin{equation*}
    %
  \end{equation*}
  Taking the mate of the right hand side, then applying braided interchange followed by lemma \ref{unit-counit mates} for the mate of $r_v$, we get
  \begin{equation*}
    %
  \end{equation*}
  Lastly, using naturality of $\mathcal R$ and unitality of $\rho^R$, then functoriality of $\mathcal R$ and Corollary 4.7 from \cite{MR3961709}, we see that this is equal to the mate of the left hand side from above:
  \begin{equation*}
    %
  \end{equation*}
\end{proof}


\begin{prop}
  The action-coherence condition is satisfied:
  \begin{align*}
    %
  \end{align*}
\end{prop}
\begin{proof}
  Using Lemma \ref{mate of fgcirc}, functoriality of $\mathcal R$, and Lemma \ref{unit-counit mates}, the mate of the left hand side is
  \begin{align*}
    %
  \end{align*}
  After expanding using the mate of $\mu$ and applying braided interchange, the mate of the right hand side becomes:
  \begin{align*}
    %
  \end{align*}
  Then we use lemma \ref{unit-counit mates} twice, the definition of $r_v$, and naturality and unitality of $\rho$ to see that this is equal to
  \begin{align*}
    %
  \end{align*}
  where the last equality is the definition of $\mu$, matching the mate of the left hand side.
\end{proof}
All together this proves:
\begin{prop}
  The triple $\E R , \rho , r \R$ is a 1-cell in $\vmodtens$.
\end{prop}


\subsection{Map on 2-cells}
Suppose we have a $\mathcal V$-monoidal natural transformation $\theta : \mathcal R \Rightarrow \mathcal S$, i.e., a 2-cell in $\vmoncat$.  We define a 2-cell $P_{\mathcal A \to \mathcal B} \E \theta \R : R \Rightarrow S$ between the corresponding 1-cells in $\vmodtens$ via, for each $a \in \mathcal A$, $P_{\mathcal A \to \mathcal B} \E \theta \R_a \coloneqq \mate \E \theta_a \R$ under the identity adjunction
\begin{align*}
  \mathcal V \E 1_{\mathcal V} \to \mathcal B \E \mathcal R \E a \R \to \mathcal S \E a \R \R \R = \mathcal B \E \mathcal R \E a \R \to \mathcal S \E a \R \R.
\end{align*}
We write $\Theta_a \coloneqq P_{\mathcal A \to \mathcal B} \E \theta \R_a$ for brevity.
\begin{prop}
  As defined above, $\Theta$ is a 2-cell in $\vmodtens$
\end{prop}
\begin{proof}
  Naturality and monoidality of $\Theta$ follow easily from the corresponding properties in the enriched setting: taking the appropriate mate of each side of the required conditions gives exactly the corresponding condition in $\vmoncat$. It remains to check that $s_v \circ \Theta_{\mathcal F_A \E v \R} = \rho_v$ for all $v \in \mathcal V$.\\
  From the hexagon axiom in $\vmoncat$, with $a = 1_{\mathcal A}$ and $b = \mathcal F_A \E v \R$, we have
  \begin{align*}
    \begin{tikzpicture}[baseline=30, smallstring]
      \node (bot) at (0,0) {$\mathcal A \E 1_{\mathcal A} \to \mathcal F_A \E v \R \R$};
      \node (top) at (1,3) {$\mathcal B \E 1_{\mathcal B} \to S \E \mathcal F_A \E v \R \R \R$};
      \node[draw, rectangle] (r) at (0,1) {$R_{1_{\mathcal A} \to \mathcal F_A \E v \R}$};
      \node[draw, rectangle] (theta) at (2,1) {$\theta_{\mathcal F_A \E v \R}$};
      \node[draw, rectangle] (circ) at (1,2) {$\blank \circ_{\mathcal B} \blank$};
      \draw (bot) to (r);
      \draw (r) to[in=-90,out=90] (circ.225);
      \draw (theta) to[in=-90,out=90] (circ.-45);
      \draw (circ) to (top);
    \end{tikzpicture}
    =
    \begin{tikzpicture}[baseline=30, smallstring]
      \node (bot) at (2,0) {$\mathcal A \E 1_{\mathcal A} \to \mathcal F_A \E v \R \R$};
      \node (top) at (1,3) {$\mathcal B \E 1_{\mathcal B} \to S \E \mathcal F_A \E v \R \R \R$};
      \node[draw, rectangle] (s) at (2,1) {$S_{1_{\mathcal A} \to \mathcal F_A \E v \R}$};
      \node[draw, rectangle] (theta) at (0,1) {$\theta_{1_{\mathcal A}}$};
      \node[draw, rectangle] (circ) at (1,2) {$\blank \circ_{\mathcal B} \blank$};
      \draw (bot) to (s);
      \draw (s) to[in=-90,out=90] (circ.-45);
      \draw (theta) to[in=-90,out=90] (circ.225);
      \draw (circ) to (top);
    \end{tikzpicture}
  \end{align*}
  Next, precompose with $\eta_v^{\mathcal F_A}$ on both sides, and use strict unitality of $\theta$:
  \begin{align*}
    \begin{tikzpicture}[baseline=50, smallstring]
      \node (bot) at (0,0) {$v$};
      \node (top) at (1,4) {$\mathcal B \E 1_{\mathcal B} \to S \E \mathcal F_A \E v \R \R \R$};
      \node[draw, rectangle] (eta) at (0,1) {$\eta_v^{\mathcal F_A}$};
      \node[draw, rectangle] (r) at (0,2) {$R_{1_{\mathcal A} \to \mathcal F_A \E v \R}$};
      \node[draw, rectangle] (theta) at (2,2) {$\theta_{\mathcal F_A \E v \R}$};
      \node[draw, rectangle] (circ) at (1,3) {$\blank \circ_{\mathcal B} \blank$};
      \draw (bot) to (eta);
      \draw (eta) to (r);
      \draw (r) to[in=-90,out=90] (circ.225);
      \draw (theta) to[in=-90,out=90] (circ.-45);
      \draw (circ) to (top);
    \end{tikzpicture}
    =
    \begin{tikzpicture}[baseline=35, smallstring]
      \node (bot) at (0,0) {$v$};
      \node (top) at (0,3) {$\mathcal B \E 1_{\mathcal B} \to S \E \mathcal F_A \E v \R \R \R$};
      \node[draw, rectangle] (eta) at (0,1) {$\eta_v^{\mathcal F_A}$};
      \node[draw, rectangle] (s) at (0,2) {$S_{1_{\mathcal A} \to \mathcal F_A \E v \R}$};
      \draw (bot) to (eta);
      \draw (eta) to (s);
      \draw (s) to (top);
    \end{tikzpicture}
  \end{align*}
  Now the right hand side is exactly the mate of $r_v$, and by Lemma \ref{mate of fgcirc}, the mate of the left hand side is
  \begin{align*}
    \begin{tikzpicture}[baseline=25, smallstring]
      \node (bot) at (0,0) {$\mathcal F_B \E v \R$};
      \node (top) at (0,3) {$S \E \mathcal F_A \E v \R \R$};
      \node[draw, rectangle] (r) at (0,1) {$\mate \E \eta_v^{\mathcal F_A} \circ R_{1_{\mathcal A} \to \mathcal F_A \E v \R} \R$};
      \node[draw, rectangle] (theta) at (0,2) {$\mate \E \theta_{\mathcal F_A \E v \R} \R$};
      \draw (bot) to (r);
      \draw (r) to (theta);
      \draw (theta) to (top);
    \end{tikzpicture}
    =
    \begin{tikzpicture}[baseline=25, smallstring]
      \node (bot) at (0,0) {$\mathcal F_B \E v \R$};
      \node (top) at (0,3) {$S \E \mathcal F_A \E v \R \R$};
      \node[draw, rectangle] (s) at (0,1) {$s_v$};
      \node[draw, rectangle] (theta) at (0,2) {$\Theta_{\mathcal F_A \E v \R}$};
      \draw (bot) to (s);
      \draw (s) to (theta);
      \draw (theta) to (top);
    \end{tikzpicture}
  \end{align*}
  which then proves that $s_v \circ \Theta_{\mathcal F_A \E v \R} = r_v$, as desired.
\end{proof}

\begin{prop}
  With the above definitions, $P: \vmoncat \to \vmodtens$ is a 2-functor.
\end{prop}
We will take the remainder of this section to prove the proposition. We need to show that
\begin{itemize}
\item Each $P_{\mathcal A \to \mathcal A}$ preserves unit 1-cells, 
\item Each $P_{\mathcal A \to \mathcal B}$ is a functor, i.e., vertical composition of 2-cells is preserved
\item Horizontal composition of 2-cells is preserved, and 
\item Composition of 1-cells is preserved.
\end{itemize}
First, note that both $\id_{P \E \mathcal A \R}$ and $P_{\mathcal A \to \mathcal A} \E \id_{\mathcal A} \R$ are the identity 1-cell on $P \E \mathcal A \R$, so unit 1-cells are preserved. Next, since $P_{\mathcal A \to \mathcal B}$ on 2-cells is simply taking mates under the identity adjunction, it follows easily from the definitions that horizontal and vertical composition of 2-cells are both preserved. The bulk of the work is involved in showing that composition of 1-cells is preserved, i.e.,
\begin{align*}
  P_{\mathcal A \to \mathcal B} \E ( \mathcal R , \rho^R ) \R \circ P_{\mathcal B \to \mathcal C} \E ( \mathcal S , \sigma^S ) \R = P_{\mathcal A \to \mathcal C} \E ( \mathcal R , \rho^R ) \circ ( \mathcal S , \sigma^S ) \R.
\end{align*}
The left hand side is given by
\begin{align*}
  P_{\mathcal A \to \mathcal B} \E ( \mathcal R , \rho^R ) \R \circ P_{\mathcal B \to \mathcal C} \E ( \mathcal S , \sigma^S ) \R%
  &= ( \mathcal R^{\mathcal V} , \rho , r ) \circ ( \mathcal S^{\mathcal V} , \sigma , s ),
\end{align*}
where $\rho$, $\sigma$, $r$, and $s$ are as defined in \ref{Define r}, and the right hand side is
\begin{align*}
  P_{\mathcal A \to \mathcal C} \E ( \mathcal R , \rho^R ) \circ ( \mathcal S , \sigma^S ) \R%
  = P_{\mathcal A \to \mathcal C} \E \mathcal R \circ \mathcal S , \sigma^S \circ S ( \rho^R ) \R
  = \E \E \mathcal R \circ \mathcal S \R^{\mathcal V} , \mate \E \sigma^S \circ S ( \rho^R ) \R , t \R,
\end{align*}
where $\sigma^S \circ S ( \rho^R )$ is as defined in \ref{Define vmodtens} and $t \coloneqq \mate \E \eta_v^{\mathcal F_A} \circ ( \mathcal R_{1_{\mathcal A} \to \mathcal F_A \E v \R} \circ \mathcal S_{1_{\mathcal B} \to \mathcal R \E \mathcal F_A \E f \R \R} ) \R$. Now, $\E \mathcal R \circ \mathcal S \R^{\mathcal V} = \mathcal R^{\mathcal V} \circ \mathcal S^{\mathcal V}$, so these are equal as functors.  Further, since $\mate \E \sigma_{c , d} \R = \sigma_{c , d}^S$, we directly have
\begin{align*}
  \mate \E 

\end{align*}  
so they are also equal as lax monoidal functors.  It remains to check that for each $v \in \mathcal V$,
\begin{align*}
  s_v \circ S \E r_v \R = \mate \E \eta_v^{\mathcal F_A} \circ ( \mathcal R_{1_{\mathcal A} \to \mathcal F_A \E v \R} \circ \mathcal S_{1_{\mathcal B} \to \mathcal R \E \mathcal F_A \E f \R \R} ) \R.
\end{align*}
To unpack this line, we expand definitions and mates, starting with the right hand side using Lemma \ref{Composition Lemma}:
\begin{align*}
  \mate \E
  %
\end{align*}
Next, expanding the left hand side, we get
\begin{align*}
  %
\end{align*}
where the large boxes act as brackets for multiple morphisms in our diagrams. Then expanding these mates and using unitality of $\rho$ and $\sigma$, we have
\begin{align*}
  %
\end{align*}
To show these are equal, we break down the calculation into a few lemmas
\begin{lem}
  \begin{align*}
    %
  \end{align*}
\end{lem}
\begin{proof}
  Since $\mathcal F_C$ is a functor, it suffices to show that
  \begin{align*}
    %
  \end{align*}
  The left hand side is exactly the mate of $r_v$, so it suffices to show that the mate of the right hand side is $r_v$ as well.  Taking the mate of the right hand side above using Lemma \ref{Composition Lemma}, we get
  \begin{align*}
    %
\\
  \end{align*}
  Then by lemma \ref{unit-counit mates}, this is equal to 
  \begin{align*}
    %
  \end{align*}
\end{lem}
\begin{proof}
  On the left hand side we expand definitions, apply lemma \ref{mate of fgcirc} and lemma \ref{unit-counit mates}, and use the definition of the counit:
  \begin{align*}
    %
  \end{align*}
\end{lem}
\begin{proof}
  Expanding definitions and applying Lemma \ref{Composition Lemma}, $r_v$ is equal to
  \begin{align*}
    \mate \E
    %
  
  \end{align*}
\end{proof}
Thus, putting together the above lemmas, along with naturality of $s$, we get the equality:
\begin{align*}
  %
\end{align*}
Therefore composition of 1-cells is preserved, and $P$ is a 2-functor.


\section{Equivalence}

\begin{prop}
  The 2-functor $P : \vmoncat \to \vmodtens$ as defined above is essentially surjective on 0-cells.
\end{prop}

\begin{proof}
  To say that $P$ is essentially surjective is just to say that for every 0-cell $( A , \mathcal F_A^Z ) \in \vmodtens$, there exists  a 0-cell $\mathcal A \in \vmoncat$ such that $P \E \mathcal A \R \simeq ( A , \mathcal F_A^Z )$.  This is exactly the main theorem in \cite{MR3961709}.
\end{proof}

\begin{defn}
  Let $( A , \mathcal F_A^Z ) , ( B , \mathcal F_B^Z ) \in \vmodtens$ be 0-cells, and let $\E R , \rho , r \R : ( A , \mathcal F_A^Z ) \to ( B , \mathcal F_B^Z )$ be a 1-cell between them.  We construct a 1-cell $\E \mathcal R , \rho^R \R$ in $\vmoncat$ such that $P_{\mathcal A \to \mathcal B} \E \E \mathcal R , \rho^R \R \R = \E R , \rho , r \R$. Since $\mathcal A$ and $\mathcal B$ share the same objects as $A$ and $B$, respectively, we can define $\mathcal R \E a \R = R \E a \R$ for all $a \in \mathcal A$.  Next, for $a , b \in \mathcal A$, we make $\mathcal R$ into a $\mathcal V$-functor via
  \begin{equation}\label{V-Functor Definition}
    \mathcal R_{a \to b} \coloneqq \mate \E
    \begin{tikzpicture}[baseline=35, smallstring]
      \node (top) at (1.5,4) {$R \E b \R$};
      \node (a) at (0,0) {$R \E a \R$};
      \node (bot) at (3,0) {$\DF a \to b \FD_{\mathcal F_B}^{\mathcal A}$};
      \node[draw, rectangle] (r) at (3,0.75) {$r_{\mathcal A \E a \to b \R}$};
      \node[draw, rectangle] (rho) at (1.5,2) {$\rho_{a , \DF a \to b \FD_{\mathcal F_A}^{\mathcal \mathcal A}}$};
      \node[draw, rectangle] (rep) at (1.5,3) {$R \E \epsilon_{a \to b}^{\mathcal F_A} \R$};
      \draw (a) to[in=-90,out=90] (rho.225);
      \draw (bot) to[in=-90,out=90] (r);
      \draw (r) to[in=-90,out=90] (rho.-45);
      \draw (rho) to[in=-90,out=90] (rep);
      \draw (rep) to (top);
    \end{tikzpicture} \R
  \end{equation}
  under Adjunction \ref{Main Adjunction}, and tensorator $\rho^R$ given by
  \begin{align*}
    \rho_{a , b}^R \coloneqq \mate \E \sigma_{a , b} \R = \rho_{a , b}
  \end{align*}
  under the equality adjunction
  \begin{align*}
    B \E R \E a \R R \E b \R \to R \E a b \R \R = \mathcal V \E 1_{\mathcal V} \to \mathcal B \E R \E a \R R \E b \R \to R \E a b \R \R \R.
  \end{align*}

\end{defn}

\begin{lem}
  The map $\mathcal R$ defined above is functorial: For all $a , b , c \in \mathcal A$,
  \begin{align*}
    \begin{tikzpicture}[baseline=25, smallstring]
      \node (top) at (1,3) {$\mathcal B \E R \E a \R \to R \E c \R \R$};
      \node (ab) at (0,0) {$\mathcal A \E a \to b \R$};
      \node (bc) at (2,0) {$\mathcal A \E b \to c \R$};
      \node[draw, rectangle] (rab) at (0,1) {$\mathcal R_{a \to b}$};
      \node[draw, rectangle] (rbc) at (2,1) {$\mathcal R_{b \to c}$};
      \node[draw, rectangle] (circ) at (1,2) {$\blank \circ_{\mathcal B} \blank$};
      \draw (ab) to[in=-90,out=90] (rab);
      \draw (bc) to[in=-90,out=90] (rbc);
      \draw (rab) to[in=-90,out=90] (circ.225);
      \draw (rbc) to[in=-90,out=90] (circ.-45);
      \draw (circ) to[in=-90,out=90] (top);
    \end{tikzpicture}
    =
    \begin{tikzpicture}[baseline=25, smallstring]
      \node (top) at (1,3) {$\mathcal B \E R \E a \R \to R \E c \R \R$};
      \node (ab) at (0,0) {$\mathcal A \E a \to b \R$};
      \node (bc) at (2,0) {$\mathcal A \E b \to c \R$};
      \node[draw, rectangle] (circ) at (1,1) {$\blank \circ_{\mathcal A} \blank$};
      \node[draw, rectangle] (r) at (1,2) {$\mathcal R_{a \to c}$};
      \draw (ab) to[in=-90,out=90] (circ.225);
      \draw (bc) to[in=-90,out=90] (circ.-45);
      \draw (circ) to[in=-90,out=90] (r);
      \draw (r) to[in=-90,out=90] (top);
    \end{tikzpicture}
  \end{align*}
\end{lem}
\begin{proof}
  Using lemma \ref{mate of fgcirc}, expanding definitions, and using naturality of $\rho$, the left hand side has mate
  \begin{align*}
    \begin{tikzpicture}[baseline=45, smallstring]
      \node (r-a-bottom) at (0,0) {$R \E a \R$};
      \node (gmm) at (2,0) {$\DF a \to b ; b \to c \FD_{\mathcal F_B}^{\mathcal A}$};
      \node (r-c-top) at (1,5) {$R \E c \R$};
      \node[draw, rectangle] (mu) at (2,1) {$\mu_{\mathcal A \E a \to b \R , \mathcal A \E b \to c \R}^{\mathcal F_B}$};
      \node[draw, rectangle] (hab) at (0,3) {$\mate \E \mathcal R_{a \to b} \R$};
      \node[draw, rectangle] (hbc) at (1,4) {$\mate \E \mathcal R_{b \to c} \R$};
      \draw (r-a-bottom) to[in=-90,out=90] (hab);
      \draw (gmm) to[in=-90,out=90] (mu);
      \draw (mu.135) to[in=-90,out=90] (hab.-45);
      \draw (mu.45) to[in=-90,out=90] (hbc.-45);
      \draw (hab) to[in=-90,out=90] (hbc.235);
      \draw (hbc) -- (r-c-top);
    \end{tikzpicture}
    =
    \begin{tikzpicture}[baseline=70, smallstring]
      \node (r-a-bottom) at (0,0) {$R \E a \R$};
      \node (gmm) at (2,0) {$\DF a \to b ; b \to c \FD_{\mathcal F_B}^{\mathcal A}$};
      \node (r-c-top) at (1,7) {$R \E c \R$};
      \node[draw, rectangle] (mu) at (2,1) {$\mu_{\mathcal A \E a \to b \R , \mathcal A \E b \to c \R}^{\mathcal F_B}$};
      \node[draw, rectangle] (repab) at (0,4) {$R \E \epsilon_{a \to b}^{\mathcal F_A} \R$};
      \node[draw, rectangle] (ri1) at (0,3) {$\rho_{a , \mathcal F_A \E \mathcal A \E a \to b \R \R}$};
      \node[draw, rectangle] (repbc) at (1,6) {$R \E \epsilon_{b \to c}^{\mathcal F_A} \R$};
      \node[draw, rectangle] (ri2) at (1,5) {$\rho_{b , \mathcal F_A \E \mathcal A \E b \to c \R \R}$};
      \node[draw, rectangle] (r1) at (1,2) {$r_{\mathcal A \E a \to b \R}$};
      \node[draw, rectangle] (r2) at (3,2) {$r_{\mathcal A \E b \to c \R}$};
      \draw (r-a-bottom) to[in=-90,out=90] (ri1.235);
      \draw (gmm) to[in=-90,out=90] (mu);
      \draw (mu.135) to[in=-90,out=90] (r1);
      \draw (r1) to[in=-90,out=90] (ri1.-45);
      \draw (mu.45) to[in=-90,out=90] (r2);
      \draw (repab) to[in=-90,out=90] (ri2.235);
      \draw (ri1) -- (repab);
      \draw (r2) to[in=-90, out=90] (ri2.-45);
      \draw (ri2) -- (repbc);
      \draw (repbc) -- (r-c-top);
    \end{tikzpicture}
    =
    \begin{tikzpicture}[baseline=60, smallstring]
      \node (r-a-bottom) at (0,0) {$R \E a \R$};
      \node (gmm) at (2,0) {$\DF a \to b ; b \to c \FD_{\mathcal F_B}^{\mathcal A}$};
      \node (r-c-top) at (1,7) {$R \E c \R$};
      \node[draw, rectangle] (mu) at (2,1) {$\mu_{\mathcal A \E a \to b \R , \mathcal A \E b \to c \R}^{\mathcal F_B}$};
      \node[draw, rectangle] (repab) at (1,5) {$R \E \epsilon_{a \to b}^{\mathcal F_A} \id_{\mathcal F_A \E \mathcal A \E b \to c \R \R} \R$};
      \node[draw, rectangle] (ri1) at (0,3) {$\rho_{a , \mathcal F_A \E \mathcal A \E a \to b \R \R}$};
      \node[draw, rectangle] (repbc) at (1,6) {$R \E \epsilon_{b \to c}^{\mathcal F_A} \R$};
      \node[draw, rectangle] (ri2) at (1,4) {$\rho_{a \mathcal F_A \E \mathcal A \E a \to b \R \R , \mathcal F_A \E \mathcal A \E b \to c \R \R}$};
      \node[draw, rectangle] (r1) at (1,2) {$r_{\mathcal A \E a \to b \R}$};
      \node[draw, rectangle] (r2) at (3,2) {$r_{\mathcal A \E b \to c \R}$};
      \draw (r-a-bottom) to[in=-90,out=90] (ri1.235);
      \draw (gmm) to[in=-90,out=90] (mu);
      \draw (mu.135) to[in=-90,out=90] (r1);
      \draw (r1) to[in=-90,out=90] (ri1.-45);
      \draw (mu.45) to[in=-90,out=90] (r2);
      \draw (repab) to[in=-90,out=90] (repbc);
      \draw (ri1) to[in=-90,out=90] (ri2.225);
      \draw (r2) to[in=-90, out=90] (ri2.-45);
      \draw (ri2) to[in=-90,out=90] (repab);
      \draw (repbc) to[in=-90,out=90] (r-c-top);
    \end{tikzpicture}
  \end{align*}
  Next, use associativity of $\rho$, the action-coherence condition, and naturality of $\rho$ again to see that this is
  \begin{align*}
    \begin{tikzpicture}[baseline=60, smallstring]
      \node (r-a-bottom) at (0,0) {$R \E a \R$};
      \node (gmm) at (3,0) {$\DF a \to b ; b \to c \FD_{\mathcal F_B}^{\mathcal A}$};
      \node (r-c-top) at (1,7) {$R \E c \R$};
      \node[draw, rectangle] (mu) at (3,1) {$\mu_{\mathcal A \E a \to b \R , \mathcal A \E b \to c \R}^{\mathcal F_B}$};
      \node[draw, rectangle] (repab) at (1,5) {$R \E \epsilon_{a \to b}^{\mathcal F_A} \id_{\mathcal F_A \E \mathcal A \E b \to c \R \R} \R$};
      \node[draw, rectangle] (ri1) at (3,3) {$\rho_{\mathcal F_A \E \mathcal A \E a \to b \R \R , \mathcal F_A \E \mathcal A \E b \to c \R \R}$};
      \node[draw, rectangle] (repbc) at (1,6) {$R \E \epsilon_{b \to c}^{\mathcal F_A} \R$};
      \node[draw, rectangle] (ri2) at (1,4) {$\rho_{a , \mathcal F_A \E \mathcal A \E a \to b \R \R \mathcal F_A \E \mathcal A \E b \to c \R \R}$};
      \node[draw, rectangle] (r1) at (2,2) {$r_{\mathcal A \E a \to b \R}$};
      \node[draw, rectangle] (r2) at (4,2) {$r_{\mathcal A \E b \to c \R}$};
      \draw (r-a-bottom) to[in=-90,out=90] (ri2.200);
      \draw (gmm) to[in=-90,out=90] (mu);
      \draw (mu.135) to[in=-90,out=90] (r1);
      \draw (r1) to[in=-90,out=90] (ri1.225);
      \draw (mu.45) to[in=-90,out=90] (r2);
      \draw (repab) to[in=-90,out=90] (repbc);
      \draw (ri1) to[in=-90,out=90] (ri2.-20);
      \draw (r2) to[in=-90, out=90] (ri1.-45);
      \draw (ri2) to[in=-90,out=90] (repab);
      \draw (repbc) to[in=-90,out=90] (r-c-top);
    \end{tikzpicture}
    =
    \begin{tikzpicture}[baseline=50, smallstring]
      \node (r-a-bottom) at (0,0) {$R \E a \R$};
      \node (gmm) at (2.5,0) {$\DF a \to b ; b \to c \FD_{\mathcal F_B}^{\mathcal A}$};
      \node (r-c-top) at (1,6) {$R \E c \R$};
      \node[draw, rectangle] (mu) at (2.5,1.85) {$R \E \mu_{\mathcal A \E a \to b \R , \mathcal A \E b \to c \R}^{\mathcal F_A} \R$};
      \node[draw, rectangle] (repab) at (1,4) {$R \E \epsilon_{a \to b}^{\mathcal F_A} \id_{\mathcal F_A \E \mathcal A \E b \to c \R \R} \R$};
      \node[draw, rectangle] (repbc) at (1,5) {$R \E \epsilon_{b \to c}^{\mathcal F_A} \R$};
      \node[draw, rectangle] (ri2) at (1,3) {$\rho_{a , \mathcal F_A \E \mathcal A \E a \to b \R \R \mathcal F_A \E \mathcal A \E b \to c \R \R}$};
      \node[draw, rectangle] (r) at (2.5,1) {$r_{\mathcal A \E a \to b \R \mathcal A \E b \to c \R}$};
      \draw (r-a-bottom) to[in=-90,out=90] (ri2.200);
      \draw (gmm) to[in=-90,out=90] (r);
      \draw (r) to[in=-90,out=90] (mu);
      \draw (repab) to[in=-90,out=90] (repbc);
      \draw (mu) to[in=-90,out=90] (ri2.-20);
      \draw (ri2) to[in=-90,out=90] (repab);
      \draw (repbc) to[in=-90,out=90] (r-c-top);
    \end{tikzpicture}
    =
    \begin{tikzpicture}[baseline=50, smallstring]
      \node (r-a-bottom) at (0,0) {$R \E a \R$};
      \node (gmm) at (2,0) {$\DF a \to b ; b \to c \FD_{\mathcal F_B}^{\mathcal A}$};
      \node (r-c-top) at (1,6) {$R \E c \R$};
      \node[draw, rectangle] (mu) at (1,3) {$R \E \id_a \mu_{\mathcal A \E a \to b \R , \mathcal A \E b \to c \R}^{\mathcal F_A} \R$};
      \node[draw, rectangle] (repab) at (1,4) {$R \E \epsilon_{a \to b}^{\mathcal F_A} \id_{\mathcal F_A \E \mathcal A \E b \to c \R \R} \R$};
      \node[draw, rectangle] (repbc) at (1,5) {$R \E \epsilon_{b \to c}^{\mathcal F_A} \R$};
      \node[draw, rectangle] (ri) at (1,2) {$\rho_{a , \mathcal F_B \E \mathcal A \E a \to b \R \mathcal A \E b \to c \R \R}$};
      \node[draw, rectangle] (r) at (2,1) {$r_{\mathcal A \E a \to b \R \mathcal A \E b \to c \R}$};
      \draw (r-a-bottom) to[in=-90,out=90] (ri.225);
      \draw (gmm) to[in=-90,out=90] (r);
      \draw (r) to[in=-90,out=90] (ri.-45);
      \draw (repab) to[in=-90,out=90] (repbc);
      \draw (ri) to[in=-90,out=90] (mu);
      \draw (mu) to[in=-90,out=90] (repab);
      \draw (repbc) to[in=-90,out=90] (r-c-top);
    \end{tikzpicture}
  \end{align*}
  Lastly, we recognize the mate of $\blank \circ_{\mathcal A} \blank$ ($R$ is a functor), use lemma \ref{unit-counit mates}, and finish with naturality of $\rho$ and $r$:
  \begin{align*}
    \begin{tikzpicture}[baseline=40, smallstring]
      \node (r-a-bottom) at (0,0) {$R \E a \R$};
      \node (gmm) at (2,0) {$\DF a \to b ; b \to c \FD_{\mathcal F_B}^{\mathcal A}$};
      \node (r-c-top) at (1,4) {$R \E c \R$};
      \node[draw, rectangle] (circ) at (1,3) {$R \E \mate \E \blank \circ_{\mathcal A} \blank \R \R$};
      \node[draw, rectangle] (ri) at (1,2) {$\rho_{a , \mathcal F_B \E \mathcal A \E a \to b \R \mathcal A \E b \to c \R \R}$};
      \node[draw, rectangle] (r) at (2,1) {$r_{\mathcal A \E a \to b \R \mathcal A \E b \to c \R}$};
      \draw (r-a-bottom) to[in=-90,out=90] (ri.225);
      \draw (gmm) to[in=-90,out=90] (r);
      \draw (r) to[in=-90,out=90] (ri.-45);
      \draw (ri) to[in=-90,out=90] (circ);
      \draw (circ) to[in=-90,out=90] (r-c-top);
    \end{tikzpicture}
    =
    \begin{tikzpicture}[baseline=35, smallstring]
      \node (rabottom) at (1,0) {$R \E a \R$};
      \node (gmm) at (3.5,0) {$\DF a \to b ; b \to c \FD_{\mathcal F_B}^{\mathcal A}$};
      \node (rctop) at (2,5) {$R \E c \R$};
      \node[draw, rectangle] (g) at (3.5, 0.85) {$r_{\mathcal A \E a \to b \R \mathcal A \E b \to c \R}$};
      \node[draw, rectangle] (rfcirc) at (2, 3) {$R \E \id_a \mathcal F_A \E \blank \circ_{\mathcal A} \blank \R \R$};
      \node[draw, rectangle] (rho) at (2, 2) {$\rho_{a , \mathcal A \E a \to b \R \mathcal A \E b \to c \R}$};
      \node[draw, rectangle] (rep) at (2, 4) {$R \E \epsilon_{a \to c}^{\mathcal F_A} \R$};
      \draw (gmm) to (g);
      \draw (rfcirc) to (rep);
      \draw (rabottom) to[in=-90, out=90] (rho.235);
      \draw (g) to[in=-90, out=90] (rho.305);
      \draw (rho) to (rfcirc);
      \draw (rep) to (rctop);
    \end{tikzpicture}
    =
    \begin{tikzpicture}[baseline=50, smallstring]
      \node (rabottom) at (1,0) {$R \E a \R$};
      \node (gmm) at (3.5,0) {$\DF a \to b ; b \to c \FD_{\mathcal F_B}^{\mathcal A}$};
      \node (rctop) at (2,5) {$R \E c \R$};
      \node[draw, rectangle] (g) at (3.5, 1) {$r_{\mathcal A \E a \to b \R \mathcal A \E b \to c \R}$};
      \node[draw, rectangle] (rfcirc) at (3.5, 1.75) {$R \E \mathcal F_A \E \blank \circ_{\mathcal A} \blank \R \R$};
      \node[draw, rectangle] (rho) at (2, 3) {$\rho_{a , \DF a \to c \FD_{\mathcal F_A}^{\mathcal A}}$};
      \node[draw, rectangle] (rep) at (2, 4) {$R \E \epsilon_{a \to c}^{\mathcal F_A} \R$};
      \draw (gmm) to (g);
      \draw (g) to (rfcirc);
      \draw (rabottom) to[in=-90, out=90] (rho.235);
      \draw (rfcirc) to[in=-90, out=90] (rho.305);
      \draw (rho) to (rep);
      \draw (rep) to (rctop);
    \end{tikzpicture}
    =
    \begin{tikzpicture}[baseline=50, smallstring]
      \node (rabottom) at (1,0) {$R \E a \R$};
      \node (gmm) at (3.5,0) {$\DF a \to b ; b \to c \FD_{\mathcal F_B}^{\mathcal A}$};
      \node (rctop) at (2,5) {$R \E c \R$};
      \node[draw, rectangle] (gcirc) at (3.5, 1) {$\mathcal F_B \E \blank \circ_{\mathcal A} \blank \R$};
      \node[draw, rectangle] (g) at (3.5, 1.75) {$r_{\mathcal A \E a \to c \R}$};
      \node[draw, rectangle] (rho) at (2, 3) {$\rho_{a , \DF a \to c \FD_{\mathcal F_A}^{\mathcal A}}$};
      \node[draw, rectangle] (rep) at (2, 4) {$R \E \epsilon_{a \to c}^{\mathcal F_A} \R$};
      \draw (gmm) to (gcirc);
      \draw (gcirc) to (g);
      \draw (rabottom) to[in=-90, out=90] (rho.235);
      \draw (g) to[in=-90, out=90] (rho.305);
      \draw (rho) to (rep);
      \draw (rep) to (rctop);
    \end{tikzpicture}
  \end{align*}
  This is exactly the mate of the right hand side via the composition lemma.
\end{proof}

\begin{lem}
  We also have that $\mathcal R$ preserves identity objects, and thus is a $\mathcal V$-functor: For all $a \in A$,
  \begin{align*}
    \begin{tikzpicture}[baseline=15, smallstring]
      \node (top) at (0,2) {$\mathcal B \E R \E a \R \to R \E a \R \R$};
      \node[draw, rectangle] (j) at (0,0) {$j_a^{\mathcal A}$};
      \node[draw, rectangle] (r) at (0,1) {$\mathcal R_{a \to a}$};
      \draw (j) to[in=-90,out=90] (r);
      \draw (r) to[in=-90,out=90] (top);
    \end{tikzpicture}
    =
    \begin{tikzpicture}[baseline=5, smallstring]
      \node (top) at (0,1) {$\mathcal B \E R \E a \R \to R \E a \R \R$};
      \node[draw, rectangle] (j) at (0,0) {$j_{R \E a \R}^{\mathcal B}$};
      \draw (j) to[in=-90,out=90] (top);
    \end{tikzpicture}
  \end{align*}

\end{lem}

\begin{proof}
  Using Lemma \ref{Composition Lemma} followed by naturality and unitality of $r$ and $\rho$, the mate of the left hand side is
  \begin{align*}
    \begin{tikzpicture}[baseline=25, smallstring]
      \node[] (Ra-bottom) at (0,-1.5) {$R \E a \R$};
      \node[] (Ra-top) at (1,3.5) {$R \E a \R$};
      \node[draw,rectangle] (Gj) at (2,-0.5) {$\mathcal F_B \E j_a^{\mathcal A} \R$};
      \node[draw,rectangle] (gamma) at (2,0.5) {$r_{j_a^{\mathcal A}}$};
      \node[draw,rectangle] (rho) at (1,1.5) {$\rho_{a , \mathcal F_A \E j_a^{\mathcal A} \R}$};
      \node[draw,rectangle] (Rep) at (1,2.5) {$R \E \epsilon_{a \to a}^{\mathcal F_A} \R$};
      \draw (Ra-bottom.90) to[in=-90,out=90] (rho.-135);
      \draw (Gj.90) to[in=-90,out=90] (gamma.-90);
      \draw (gamma.90) to[in=-90,out=90] (rho.-45);
      \draw (rho.90) to[in=-90,out=90] (Rep.270);
      \draw (Rep.90) to[in=-90,out=90] (Ra-top.270);
    \end{tikzpicture}
    \commentout{\te{ by naturality and strict unitality of $s$}} =
    \begin{tikzpicture}[baseline=25, smallstring]
      \node[] (Ra-bottom) at (0,-1) {$R \E a \R$};
      \node[] (Ra-top) at (1,3) {$R \E a \R$};
      \node[draw,rectangle] (rfj) at (2,0) {$R \E \mathcal F_A \E j_a^{\mathcal A} \R \R$};
      \node[draw,rectangle] (rho) at (1,1) {$\rho_{a , \mathcal F_A \E j_a^{\mathcal A} \R}$};
      \node[draw,rectangle] (Rep) at (1,2) {$R \E \epsilon_{a \to a}^{\mathcal F_A} \R$};
      \draw (Ra-bottom.90) to[in=-90,out=90] (rho.-135);
      \draw (rfj.90) to[in=-90,out=90] (rho.-45);
      \draw (rho.90) to[in=-90,out=90] (Rep.270);
      \draw (Rep.90) to[in=-90,out=90] (Ra-top.270);
    \end{tikzpicture}
    =
    \begin{tikzpicture}[baseline=35, smallstring]
      \node[] (Ra-bottom) at (0,0) {$R \E a \R$};
      \node[] (Ra-top) at (0,3) {$R \E a \R$};
      \node[draw,rectangle] (Rep) at (0,2) {$R \E \epsilon_{a \to a}^{\mathcal F_A} \R$};
      \node[draw,rectangle] (rfj) at (0,1) {$R \E \id_a \mathcal F_A \E j_a^{\mathcal A} \R \R$};
      \draw (Ra-bottom.90) to[in=-90,out=90] (rfj);
      \draw (rfj.90) to[in=-90,out=90] (Rep.270);
      \draw (Rep.90) to[in=-90,out=90] (Ra-top.270);
    \end{tikzpicture}
    \commentout{\te{ mate of $j_a^{\mathcal A}$}} =
    \begin{tikzpicture}[baseline=20, smallstring]
      \node[] (Ra-bottom) at (1,-0) {$R \E a \R$};
      \node[] (Ra-top) at (1,2) {$R \E a \R$};
      \node[draw,rectangle] (r1) at (1,1) {$R \E \id_a \R$};
      \draw (Ra-bottom.90) to[in=-90,out=90] (r1);
      \draw (r1.90) to[in=-90,out=90] (Ra-top.270);
    \end{tikzpicture}
    =
    \mate \E j_{R \E a \R}^{\mathcal B} \R
  \end{align*}
\end{proof}

\begin{lem}
  The $\mathcal V$-functor $\mathcal R$ is strictly unital, and the laxitor $\rho^R$ is strongly unital, i.e., $\mathcal R \E 1_{\mathcal A} \R = 1_{\mathcal B}$ and $\rho_{1_{\mathcal A} , a}^R = \id_{\mathcal R \E a \R} = \rho_{a , 1_{\mathcal A}}^R$.
\end{lem}

\begin{proof}
  First, $R$ is strictly unital, so $\mathcal R \E 1_{\mathcal A} \R = R \E 1_{\mathcal A} \R = 1_{\mathcal B}$.\\
  Then $\rho$ is strongly unital, so 

  \begin{align*}
    \rho_{1_{\mathcal A} , a}^R
    &= \mate \E \rho_{1_{\mathcal A} , a} \R = \mate \E \id_{R \E a \R} \R = j_a^{\mathcal B}
  \end{align*}

  and similarly for the other side.
\end{proof}


\begin{lem}
  The laxitor $\rho^R$ is natural: For all $a , b , c , d \in A$,

  \begin{align*}
 
  \end{align*}

\end{lem}

\begin{proof}
  Let $a , b , c , d \in A$.  Using Lemma \ref{mate of fgcirc} and Lemma \ref{Composition Lemma}, the left hand side has mate
  \begin{align*}
    %
  \end{align*}
  Next, we use naturality of $r$ and $\rho$ to get
  \begin{align*}
    %
  \end{align*}
  and lastly we use lemma \ref{unit-counit mates} and the definition of $\blank \otimes_{\mathcal A} \blank$ to see that this is
  \begin{align*}  %
  \end{align*}
  On the other hand, taking the mate of the right hand side using Lemma \ref{Composition Lemma}, Lemma \ref{mate of fgtens}, and expanding the mates of $\mathcal R$, we have
  \begin{align*}
    %
  \end{align*}
  Now we use naturality and associativity of $\rho$ along with the half-braiding coherence to get
  \begin{align*}
    %
    
  \end{align*}

  Lastly, using naturality and associativity of $\rho$ again, we have

  \begin{align*}
    %
  \end{align*}

  by the action-coherence condition on $r$ and naturality of $\rho$.

\end{proof}


\begin{lem}
  The laxitor $\rho^R$ is associative: For all $a , b , c \in A$,

  \begin{align*}
    %
  \end{align*}
\end{lem}

\begin{proof}
  Taking the mate of the left hand side and applying associativity of $\rho$ and lemma \ref{mate of fgtens} twice, we get
  \begin{align*}
    %
  \end{align*}
  which is exactly the mate of the right hand side.
\end{proof}

We have proven:

\begin{prop}
  The pair $\E \mathcal R , \rho^R \R$ defined above is a 1-cell in $\vmoncat$.
\end{prop}


We've constructed a map from 1-cells of $\vmodtens$ to 1-cells of $\vmoncat$.  To show that $P$ is essentially surjective on 1-cells, it now suffices to prove the following proposition.

\begin{prop}
  For each pair $\mathcal A , \mathcal B$ of 0-cells in $\vmoncat$, the above map is an essential inverse to $P_{\mathcal A \to \mathcal B}$ on 1-cells.
\end{prop}

\begin{proof}
  Starting with a 1-cell $\E R , \rho , r \R$ in $\vmodtens \E ( A , \mathcal F_A^Z ) \to ( B , \mathcal F_B^Z ) \R$, we have
  \begin{align*}
    P_{\mathcal A \to \mathcal B} \E \E R , \rho , r \R \R \coloneqq \E \mathcal R , \rho^R \R, 
  \end{align*}
  where for $a , b \in A$, $\mathcal R \E a \R = R \E a \R$, $\rho_{a , b}^R = \mate \E \rho_{a , b} \R = \rho_{a , b}$ and $\mathcal R_{a \to b}$ is defined in \ref{V-Functor Definition}. Then $( \mathcal R , \rho^R )$ maps back to $( \boldsymbol R , \boldsymbol \rho , \boldsymbol r )$, another 1-cell in $\vmodtens$. Now we show that $( \boldsymbol R , \boldsymbol \rho , \boldsymbol r )$ and $( R , \rho , r )$ are equal as 1-cells.

  \begin{itemize}
  \item As functors. For $a \in A$, $\boldsymbol R \E a \R = \mathcal R \E a \R = R \E a \R$. For $a , b \in A$ and a morphism $f \in A \E a \to b \R = \mathcal V \E 1_{\mathcal V} \to \mathcal A \E a \to b \R \R$, we take the same mate of $f \circ \mathcal R_{a \to b}$ in two different ways. Under adjunction $\ref{Main Adjunction}$, using the definition of $\boldsymbol R$ as the underlying functor of $\mathcal R$, we have
    \begin{align*}
      \mate \E f \circ \mathcal R_{a \to b} \R = \boldsymbol R \E f \R.
    \end{align*}
    Next, we take the same mate via Lemma \ref{Composition Lemma}, making use of naturality and unitality of $r$ and $\rho$:
    \begin{align*}
      \mate \E
      \begin{tikzpicture}[baseline=15, smallstring]
        \node (top) at (0,2) {$\mathcal B \E R \E a \R \to R \E b \R \R$};
        \node[draw, rectangle] (f) at (0,0) {$f$};
        \node[draw, rectangle] (R) at (0,1) {$\mathcal R_{a \to b}$};
        \draw (f) to (R);
        \draw (R) to (top);
      \end{tikzpicture}
      \R
      =
      \begin{tikzpicture}[baseline=25, smallstring]
        \node (a) at (0,0) {$a$};
        \node (b) at (1,3) {$b$};
        \node[draw, rectangle] (f) at (2,1) {$\mathcal F_B \E f \R$};
        \node[draw, rectangle] (R) at (1,2) {$\mate \E \mathcal R_{a \to b} \R$};
        \draw (a) to[in=-90,out=90] (R.225);
        \draw (f) to[in=-90,out=90] (R.-45);
        \draw (R) to[in=-90,out=90] (b);
      \end{tikzpicture}
      =
      \begin{tikzpicture}[baseline=40, smallstring]
        \node (a) at (0,0) {$a$};
        \node (b) at (1,5) {$b$};
        \node[draw, rectangle] (f) at (2,1) {$\mathcal F_B \E f \R$};
        \node[draw, rectangle] (r) at (2,2) {$r_{\mathcal A \E a \to b \R}$};
        \node[draw, rectangle] (rho) at (1,3) {$\rho_{a, \mathcal F_A \E \mathcal A \E a \to b \R \R}$};
        \node[draw, rectangle] (epsilon) at (1,4) {$R \E \epsilon_{a , \mathcal A \E a \to b \R} \R$};
        \draw (a) to[in=-90,out=90] (rho.225);
        \draw (f) to[in=-90,out=90] (r);
        \draw (r) to[in=-90,out=90] (rho.-45);
        \draw (rho) to[in=-90,out=90] (epsilon);
        \draw (epsilon) to[in=-90,out=90] (b);
      \end{tikzpicture}
      =
      \begin{tikzpicture}[baseline=30, smallstring]
        \node (a) at (0,0) {$a$};
        \node (b) at (1,4) {$b$};
        \node[draw, rectangle] (f) at (2,1) {$R \E \mathcal F_A \E f \R \R$};
        \node[draw, rectangle] (rho) at (1,2) {$\rho_{a, \mathcal F_A \E \mathcal A \E a \to b \R \R}$};
        \node[draw, rectangle] (epsilon) at (1,3) {$R \E \epsilon_{a , \mathcal A \E a \to b \R} \R$};
        \draw (a) to[in=-90,out=90] (rho.225);
        \draw (f) to[in=-90,out=90] (rho.-45);
        \draw (rho) to[in=-90,out=90] (epsilon);
        \draw (epsilon) to[in=-90,out=90] (b);
      \end{tikzpicture}
      =
      \begin{tikzpicture}[baseline=25, smallstring]
        \node (a) at (0,0) {$a$};
        \node (b) at (0,3) {$b$};
        \node[draw, rectangle] (f) at (0,1) {$R \E \id_a \mathcal F_A \E f \R \R$};
        \node[draw, rectangle] (epsilon) at (0,2) {$R \E \epsilon_{a , \mathcal A \E a \to b \R} \R$};
        \draw (a) to[in=-90,out=90] (f);
        \draw (f) to[in=-90,out=90] (epsilon);
        \draw (epsilon) to[in=-90,out=90] (b);
      \end{tikzpicture}
    \end{align*}
    By Lemma \ref{unit-counit mates} this is equal on the nose to $R \E f \R$.
  \item As lax monoidal functors: for $a , b \in A$, $\boldsymbol \rho_{a , b} = \rho_{a , b}^R = \rho_{a , b}$.
  \item As 1-cells in $\vmodtens$: for $v \in \mathcal V$, we compute
    \begin{align*}
      \boldsymbol r_v &= \mate \E
      \begin{tikzpicture}[baseline=35, smallstring]
        \node (top) at (0,3) {$\mathcal B \E 1_{\mathcal B} \to R \E \mathcal F_A \E v \R \R \R$};
        \node (bot) at (0,0) {$v$};
        \node[draw, rectangle] (eta) at (0,1) {$\eta_v^{\mathcal F_A}$};
        \node[draw, rectangle] (R) at (0,2) {$\mathcal R_{1_{\mathcal A} \to \mathcal F_A \E v \R}$};
        \draw (bot) to (eta);
        \draw (eta) to (R);
        \draw (R) to (top);
      \end{tikzpicture} \R
      = \mate \E
      \begin{tikzpicture}[baseline=55, smallstring]
        \node (top) at (1,6) {$\mathcal B \E 1_{\mathcal B} \to R \E \mathcal F_A \E v \R \R \R$};
        \node (bot) at (1,0) {$v$};
        \node[draw, rectangle] (eta) at (1,1) {$\eta_v^{\mathcal F_A}$};
        \node[draw, rectangle] (r) at (1,3) {$r_{\mathcal A \E a \to b \R}$};
        \node[draw, rectangle] (epsilon) at (1,4) {$R \E \epsilon_{1_{\mathcal A} \to \mathcal F_A \E v \R}^{\mathcal F_A} \R$};
        \node (empty) at (1,2) {};
        \node (empty2) at (1,5) {};
        \node (mate) at (-1.5,3.5) {$\mate$};
        \draw ($(empty) + (-2,0)$) rectangle ($(empty2) + (2,0)$);
        \draw (bot) to[in=-90,out=90] (eta);
        \draw (eta) to[in=-90,out=90] (empty);
        \draw (empty) to[in=-90,out=90] (r);
        \draw (r) to[in=-90,out=90] (epsilon);
        \draw (epsilon) to[in=-90,out=90] (empty2);
        \draw (empty2) to[in=-90,out=90] (top);
      \end{tikzpicture} \R
      =
      \begin{tikzpicture}[baseline=40, smallstring]
        \node (top) at (0,4) {$R \E \mathcal F_A \E v \R \R$};
        \node (bot) at (0,0) {$\mathcal F_B \E v \R$};
        \node[draw, rectangle] (eta) at (0,1) {$\mathcal F_B \E \eta_v^{\mathcal F_A} \R$};
        \node[draw, rectangle] (r) at (0,2) {$r_{\mathcal A \E a \to b \R}$};
        \node[draw, rectangle] (epsilon) at (0,3) {$R \E \epsilon_{1_{\mathcal A} \to \mathcal F_A \E v \R}^{\mathcal F_A} \R$};
        \draw (bot) to[in=-90,out=90] (eta);
        \draw (eta) to[in=-90,out=90] (r);
        \draw (r) to[in=-90,out=90] (epsilon);
        \draw (epsilon) to[in=-90,out=90] (top);
      \end{tikzpicture}\\
    \end{align*}
    by unitality of $\rho$ and lemma \ref{Composition Lemma}.  We now show that this is equal to $r_v$ using Lemma \ref{unit-counit mates} with $f = \eta_v^{\mathcal F_A}$ and naturality of $r$:
    \begin{align*}
      \begin{tikzpicture}[baseline=20, smallstring]
        \node (bot) at (0,0) {$\mathcal F_B \E v \R$};
        \node (top) at (0,2) {$R \E \mathcal F_A \E v \R \R$};
        \node[draw, rectangle] (r) at (0,1) {$r_v$};
        \draw (bot) to (r);
        \draw (r) to (top);
      \end{tikzpicture}
      =
      \begin{tikzpicture}[baseline=45, smallstring]
        \node (bot) at (0,0) {$\mathcal F_B \E v \R$};
        \node (top) at (0,4) {$R \E \mathcal F_A \E v \R \R$};
        \node[draw, rectangle] (r) at (0,1) {$r_v$};
        \node[draw, rectangle] (eta) at (0,2) {$R \E \mathcal F_A \E \eta_v^{\mathcal F_A} \R \R$};
        \node[draw, rectangle] (epsilon) at (0,3) {$R \E \epsilon_{1_{\mathcal A} \to \mathcal F_A \E v \R} \R$};
        \draw (bot) to (r);
        \draw (r) to (eta);
        \draw (eta) to (epsilon);
        \draw (epsilon) to (top);
      \end{tikzpicture}
      =
      \begin{tikzpicture}[baseline=45, smallstring]
        \node (bot) at (0,0) {$\mathcal F_B \E v \R$};
        \node (top) at (0,4) {$R \E \mathcal F_A \E v \R \R$};
        \node[draw, rectangle] (eta) at (0,1) {$\mathcal F_B \E \eta_v^{\mathcal F_A} \R$};
        \node[draw, rectangle] (r) at (0,2) {$r_{\mathcal A \E 1_{\mathcal A} \to \mathcal F_A \E v \R \R}$};
        \node[draw, rectangle] (epsilon) at (0,3) {$R \E \epsilon_{1_{\mathcal A} \to \mathcal F_A \E v \R} \R$};
        \draw (bot) to (eta);
        \draw (eta) to (r);
        \draw (r) to (epsilon);
        \draw (epsilon) to (top);
      \end{tikzpicture}
    \end{align*}
    
  \end{itemize}
  Therefore composing $P_{A \to B}$ with this inverse map is the same as taking mates twice.
\end{proof}
\begin{defn}
  Given a monoidal natural transformation $\Theta : \E R , \rho , r \R \Rightarrow \E S , \sigma , s \R$, where $\E R , \rho , r \R$ and $\E S , \sigma , s \R$ are 1-cells in $\vmodtens \E \E A , \mathcal F_A \R \to \E B , \mathcal F_B \R \R$, we define a $\mathcal V$-monoidal natural transformation $\theta : \E \mathcal R , \rho^R \R \Rightarrow \E \mathcal S , \sigma^S \R$ via $\theta_a \coloneqq \mate \E \Theta_a \R = \Theta_a$ under the identity adjunction
  \begin{align*}
    \mathcal V \E 1_{\mathcal V} \to \mathcal B \E R \E a \R \to S \E a \R \R \R = B \E R \E a \R \to S \E a \R \R.
  \end{align*}
\end{defn}

\begin{prop}
  As defined, $\theta$ is a $\mathcal V$-monoidal natural transformation.
\end{prop}

\begin{proof}
  First, naturality: We need to show that

  \begin{align*}
    \begin{tikzpicture}[baseline=35, smallstring]
      \node (bot) at (0,0) {$\mathcal A \E a \to b \R$};
      \node (top) at (1,3) {$\mathcal B \E R \E a \R \to S \E b \R \R$};
      \node[draw, rectangle] (r) at (0,1) {$\mathcal R_{a \to b}$};
      \node[draw, rectangle] (theta) at (2,1) {$\theta_b$};
      \node[draw, rectangle] (circ) at (1,2) {$\blank \circ_{\mathcal B} \blank$};
      \draw (bot) to (r);
      \draw (r) to[in=-90,out=90] (circ.225);
      \draw (theta) to[in=-90,out=90] (circ.-45);
      \draw (circ) to (top);
    \end{tikzpicture}
    =
    \begin{tikzpicture}[baseline=35, smallstring]
      \node (bot) at (2,0) {$\mathcal A \E a \to b \R$};
      \node (top) at (1,3) {$\mathcal B \E R \E a \R \to S \E b \R \R$};
      \node[draw, rectangle] (r) at (2,1) {$\mathcal S_{a \to b}$};
      \node[draw, rectangle] (theta) at (0,1) {$\theta_a$};
      \node[draw, rectangle] (circ) at (1,2) {$\blank \circ_{\mathcal B} \blank$};
      \draw (bot) to (r);
      \draw (r) to[in=-90,out=90] (circ.-45);
      \draw (theta) to[in=-90,out=90] (circ.225);
      \draw (circ) to (top);
    \end{tikzpicture}    
  \end{align*}

  By lemma \ref{mate of fgcirc} and expanding, the mate of the left hand side is 
  \begin{align*}
    \begin{tikzpicture}[baseline=40, smallstring]
      \node (ra) at (0,0) {$R \E a \R$};
      \node (bot) at (2,0) {$\mathcal F_B \E \mathcal A \E a \to b \R \R$};
      \node (top) at (1,4) {$S \E b \R$};
      \node[draw, rectangle] (r) at (1,2) {$\mate \E \mathcal R_{a \to b} \R$};
      \node[draw, rectangle] (theta) at (1,3) {$\mate \E \theta_b \R$};
      \draw (ra) to[in=-90,out=90] (r.225);
      \draw (bot) to[in=-90,out=90] (r.-45);
      \draw (r) to[in=-90,out=90] (theta);
      \draw (theta) to (top);
    \end{tikzpicture}
    =
    \begin{tikzpicture}[baseline=45, smallstring]
      \node (ra) at (0,0) {$R \E a \R$};
      \node (bot) at (2,0) {$\mathcal F_B \E \mathcal A \E a \to b \R \R$};
      \node (top) at (1,5) {$S \E b \R$};
      \node[draw, rectangle] (r) at (2,1) {$r_{\mathcal A \E a \to b \R}$};
      \node[draw, rectangle] (rho) at (1,2) {$\rho_{a , \mathcal F_A \E \mathcal A \E a \to b \R \R}$};
      \node[draw, rectangle] (rep) at (1,3) {$R \E \epsilon_{a \to b}^{\mathcal F_A} \R$};
      \node[draw, rectangle] (theta) at (1,4) {$\Theta_b$};
      \draw (ra) to[in=-90,out=90] (rho.225);
      \draw (bot) to[in=-90,out=90] (r);
      \draw (r) to[in=-90,out=90] (rho.-45);
      \draw (rho) to[in=-90,out=90] (rep);
      \draw (rep) to (theta);
      \draw (theta) to (top);
    \end{tikzpicture}
    =
    \begin{tikzpicture}[baseline=55, smallstring]
      \node (ra) at (0,0) {$R \E a \R$};
      \node (bot) at (2,0) {$\mathcal F_B \E \mathcal A \E a \to b \R \R$};
      \node (top) at (1,5) {$S \E b \R$};
      \node[draw, rectangle] (r) at (2,1) {$r_{\mathcal A \E a \to b \R}$};
      \node[draw, rectangle] (eta) at (0,2) {$\Theta_a$};
      \node[draw, rectangle] (etf) at (2,2) {$\Theta_{\mathcal F_A \E \mathcal A \E a \to b \R \R}$};
      \node[draw, rectangle] (sigma) at (1,3) {$\sigma_{a , \mathcal F_A \E \mathcal A \E a \to b \R \R}$};
      \node[draw, rectangle] (sep) at (1,4) {$S \E \epsilon_{a \to b}^{\mathcal F_A} \R$};
      \draw (ra) to[in=-90,out=90] (eta);
      \draw (bot) to[in=-90,out=90] (r);
      \draw (r) to (etf);
      \draw (etf) to[in=-90,out=90] (sigma.-45);
      \draw (eta) to[in=-90,out=90] (sigma.225);
      \draw (sigma) to[in=-90,out=90] (sep);
      \draw (sep) to (top);
    \end{tikzpicture}
  \end{align*}
  where the last equality uses naturality and monoidality of $\Theta$. Now appealing to the coherence on $r$ and $s$ and recognizing the mate of $\mathcal S_{a \to b}$, this becomes
  \begin{align*}
    \begin{tikzpicture}[baseline=40, smallstring]
      \node (ra) at (0,0) {$R \E a \R$};
      \node (bot) at (2,0) {$\mathcal F_B \E \mathcal A \E a \to b \R \R$};
      \node (top) at (1,4) {$S \E b \R$};
      \node[draw, rectangle] (s) at (2,1) {$s_{\mathcal A \E a \to b \R}$};
      \node[draw, rectangle] (eta) at (0,1) {$\Theta_a$};
      \node[draw, rectangle] (sigma) at (1,2) {$\sigma_{a , \mathcal F_A \E \mathcal A \E a \to b \R \R}$};
      \node[draw, rectangle] (sep) at (1,3) {$S \E \epsilon_{a \to b}^{\mathcal F_A} \R$};
      \draw (ra) to[in=-90,out=90] (eta);
      \draw (bot) to[in=-90,out=90] (s);
      \draw (s) to[in=-90,out=90] (sigma.-45);
      \draw (eta) to[in=-90,out=90] (sigma.225);
      \draw (sigma) to[in=-90,out=90] (sep);
      \draw (sep) to (top);
    \end{tikzpicture}
    =
    \begin{tikzpicture}[baseline=30, smallstring]
      \node (ra) at (0,0) {$R \E a \R$};
      \node (bot) at (2,0) {$\mathcal F_B \E \mathcal A \E a \to b \R \R$};
      \node (top) at (1,3) {$S \E b \R$};
      \node[draw, rectangle] (eta) at (0,1) {$\Theta_a$};
      \node[draw, rectangle] (sab) at (1,2) {$\mate \E \mathcal S_{a \to b} \R$};
      \draw (ra) to[in=-90,out=90] (eta);
      \draw (bot) to[in=-90,out=90] (sab.-45);
      \draw (eta) to[in=-90,out=90] (sab.225);
      \draw (sab) to (top);
    \end{tikzpicture}
    =
    \begin{tikzpicture}[baseline=30, smallstring]
      \node (ra) at (0,0) {$R \E a \R$};
      \node (bot) at (2,0) {$\mathcal F_B \E \mathcal A \E a \to b \R \R$};
      \node (top) at (1,3) {$S \E b \R$};
      \node[draw, rectangle] (eta) at (0,1) {$\mate \E \theta_a \R$};
      \node[draw, rectangle] (sab) at (1,2) {$\mate \E \mathcal S_{a \to b} \R$};
      \draw (ra) to[in=-90,out=90] (eta);
      \draw (bot) to[in=-90,out=90] (sab.-45);
      \draw (eta) to[in=-90,out=90] (sab.225);
      \draw (sab) to (top);
    \end{tikzpicture}
  \end{align*}
  which is exactly the mate of the right hand side. To show monoidality, we need to check that

  \begin{align*}
    \begin{tikzpicture}[baseline=15, smallstring]
      \node (top) at (1,2) {$\mathcal B \E R \E a \R R \E b \R \to S \E a b \R \R$};
      \node[draw, rectangle] (sigma) at (0,0) {$\rho_{a , b}^R$};
      \node[draw, rectangle] (theta) at (2,0) {$\theta_{a b}$};
      \node[draw, rectangle] (circ) at (1,1) {$\blank \circ_{\mathcal B} \blank$};
      \draw (sigma) to[in=-90,out=90] (circ.225);
      \draw (theta) to[in=-90,out=90] (circ.-45);
      \draw (circ) to (top);
    \end{tikzpicture}
    =
    \begin{tikzpicture}[baseline=25, smallstring]
      \node (top) at (2,3) {$\mathcal B \E R \E a \R R \E b \R \to S \E a b \R \R$};
      \node[draw, rectangle] (eta) at (0,0) {$\theta_a$};
      \node[draw, rectangle] (etb) at (2,0) {$\theta_b$};
      \node[draw, rectangle] (tens) at (1,1) {$\blank \otimes_{\mathcal B} \blank$};
      \node[draw, rectangle] (sigma) at (3,1) {$\sigma_{a , b}^S$};
      \node[draw, rectangle] (circ) at (2,2) {$\blank \circ_{\mathcal B} \blank$};
      \draw (eta) to[in=-90,out=90] (tens.225);
      \draw (etb) to[in=-90,out=90] (tens.-45);
      \draw (tens) to[in=-90,out=90] (circ.225);
      \draw (sigma) to[in=-90,out=90] (circ.-45);
      \draw (circ) to (top);
    \end{tikzpicture}
  \end{align*}

  The mate of the left hand side, after using lemma \ref{mate of fgcirc}, monoidality of $\Theta$, and lemma \ref{mate of fgtens}, becomes 
  \begin{align*}
    \begin{tikzpicture}[baseline=30, smallstring]
      \node (bot) at (0,0) {$R \E a \R R \E b \R$};
      \node (top) at (0,3) {$S \E a b \R$};
      \node[draw, rectangle] (sigma) at (0,1) {$\mate \E \rho_{a , b}^R \R$};
      \node[draw, rectangle] (theta) at (0,2) {$\mate \E \theta_{a b} \R$};
      \draw (bot) to[in=-90,out=90] (sigma);
      \draw (sigma) to (theta);
      \draw (theta) to (top);
    \end{tikzpicture}
    =
    \begin{tikzpicture}[baseline=30, smallstring]
      \node (bot) at (0,0) {$R \E a \R R \E b \R$};
      \node (top) at (0,3) {$S \E a b \R$};
      \node[draw, rectangle] (rho) at (0,1) {$\rho_{a , b}$};
      \node[draw, rectangle] (theta) at (0,2) {$\Theta_{a b}$};
      \draw (bot) to[in=-90,out=90] (rho);
      \draw (rho) to (theta);
      \draw (theta) to (top);
    \end{tikzpicture}
    =
    \begin{tikzpicture}[baseline=25, smallstring]
      \node (ra) at (0,0) {$R \E a \R$};
      \node (rb) at (2,0) {$R \E b \R$};
      \node (top) at (1,3) {$S \E a b \R$};
      \node[draw, rectangle] (eta) at (0,1) {$\Theta_a$};
      \node[draw, rectangle] (etb) at (2,1) {$\Theta_b$};
      \node[draw, rectangle] (sigma) at (1,2) {$\sigma_{a , b}$};y
      \draw (ra) to (eta);
      \draw (rb) to (etb);
      \draw (eta) to[in=-90,out=90] (sigma.225);
      \draw (etb) to[in=-90,out=90] (sigma.-45);
      \draw (sigma) to (top);
    \end{tikzpicture}
    =
    \begin{tikzpicture}[baseline=45, smallstring]
      \node (bot) at (3,0) {$R \E a \R R \E b \R$};
      \node (top) at (3,4.5) {$S \E a b \R$};
      \node[draw, rectangle] (eta) at (2,1) {$\theta_a$};
      \node[draw, rectangle] (etb) at (4,1) {$\theta_b$};
      \node[draw, rectangle] (tens) at (3,2) {$\blank \otimes_{\mathcal B} \blank$};
      \node (empty) at (3,1) {};
      \node (empty2) at (3,2.5) {};
      \draw ($(empty)+(-1.5,-0.5)$) rectangle ($(empty2)+(1.5,0)$);
      \node at (1,1.5) {$\mate$};
      \node[draw, rectangle] (sigma) at (3,3.5) {$\mate \E \sigma_{a , b}^S \R$};
      \draw (eta) to[in=-90,out=90] (tens.225);
      \draw (etb) to[in=-90,out=90] (tens.-45);
      \draw (tens) to[in=-90,out=90] (empty2);
      \draw (bot) to ($(empty)+(0,-0.5)$);
      \draw (empty2) to (sigma);
      \draw (sigma) to (top);
    \end{tikzpicture}
  \end{align*}
  which is exactly the mate of the right hand side under lemma \ref{mate of fgcirc}.
\end{proof}

Thus we have proven

\begin{prop}
  The 2-functor $P$ is fully faithful on 2-cells.
\end{prop}

In turn, this finishes the proof of our main theorem:

\begin{thm} \label{Main Theorem}
  The 2-functor $P : \vmoncat \to \vmodtens$ is a 2-equivalence.
\end{thm}

\begin{remark}
  If we instead demand that all of our lax functors are in fact strong, i.e., we restrict to the subcategories where 1-cells are only the strong monoidal and strong $\mathcal V$-monoidal functors, then we get a 2-equivalence by the same proof.  Showing that a lax monoidal functor is in fact strong is simply a check that the laxitors are invertible, and this translates directly through $P$, since $P$ is the identity on each of the laxitors.
\end{remark}



\section{$G$-gradings} \label{G-gradings}

For the remainder of the article, fix a finite group $G$ and assume that $\mathcal V$ is linear.


\begin{defn}
  We say that a $\mathcal V$-category $\mathscr C$ is \emph{additive} if its underlying category is additive.  We say a $\mathcal V$-functor between additive $\mathcal V$-categories is additive if its underlying functor is additive.
\end{defn}

\begin{defn}
  Let $\mathcal A$ and $\mathcal B$ be $\mathcal V$-categories.  We define the \emph{external direct sum} $\mathcal A \oplus \mathcal B$ to be the $\mathcal V$-category with objects $a \oplus b$ for $a \in \mathcal A$ and $b \in \mathcal B$, and hom objects given by $\E \mathcal A \oplus \mathcal B \R \E a_1 \oplus b_1 \to a_2 \oplus b_2 \R \coloneqq \mathcal A \E a_1 \to a_2 \R \oplus \mathcal B \E b_1 \to b_2 \R$, using the direct sum on objects of $\mathcal V$.
\end{defn}

\begin{defn}
  A \emph{$G$-graded $\mathcal V$-monoidal category} is a $\mathcal V$-monoidal category $\mathscr C$ along with a decomposition $\mathscr C = \bigoplus_{g \in G} \mathscr C_g$, such that for each $c_g \in \mathscr C_g$ and $c_h \in \mathscr C_h$, we have $c_g c_h \in \mathscr C_{gh}$.  We say that the $G$-grading is \emph{faithful} if for each $g \in G$ there exist $a_g , b_g \in \mathscr C_g$ such that $\mathscr C_g ( a_g \to b_g ) \neq 0_{\mathcal V}$. If $\mathscr C$ and $\mathscr D$ are $G$-graded $\mathcal V$-monoidal categories, a $\mathcal V$-monoidal functor $( \mathcal F , \nu ) : \mathscr C \to \mathscr D$ is called $G$-graded if the underlying functor $F$ is $G$-graded, i.e., $F$ is additive and for every $g \in G$ and every $c_g \in C_g$ we have $F ( c_g ) \in \mathscr D_g$. We say that a $G$-graded $\mathcal V$-monoidal functor $( \mathcal F , \nu )$ is an \emph{equivalence} if it is a $\mathcal V$-monoidal equivalence of $\mathcal V$-monoidal categories.  Note that in this case, $\mathcal F \restrict_{\mathscr C_g} : \mathscr C_g \to \mathscr D_g$ is a $\mathcal V$-equivalence for all $g \in G$, and is a $\mathcal V$-monoidal equivalence when $g = e$.
\end{defn}

\begin{defn} (\cite{1910.03178})
  A $G$-graded $\mathcal V$-module tensor category consists of a linear monoidal category $C$, a decomposition $C = \bigoplus_{g \in G} C_g$ into linear categories $C_g$, and a braided strong monoidal functor $\mathcal F_C : \mathcal V \to Z ( C )$ such that $\mathcal F_C^Z ( \mathcal V ) \subseteq \Rep ( G )' \subseteq Z ( C )$ (equivalently, $\mathcal F_C ( \mathcal V ) \subseteq C_e$ \cite[Lemma 2.11]{1910.03178}). Given $( C, \mathcal F_C )$ and $( D , \mathcal F_D )$ $G$-graded tensored $\mathcal V$-module tensor categories, a morphism $( R , \rho , r ) : C \to D$ is called $G$-graded if $R$ is additive and for $g \in G$ and $c_g \in C_g$, $R ( c_g ) \in D_g$.  We say that $( R , \rho , r )$ is an \emph{equivalence} if $( R , \rho )$ is a strong monoidal equivalence and $r$ is a natural isomorphism.
\end{defn}


\subsection{$G$-graded correspondence} \label{G-graded correspondence}

On the way to a $G$-graded version of Theorem \ref{Main theorem}, we define $G$-graded analogues to $\vmoncat$ and $\vmodtens$ and construct a correspondence on 0-cells.

\begin{defn}
  Define the 2-category $\ggrvmoncat$ as follows:
  \begin{itemize}
  \item 0-cells are $G$-graded (weakly) tensored $\mathcal V$-monoidal categories
  \item 1-cells are $G$-graded $\mathcal V$-monoidal functors
  \item 2-cells are $\mathcal V$-monoidal natural transformations.    
  \end{itemize}
  The same proofs as in Section \ref{Define bicategories} show that $\ggrvmoncat$ is again a 2-category under the same compositions.  Similarly, we define the 2-category $\ggrvmodtens$ via
  \begin{itemize}
  \item 0-cells are $G$-graded (weakly) tensored $\mathcal V$-module tensor categories
  \item 1-cells are $G$-graded monoidal functors with action-coherence natural transformations
  \item 2-cells are monoidal natural transformations satisfying the coherence in \ref{Define bicategories}    
  \end{itemize}
\end{defn}

Working towards extending Theorem \ref{Main theorem}, we first extend Theorem \ref{Main MP Theorem} for the 0-cells:
\begin{thm} \label{thm:G-graded correspondence}
  There is a bijective correspondence between equivalence classes:
  \[
  \left\{\,
  \parbox{7cm}{\rm Faithfully $G$-graded (weakly) tensored rigid $\mathcal V$-monoidal categories
    $\mathscr C = \bigoplus_{g \in G} \mathscr C_g$
  }\,\right\}
  \,\,\cong\,\,
  \left\{\,\parbox{7.5cm}{\rm
    Faithfully $G$-graded
    (weakly) tensored rigid $\mathcal V$-module tensor categories
    $(C = \bigoplus_{g \in G} C_g,\mathcal F_C^{\scriptscriptstyle Z})$
  }\,\right\}.
  \]
\end{thm}
\begin{proof}
  In the following, we revert to notation of \cite{MR3961709} for clarity, with the underlying category of $\mathscr C$ denoted by $\mathscr C^{\mathcal V}$, and the $\mathcal V$-monoidal category associated to a $\mathcal V$-module tensor category $\E C , \mathcal F_C \R$ being denoted by $C_{\sslash \mathcal F}$. We use this to check that our abuse of notation is valid up to equivalence - here a $G$-graded monoidal equivalence.  We will continue to use ordinary font in subscripts and superscripts to specify relation to the underlying category, e.g., $C$ refers to $\mathscr C^{\mathcal V}$. We first show that the correspondence in Theorem \ref{Main MP Theorem} can be extended to preserve $G$-gradings.
  
  Suppose we start with a (weakly) tensored $G$-graded $\mathcal V$-monoidal category $\mathscr C = \bigoplus_{g \in G} \mathscr C_G$. We show first that $\mathcal F_C \E \mathcal V \R \subseteq \mathscr C_e^{\mathcal V}$. Fix $b \in \mathscr C_g$ for some nonidentity $g \in G$; we know $\mathscr C \E 1_{\mathscr C} \to b \R \cong 0_{\mathcal V}$ since $\mathscr C$ is $G$-graded (and $1_{\mathscr C} \in \mathscr C_e$), so for all $v \in V$,
  \begin{align*}
    0_{\mathrm{\Vec}} \cong \mathcal V \E v \to \mathscr C \E 1_{\mathscr C} \to b \R \R \cong \mathscr C^{\mathcal V} \E \mathcal F_C \E v \R \to b \R.
  \end{align*}
  Thus $\mathscr C^{\mathcal V} \E \mathcal F_C \E v \R \to b \R \cong 0_{\mathrm{\Vec}}$ for all $b$ not in $\mathscr C^{\mathcal V}_e$, so $\mathcal F_C \E \mathcal V \R \subseteq \mathscr C_e$. We also have, with all direct sums over $g \in G$,
  \begin{align*}
    \mathscr C^{\mathcal V} \E \bigoplus a_g \to \bigoplus b_g \R%
    &= \mathcal V \E 1_{\mathcal V} \to \mathscr C \E \bigoplus a_g \to \bigoplus b_g \R \R\\
    &= \mathcal V \E 1_{\mathcal V} \to \bigoplus \mathscr C_g \E a_g \to b_g \R \R\\
    &\cong \bigoplus \mathcal V \E 1_{\mathcal V} \to \mathscr C_g \E a_g \to b_g \R \R\\
    &= \bigoplus \mathscr C_g^{\mathcal V} \E a_g \to b_g \R,
  \end{align*}
  so $\mathscr C^{\mathcal V}$ is equivalent to $\oplus_g \mathscr C_g^{\mathcal V}$, and thus $( \mathscr C^{\mathcal V} , \mathcal F_C )$ is a (weakly) tensored $G$-graded $\mathcal V$-module tensor category.

  Now suppose we have a $G$-graded $\mathcal V$-module tensor category $( C = \bigoplus C_g , \mathcal F_C )$. Then we get a $\mathcal V$-monoidal category $C_{\sslash \mathcal F}$ by \ref{Map on 0-cells}.  Further, $C \cong \bigoplus C_g$ as $\mathcal V$-module categories with the actions $( v , c_g ) \mapsto c_g \mathcal F_C ( v )$; this is valid because $\mathcal F_C \E v \R \subseteq C_e$, so $c_g \mathcal F_C \E v \R$ is again in $C_g$. Since every element of $\mathscr C$ can be written as a direct sum of objects from the $\mathscr C_g$, we have the following natural isomorphisms:
  \begin{align*}
    \mathcal V \E v \to \mathscr C \E \bigoplus a_g \to \bigoplus b_g \R \R%
    &\cong \mathscr C^{\mathcal V} \E \E \bigoplus a_g \R \mathcal F_C \E v \R \to \bigoplus b_g \R\\
    &\cong \mathscr C^{\mathcal V} \E \E \bigoplus a_g \mathcal F_C \E v \R \R \to \bigoplus b_g \R\\
    &\cong \bigoplus \mathscr C_g^{\mathcal V} \E \E a_g \mathcal F_C \E v \R \R \to b_g \R\\
    &\cong \bigoplus \mathcal V \E v \to \mathscr C_g \E a_g \to b_g \R \R\\
    &\cong \mathcal V \E v \to \bigoplus \mathscr C_g \E a_g \to b_g \R \R\\
    &= \mathcal V \E v \to \E \bigoplus \mathscr C_g \R \E \bigoplus a_g \to \bigoplus b_g \R \R.
  \end{align*}
  Now by the Yoneda Lemma we see that $\mathscr C$ is equivalent to $\bigoplus \mathscr C_g$ as $\mathcal V$-monoidal categories, so $\mathscr C$ is a $G$-graded $\mathcal V$-monoidal category.
  
  To finish the proof, we see that in \cite[sections 6 and 7]{MR3961709} the authors construct a strong monoidal equivalence of categories $C \cong C_{\sslash \mathcal F}^{\mathcal V}$, and extend it to an equivalence of $\mathcal V$-module tensor categories. They analogously construct an equivalence of the $\mathcal V$-monoidal categories $\mathscr C \cong \mathscr C_{\sslash \mathcal F}^{\mathcal V}$. Both equivalences are the identity on objects, so both are $G$-graded.
\end{proof}
\begin{thm} \label{G-graded equivalence}
  Theorem \ref{Main theorem} extends to an equivalence of 2-categories between $\ggrvmoncat$ and $\ggrvmodtens$.
\end{thm}
\begin{proof}
  We construct a 2-functor $P_G : \ggrvmoncat \to \ggrvmodtens$ in the same way as in Section \ref{Define 2-functor}.  The 0-cells in the image of $P_G$ are $G$-graded by the above, and the 1-cells are $G$-graded since $P_G$ doesn't change how 1-cells act on objects; 2-cells are unchanged from the non-$G$-graded case. Since $\ggrvmoncat$ and $\ggrvmodtens$ have the same composition as $\vmoncat$ and $\vmodtens$, respectively, it follows that $P_G$ is a 2-functor.

  To check that $P_G$ is a 2-equivalence, we again just need to check that it is essentially surjective on 0-cells, essentially surjective on 1-cells, and fully faithful on 2-cells.  Essential surjectivity on 0-cells is exactly Theorem \ref{thm:G-graded correspondence}.

  Next, let $\mathscr C$ and $\mathscr D$ be 0-cells in $\ggrvmoncat$ and consider a 1-cell $\E \mathcal R , \rho^R \R : \mathscr C \to \mathscr D$ in $\vmoncat$. Since $\mathcal R \E a \R = P \E \mathcal R \R \E a \R$ for all objects $a \in \mathscr C$, it follows that  $\mathcal R$ is $G$-graded if and only if $P \E \mathcal R \R$ is $G$-graded. Thus the same construction as for 1-cells in $\vmoncat$ and $\vmodtens$ gives us a map from 1-cells in $\ggrvmoncat$ to 1-cells in $\ggrvmodtens$. Since $P$ is essentially surjective on 1-cells, it now follows that $P_G$ is as well.

  Given two 1-cells $\mathcal R , \mathcal S$ between $G$-graded $\mathcal V$-monoidal categories, note that $\mathcal R$ and $\mathcal S$ are still 1-cells in $\vmoncat$ as well. Since 2-cells don't have any inherent $G$-grading, we have $\ggrvmoncat \E \mathcal R \to \mathcal S \R = \vmoncat \E \mathcal R \to \mathcal S \R$. This also holds for $\vmodtens$ and $\ggrvmodtens$, so $P_G$ is fully faithful on 2-cells as before.
\end{proof}


\section{$G$-extensions} \label{G-extensions}

Now fix a finite group $G$, a tensored $\mathcal V$-monoidal category $\mathcal A$ and its corresponding tensored $\mathcal V$-module tensor category $( A , \mathcal F_A^Z )$. For this section assume that $\mathcal V$ is linear and $\mathcal F_A^Z$ is strong monoidal.

In this section we extend Theorem \ref{G-graded equivalence} to the setting of $G$ extensions of $\mathcal A$ and $\E A , \mathcal F_A \R$. We construct $\gextvmoncat$ and $\gextvmodtens$, prove a correspondence between zero cells, and extend the correspondence to an equivalence of 2-categories.

\begin{defn}
  A \emph{$G$-graded extension} of $\mathcal A$ is a faithfully $G$-graded $\mathcal V$-monoidal category $\mathscr C = \bigoplus_{g \in G} \mathscr C_g$ equipped with a $\mathcal V$-monoidal equivalence $( \mathcal I^C , \iota^C ) : \mathcal A \to \mathscr C_e$.  For example, any faithfully $G$-graded $\mathcal V$-monoidal category $\mathscr C$, equipped with the identity on $\mathscr C_e$ is a $G$-extension of $\mathscr C_e$. If $( \mathscr C , \mathcal I^C , \iota^C )$ and $( \mathscr D , \mathcal I^D , \iota^D )$ are $G$-extensions of $\mathcal A$. We define a morphism from $\mathscr C$ to $\mathscr D$ to be a triple $\E \mathcal R , \rho , p \R$, where $\E \mathcal R , \rho \R$ is a $G$-graded strong $\mathcal V$-monoidal functor, and $p$ is a $\mathcal V$-monoidal natural transformation $p : \mathcal I^D \Rightarrow \mathcal I^C \circ \mathcal R_e$.
\end{defn}
\begin{defn}
  A \emph{$G$-graded extension} of $( A , \mathcal F_A^Z )$ is a faithfully $G$-graded rigid tensored $\mathcal V$-module tensor category $( C = \bigoplus_{g \in G} C_g , \mathcal F_C^Z )$ equipped with an equivalence $( I^C , i^C , \iota^C ) : A \to C_e$ of $\mathcal V$-module tensor categories. If $( C , \mathcal F_C^Z , I^C , i^C , \iota^C )$ and $( D , \mathcal F_D^Z , I^D , i^D , \iota^D )$ are $G$-extensions of $\E A , \mathcal F_A \R$, we define a morphism from $C$ to $D$ to be a tuple $\E R , \rho , r , p \R$, where $\E R , \rho , r \R : C \to D$ is a 1-cell in $\ggrvmodtens \E C \to D \R$, and $p : I^D \Rightarrow I^C \circ R_e$ is a monoidal natural transformation compatible with $r$: We require
  \begin{align} \label{Inclusion coherence}
    \begin{tikzpicture}[smallstring, baseline=25]
      \node (top) at (0,3) {$R_e \E I^C \E \mathcal F_A \E v \R \R \R$};
      \node (bot) at (0,0) {$\mathcal F_D \E v \R$};
      \node[draw, rectangle] (rd) at (0,1) {$i_v^D$};
      \node[draw, rectangle] (p) at (0,2) {$p_{\mathcal F_A \E v \R}$};
      \draw (bot) to (rd);
      \draw (rd) to (p);
      \draw (p) to (top);
    \end{tikzpicture}
    =
    \begin{tikzpicture}[smallstring, baseline=25]
      \node (top) at (0,3) {$R_e \E I^C \E \mathcal F_A \E v \R \R \R$};
      \node (bot) at (0,0) {$\mathcal F_D \E v \R$};
      \node[draw, rectangle] (re) at (0,1) {$r_v$};
      \node[draw, rectangle] (rc) at (0,2) {$R_e \E i_v^C \R$};
      \draw (bot) to (re);
      \draw (re) to (rc);
      \draw (rc) to (top);
    \end{tikzpicture}
  \end{align}
\end{defn}

\begin{defn}
  Define the 2-category $\gextvmoncat$ via:
\begin{itemize}
\item 0-cells are the $G$-extensions of $\mathcal A$.
\item 1-cells are the morphisms defined above.
\item 2-cell are the $1_{\mathcal V}$-graded monoidal natural transformations $\theta : \E \mathcal R , \rho , p \R \Rightarrow \E \mathcal S , \sigma , q \R : \mathscr C \to \mathscr D$ such that $\E p_a \theta_a \R \circ \E \blank \circ_{\mathscr D} \blank \R = q_a$ for each $a \in \mathscr C$
\end{itemize}

We define the 2-category $\gextvmodtens$ via:
\begin{itemize}
\item 0-cells are the $G$-extensions of $\E A , \mathcal F_A^Z \R$.
\item 1-cells are the morphisms defined above.
\item 2-cells are the monoidal natural transformations $\Theta : \E R , \rho , r , p \R \Rightarrow \E S , \sigma , s , q \R : C \to D$ such that for all $v \in \mathcal V$, $r_v \circ \Theta_{\mathcal F_C \E v \R} = s_v$, and for all $a \in C$, $p_a \circ \Theta_{I^C \E a \R} = q_a$. 
\end{itemize}

We again use the same compositions as for $\vmoncat$ and $\vmodtens$, with the slight change that for 1-cells we now need to compose the new coherences $p$ as well; we take the usual vertical composition of ($1_{\mathcal V}$-graded) monoidal natural transformations. All checks that $\gextvmoncat$ and $\gextvmodtens$ are 2-categories are analogous to the checks for $\vmoncat$ and $\vmodtens$.
\end{defn}


\begin{thm} \label{Extension correspondence}
  There is a bijective correspondence between equivalence classes
  \[
  \left\{\,
  \parbox{4cm}{$\mathcal V$-monoidal $G$-extensions \\$( \mathscr C = \bigoplus_{g \in G} \mathscr C_g , \mathcal I , \iota)$ of $\mathcal A$
  }\,\right\}
  \,\,\cong\,\,
  \left\{\,\parbox{6cm}{\rm
    $G$-extensions \\$(C = \bigoplus_{g \in G} C_g , \mathcal F_C^{\scriptscriptstyle Z} , I , i , \iota)$ of $( A , \mathcal F_A^Z )$
  }\,\right\}.
  \]
\end{thm}
\begin{proof}
  From Theorem \ref{thm:G-graded correspondence}, we have a correspondence of 0-cells $P_G : \ggrvmoncat \to \ggrvmodtens$. Given a 0-cell $\E \mathscr C , \mathcal I , \iota \R$ in $\gextvmoncat$, we get a $G$-graded $\mathcal V$-module tensor category $\E C , \mathcal F_C \R$ via the map on 0-cells induced by $P_G$. By the map on 1-cells induced by $P_G$, we also get a $\mathcal V$-monoidal equivalence $\E I , \iota , i \R$ from $\E \mathcal I , \iota \R$. Thus we have a 0-cell $\E C , \mathcal F_C , I , \iota , i \R$ in $\gextvmodtens$. Further, since both of the induced maps from $P_G$ are bijections on equivalence classes (of 0-cells and 1-cells, respectively), it follows that we get a bijection from equivalence classes of 0-cells in $\gextvmoncat$ to 0-cells of $\gextvmodtens$.
\end{proof}

\subsection{$G$-extension 2-functor}

We lift our 2-functor $P_G : \ggrvmoncat \to \ggrvmodtens$ to a 2-functor $P_G^{\mathcal A} : \gextvmoncat \to \gextvmodtens$.  On 0-cells, we begin the definition of $P_G^{\mathcal A}$ with the correspondence in Theorem \ref{Extension correspondence}: From a 0-cell $( \mathscr C , \mathcal I^C , \iota^C )$ in $\gextvmoncat$, we use Theorem \ref{Extension correspondence} to get a 0-cell in $\gextvmodtens$.

Given 0-cells $( \mathscr C , \mathcal I^C , \iota^C )$, $( \mathscr D , \mathcal I^D , \iota^D )$ in $\gextvmoncat$ and a 1-cell $\E \mathcal R , \rho , p \R \in \gextvmoncat \E \mathscr C \to \mathscr D \R$, we construct the first three components of a 1-cell, $\E R , \rho , r \R$, as in Theorem \ref{G-graded equivalence}.  Lastly, we need a coherence natural isomorphism $p$. We notice that
\begin{align*}
  p_a \in \mathcal V \E 1_{\mathcal V} \to \mathscr D \E \mathcal I_D \E a \R \to \mathcal S \E \mathcal R \E \mathcal I_B \E a \R \R \R \R \R = D \E I_D \E a \R \to S \E R \E I_B \E a \R \R \R \R,
\end{align*}
so we use the same one, i.e., we take the mate of $p$ under this identity adjunction.
\begin{lem}
  Considered together, $\E R , \rho , r , p \R$ make up a 1-cell in $\gextvmodtens$. That is, coherence \ref{Inclusion coherence} is satisfied:
  \begin{align*}
    \begin{tikzpicture}[smallstring, baseline=25]
      \node (top) at (0,3) {$R_e \E I^C \E \mathcal F_A \E v \R \R \R$};
      \node (bot) at (0,0) {$\mathcal F_D \E v \R$};
      \node[draw, rectangle] (i) at (0,1) {$i_v^D$};
      \node[draw, rectangle] (p) at (0,2) {$p_{\mathcal F_A \E v \R}$};
      \draw (bot) to (i);
      \draw (i) to (p);
      \draw (p) to (top);
    \end{tikzpicture}
    =
    \begin{tikzpicture}[smallstring, baseline=25]
      \node (top) at (0,3) {$R_e \E I^C \E \mathcal F_A \E v \R \R \R$};
      \node (bot) at (0,0) {$\mathcal F_D \E v \R$};
      \node[draw, rectangle] (rv) at (0,1) {$r_v$};
      \node[draw, rectangle] (i) at (0,2) {$R_e \E i_v^C \R$};
      \draw (bot) to (rv);
      \draw (rv) to (i);
      \draw (i) to (top);
    \end{tikzpicture}
  \end{align*}
\end{lem}

\begin{proof}
  Taking the mate of the left hand side and applying naturality and unitality of $p$, we have
  \begin{align*}
    \begin{tikzpicture}[smallstring, baseline=30]
      \node (top) at (1,3) {$\mathscr D \E 1_{\mathscr D} \to R_e \E I^C \E \mathcal F_A \E v \R \R \R \R$};
      \node (bot) at (0,0) {$v$};
      \node[draw, rectangle] (i) at (0,1) {$\mate \E i_v^D \R$};
      \node[draw, rectangle] (p) at (2,1) {$p_{\mathcal F_A \E v \R}^{\mathcal R}$};
      \node[draw, rectangle] (circ) at (1,2) {$\blank \circ \blank$};
      \draw (bot) to[in=-90,out=90] (i);
      \draw (i) to[in=-90,out=90] (circ.225);
      \draw (p) to[in=-90,out=90] (circ.-45);
      \draw (circ) to[in=-90,out=90] (top);
    \end{tikzpicture}
    =
    \begin{tikzpicture}[smallstring, baseline=30]
      \node (top) at (1,4) {$\mathscr D \E 1_{\mathscr D} \to R_e \E I^C \E \mathcal F_A \E v \R \R \R \R$};
      \node (bot) at (0,0) {$v$};
      \node[draw, rectangle] (eta) at (0,1) {$\eta_v^A$};
      \node[draw, rectangle] (I) at (0,2) {$\mathcal I_{1_{\mathcal A} \to \mathcal F_A \E v \R}^D$};
      \node[draw, rectangle] (p) at (2,2) {$p_{\mathcal F_A \E v \R}^{\mathcal R}$};
      \node[draw, rectangle] (circ) at (1,3) {$\blank \circ \blank$};
      \draw (bot) to[in=-90,out=90] (eta);
      \draw (eta) to[in=-90,out=90] (I);
      \draw (I) to[in=-90,out=90] (circ.225);
      \draw (p) to[in=-90,out=90] (circ.-45);
      \draw (circ) to[in=-90,out=90] (top);
    \end{tikzpicture}
    =
    \begin{tikzpicture}[smallstring, baseline=40]
      \node (top) at (1,5) {$\mathscr D \E 1_{\mathscr D} \to R_e \E I^C \E \mathcal F_A \E v \R \R \R \R$};
      \node (bot) at (2,0) {$v$};
      \node[draw, rectangle] (eta) at (2,1) {$\eta_v^A$};
      \node[draw, rectangle] (I) at (2,2) {$\mathcal I_{1_{\mathcal A} \to \mathcal F_A \E v \R}^C$};
      \node[draw, rectangle] (p) at (0,3) {$p_{1_{\mathscr D}}^{\mathcal R}$};
      \node[draw, rectangle] (R) at (2,3) {$\mathcal R_{1_{\mathscr C} \to I^C \E \mathcal F_A \E v \R \R}$};
      \node[draw, rectangle] (circ) at (1,4) {$\blank \circ \blank$};
      \draw (bot) to[in=-90,out=90] (eta);
      \draw (eta) to[in=-90,out=90] (I);
      \draw (I) to[in=-90,out=90] (R);
      \draw (p) to[in=-90,out=90] (circ.225);
      \draw (R) to[in=-90,out=90] (circ.-45);
      \draw (circ) to[in=-90,out=90] (top);
    \end{tikzpicture}
    =
    \begin{tikzpicture}[smallstring, baseline=30]
      \node (top) at (0,4) {$\mathscr D \E 1_{\mathscr D} \to R_e \E I^C \E \mathcal F_A \E v \R \R \R \R$};
      \node (bot) at (0,0) {$v$};
      \node[draw, rectangle] (eta) at (0,1) {$\eta_v^A$};
      \node[draw, rectangle] (I) at (0,2) {$\mathcal I_{1_{\mathcal A} \to \mathcal F_A \E v \R}^C$};
      \node[draw, rectangle] (R) at (0,3) {$\mathcal R_{1_{\mathscr C} \to I^C \E \mathcal F_A \E v \R \R}$};
      \draw (bot) to[in=-90,out=90] (eta);
      \draw (eta) to[in=-90,out=90] (I);
      \draw (I) to[in=-90,out=90] (R);
      \draw (R) to[in=-90,out=90] (top);
    \end{tikzpicture}
  \end{align*}
  On the right hand side, taking mates, using functoriality of $\mathcal R$, and applying Lemma \ref{unit-counit mates} gives
  \begin{align*}
    \begin{tikzpicture}[smallstring, baseline=20]
      \node (top) at (1.5,4.2) {$\mathscr D \E 1_{\mathscr D} \to R_e \E I^C \E \mathcal F_A \E v \R \R \R \R$};
      \node (bot) at (0,0) {$v$};
      \node[draw, rectangle] (eta) at (0,1) {$\eta_v^C$};
      \node[draw, rectangle] (i) at (3,1) {$i_v^C$};
      \node[draw, rectangle] (R1) at (0,2) {$\mathcal R_{1_{\mathscr C} \to \mathcal F_C \E c\R}$};
      \node[draw, rectangle] (R2) at (3,2) {$\mathcal R_{\mathcal F_C \E v \R \to E^C \E \mathcal F_A \E v \R \R}$};
      \node[draw, rectangle] (circ) at (1.5,3.2) {$\blank \circ \blank$};
      \draw (bot) to[in=-90,out=90] (eta);
      \draw (eta) to[in=-90,out=90] (R1);
      \draw (i) to[in=-90,out=90] (R2);
      \draw (R1) to[in=-90,out=90] (circ.225);
      \draw (R2) to[in=-90,out=90] (circ.-45);
      \draw (circ) to[in=-90,out=90] (top);
    \end{tikzpicture}
    =
    \begin{tikzpicture}[smallstring, baseline=30]
      \node (top) at (1,4) {$\mathscr D \E 1_{\mathscr D} \to R_e \E I^C \E \mathcal F_A \E v \R \R \R \R$};
      \node (bot) at (0,0) {$v$};
      \node[draw, rectangle] (eta) at (0,1) {$\eta_v^C$};
      \node[draw, rectangle] (i) at (2,1) {$i_v^C$};
      \node[draw, rectangle] (circ) at (1,2) {$\blank \circ \blank$};
      \node[draw, rectangle] (R) at (1,3) {$\mathcal R_{1_{\mathscr C} \to I^C \E \mathcal F_A \E v \R \R}$};
      \draw (bot) to[in=-90,out=90] (eta);
      \draw (eta) to[in=-90,out=90] (circ.225);
      \draw (i) to[in=-90,out=90] (circ.-45);
      \draw (circ) to[in=-90,out=90] (R);
      \draw (R) to[in=-90,out=90] (top);
    \end{tikzpicture}
    =
    \begin{tikzpicture}[smallstring, baseline=30]
      \node (top) at (0,3) {$\mathscr D \E 1_{\mathscr D} \to R_e \E I^C \E \mathcal F_A \E v \R \R \R \R$};
      \node (bot) at (0,0) {$v$};
      \node[draw, rectangle] (i) at (0,1) {$\mate \E i_v^C \R$};
      \node[draw, rectangle] (R) at (0,2) {$\mathcal R_{1_{\mathscr C} \to I^C \E \mathcal F_A \E v \R \R}$};
      \draw (bot) to[in=-90,out=90] (i);
      \draw (i) to[in=-90,out=90] (R);
      \draw (R) to[in=-90,out=90] (top);
    \end{tikzpicture}
  \end{align*}
  which is the same as the mate of the left hand side above.
\end{proof}

Lastly, we need to work with 2-cells.  Given 1-cells $\E \mathcal R , \rho , p \R , \E \mathcal S , \sigma , q \R : \mathscr C \to \mathscr D$ and a 2-cell $\theta : \mathcal R \Rightarrow \mathcal S$, we construct $\Theta : \E R , \rho , r , p \R \Rightarrow \E S , \sigma , s , q \R$ via $\Theta_a = \theta_a$ using the identity adjunction
\begin{align*}
  \mathcal V \E 1_{\mathcal V} \to \mathscr D \E \mathcal R \E a \R \to \mathcal S \E a \R \R \R = D \E R \E a \R \to S \E a \R \R.
\end{align*}
\begin{lem}
  Under these definitions, we have $p_a \circ \Theta_{I^C \E a \R} = q_a$.
\end{lem}
\begin{proof}
  Taking mates on both sides gives exactly the coherence condition on $\theta$.
\end{proof}
\begin{prop}
  The above construction of $P_G^{\mathcal A} : \gextvmoncat \to \gextvmodtens$ defines a 2-functor.
\end{prop}
\begin{proof}
  We already have a map on 0-cells, and for each pair of 0-cells $\mathscr C$ and $\mathscr D$ in $\vmoncat$, we can use the above maps on 1-cells and 2-cells to construct the functor
  \begin{align*}
    P_G^{\mathcal A}{}_{\mathscr C \to \mathscr D} : \gextvmoncat \E ( \mathscr C , \mathcal I^C , \iota^C ) \to ( \mathscr D , \mathcal I^D , \iota^D ) \R \to \gextvmodtens \E P_G^{\mathcal A} \E  ( \mathscr C , \mathcal I^C , \iota^C ) \R \to P_G^{\mathcal A} \E ( \mathscr D , \mathcal I^D , \iota^D ) \R \R
  \end{align*}
  Now we just need to check that $P_G^{\mathcal A}$ as defined here is a 2-functor:
  \begin{itemize}
  \item For each 0-cell $\E \mathscr C , \mathcal I^C , \iota^C \R$, both $\id_{P_G^{\mathcal A} \E \mathscr C \R}$ and $P_G^{\mathcal A}{}_{\mathscr C \to \mathscr C} \E \id_{\mathscr C} \R$ are the identity 1-cell on $\mathscr C$.
  \item Composition of 1-cells is preserved by $P_G^{\mathcal A}$. To see this, notice that the only change from Section \ref{G-graded correspondence} is that we now have an extra coherence natural transformation $p$. However, $p$ is mapped directly to its mate under adjunction \ref{Main Adjunction}, and mates preserve composition (Lemma \ref{mate of fgcirc}).
  \item Composition of 2-cells is unchanged from previous sections, so both horizontal and vertical composition of 2-cells are preserved.
  \end{itemize}

\end{proof}

\subsection{$G$-extension equivalence}

Now to finish the equivalence, we need to show that $P_G^{\mathcal A}$ is fully faithful and essentially surjective.

\begin{lem}
  The 2-functor $P_G^{\mathcal A}$ is essentially surjective on 0-cells.
\end{lem}

\begin{proof}
  This is exactly Theorem \ref{Extension correspondence}.
\end{proof}
\begin{lem}
  The 2-functor $P_G^{\mathcal A}$ is also essentially surjective on 1-cells.
\end{lem}
\begin{proof}  
  Given $\E R , \rho , r , p \R : \E C , I^C , \iota^C , i^C \R \to \E D , I^D , \iota^D , i^D \R$ a 1-cell of $G$-graded extensions of $\E A , \mathcal F_A \R$, we construct $\E \mathcal R , \rho \R$ as before, and again take mates under the identity adjunction for $p$.  The only check needed to prove that $\E \mathcal R , \rho , p \R$ is a 1-cell in $\gextvmoncat$ is to show that $p$ is a $\mathcal V$-monoidal natural transformation.  Monoidality follows directly from taking mates and applying monoidality of $p$, but naturality is more involved. We need to show that
  \begin{align*}
    \begin{tikzpicture}[smallstring, baseline=30]
      \node (top) at (1,4) {$\mathscr D \E I^D \E a \R  \to R \E I^C \E b \R \R \R$};
      \node (bot) at (2,0) {$\mathcal A \E a \to b \R$};
      \node[draw, rectangle] (I) at (2,1) {$\mathcal I_{a \to b}^C$};
      \node[draw, rectangle] (p) at (0,2) {$p_a^{\mathcal R}$};
      \node[draw, rectangle] (R) at (2,2) {$\mathcal R_{I^C \E a \R \to I^C \E b \R}$};
      \node[draw, rectangle] (circ) at (1,3) {$\blank \circ \blank$};
      \draw (bot) to[in=-90,out=90] (I);
      \draw (I) to[in=-90,out=90] (R);
      \draw (p) to[in=-90,out=90] (circ.225);
      \draw (R) to[in=-90,out=90] (circ.-45);
      \draw (circ) to[in=-90,out=90] (top);
    \end{tikzpicture}
    =
    \begin{tikzpicture}[smallstring, baseline=30]
      \node (top) at (1,3) {$\mathscr D \E I^D \E a \R  \to R \E I^C \E b \R \R \R$};
      \node (bot) at (0,0) {$\mathcal A \E a \to b \R$};
      \node[draw, rectangle] (I) at (0,1) {$\mathcal I_{a \to b}^D$};
      \node[draw, rectangle] (p) at (2,1) {$p_b^{\mathcal R}$};
      \node[draw, rectangle] (circ) at (1,2) {$\blank \circ \blank$};
      \draw (bot) to[in=-90,out=90] (I);
      \draw (I) to[in=-90,out=90] (circ.225);
      \draw (p) to[in=-90,out=90] (circ.-45);
      \draw (circ) to[in=-90,out=90] (top);
    \end{tikzpicture}
  \end{align*}
  Starting with the right hand side, taking mates and using naturality of $p$ twice (note the composite laxitor), we get
  \begin{align*}
    \begin{tikzpicture}[smallstring, baseline=30]
      \node (top) at (1,5) {$R \E I^C \E b \R \R$};
      \node (botI) at (0,0) {$I^D \E a \R$};
      \node (botF) at (2,0) {$\mathcal F_D \E \mathcal A \E a \to b \R \R$};
      \node[draw, rectangle] (i) at (2,0.9) {$i_{\mathcal A \E a \to b \R}^D$};
      \node[draw, rectangle] (iota) at (1,2) {$\iota_{a , \mathcal F_A \E \mathcal A \E a \to b \R \R}^D$};
      \node[draw, rectangle] (epsilon) at (1,3) {$I^D \E \epsilon_{a \to b}^{\mathcal F_A} \R$};
      \node[draw, rectangle] (p) at (1,4) {$p_b^R$};
      \draw (botI) to[in=-90,out=90] (iota.225);
      \draw (botF) to[in=-90,out=90] (i);
      \draw (i) to[in=-90,out=90] (iota.-45);
      \draw (iota) to[in=-90,out=90] (epsilon);
      \draw (epsilon) to[in=-90,out=90] (p);
      \draw (p) to[in=-90,out=90] (top);
    \end{tikzpicture}
    =
    \begin{tikzpicture}[smallstring, baseline=30]
      \node (top) at (1,5) {$R \E I^C \E b \R \R$};
      \node (botI) at (0,0) {$I^D \E a \R$};
      \node (botF) at (2,0) {$\mathcal F_D \E \mathcal A \E a \to b \R \R$};
      \node[draw, rectangle] (i) at (2,0.9) {$i_{\mathcal A \E a \to b \R}^D$};
      \node[draw, rectangle] (iota) at (1,2) {$\iota_{a , \mathcal F_A \E \mathcal A \E a \to b \R \R}^D$};
      \node[draw, rectangle] (p) at (1,3) {$p_{a \mathcal F_A \E \mathcal A \E a \to b \R \R}^R$};
      \node[draw, rectangle] (epsilon) at (1,4) {$R \E I^C \E \epsilon_{a \to b}^{\mathcal F_A} \R \R$};
      \draw (botI) to[in=-90,out=90] (iota.225);
      \draw (botF) to[in=-90,out=90] (i);
      \draw (i) to[in=-90,out=90] (iota.-45);
      \draw (iota) to[in=-90,out=90] (p);
      \draw (p) to[in=-90,out=90] (epsilon);
      \draw (epsilon) to[in=-90,out=90] (top);
    \end{tikzpicture}
    =
    \begin{tikzpicture}[smallstring, baseline=30]
      \node (top) at (1,6) {$R \E I^C \E b \R \R$};
      \node (botI) at (0,0) {$I^D \E a \R$};
      \node (botF) at (2,0) {$\mathcal F_D \E \mathcal A \E a \to b \R \R$};
      \node[draw, rectangle] (i) at (2,1) {$i_{\mathcal A \E a \to b \R}^D$};
      \node[draw, rectangle] (p1) at (0,2) {$p_a^R$};
      \node[draw, rectangle] (p2) at (2,2) {$p_{\mathcal F_A \E \mathcal A \E a \to b \R \R}^R$};
      \node[draw, rectangle] (rho) at (1,3) {$\rho_{I^C \E a \R , I^C \E \mathcal F_A \E \mathcal A \E a \to b \R \R \R}$};
      \node[draw, rectangle] (iota) at (1,4) {$R \E \iota_{a , \mathcal F_A \E \mathcal A \E a \to b \R \R} \R$};
      \node[draw, rectangle] (epsilon) at (1,5) {$R \E I^C \E \epsilon_{a \to b}^{\mathcal F_A} \R \R$};
      \draw (botI) to[in=-90,out=90] (p1);
      \draw (botF) to[in=-90,out=90] (i);
      \draw (i) to[in=-90,out=90] (p2);
      \draw (p2) to[in=-90,out=90] (rho.-45);
      \draw (p1) to[in=-90,out=90] (rho.225);
      \draw (rho) to[in=-90,out=90] (iota);
      \draw (iota) to[in=-90,out=90] (epsilon);
      \draw (epsilon) to[in=-90,out=90] (top);
    \end{tikzpicture}
  \end{align*}
  Next, we use the coherence on $p$ and $r$ and recognize the mate of the composite functor $\mathcal I^C \circ \mathcal R$:
  \begin{align*}
    \begin{tikzpicture}[smallstring, baseline=30]
      \node (top) at (1,6) {$R \E I^C \E b \R \R$};
      \node (botI) at (0,0) {$I^D \E a \R$};
      \node (botF) at (2,0) {$\mathcal F_D \E \mathcal A \E a \to b \R \R$};
      \node[draw, rectangle] (r) at (2,1) {$r_{\mathcal A \E a \to b \R}$};
      \node[draw, rectangle] (p) at (0,2) {$p_a^R$};
      \node[draw, rectangle] (i) at (2,2) {$R \E i_{\mathcal A \E a \to b \R}^C \R$};
      \node[draw, rectangle] (rho) at (1,3) {$\rho_{I^C \E a \R , I^C \E \mathcal F_A \E \mathcal A \E a \to b \R \R \R}$};
      \node[draw, rectangle] (iota) at (1,4) {$R \E \iota_{a , \mathcal F_A \E \mathcal A \E a \to b \R \R} \R$};
      \node[draw, rectangle] (epsilon) at (1,5) {$R \E I^C \E \epsilon_{a \to b}^{\mathcal F_A} \R \R$};
      \draw (botI) to[in=-90,out=90] (p);
      \draw (botF) to[in=-90,out=90] (r);
      \draw (r) to[in=-90,out=90] (i);
      \draw (i) to[in=-90,out=90] (rho.-45);
      \draw (p) to[in=-90,out=90] (rho.225);
      \draw (rho) to[in=-90,out=90] (iota);
      \draw (iota) to[in=-90,out=90] (epsilon);
      \draw (epsilon) to[in=-90,out=90] (top);
    \end{tikzpicture}
    =
    \begin{tikzpicture}[smallstring, baseline=30]
      \node (top) at (1,3) {$R \E I^C \E b \R \R$};
      \node (botI) at (0,0) {$I^D \E a \R$};
      \node (botF) at (2,0) {$\mathcal F_D \E \mathcal A \E a \to b \R \R$};
      \node[draw, rectangle] (p) at (0,1) {$p_a$};
      \node[draw, rectangle] (mate) at (1,2) {$\mate \E \E \mathcal I^C \circ \mathcal R \R_{a \to b} \R$};
      \draw (botI) to[in=-90,out=90] (p);
      \draw (p) to[in=-90,out=90] (mate.225);
      \draw (botF) to[in=-90,out=90] (mate.-45);
      \draw (mate) to[in=-90,out=90] (top);
    \end{tikzpicture}
  \end{align*}
  and this is exactly the mate of the left hand side. Thus we get a 1-cell in $\gextvmoncat$.  With the exception of the new coherence natural transformation $p$, this map was already shown to be an essential inverse to $P_G^{\mathcal A}$ in Theorem \ref{G-graded equivalence}.  But we've also inverted $P_G^{\mathcal A}$ for $p$, since $P_G^{\mathcal A}$ takes $p$ to itself, so $P_G^{\mathcal A}$ is essentially surjective on 1-cells.
\end{proof}
\begin{lem}
  Lastly, $P_G^{\mathcal A}$ is fully faithful on 2-cells.
\end{lem}
\begin{proof}
  Given $\E R , \rho , r , p \R , \E S , \sigma , s , q \R : \E C , I^C , \iota^C , i^C \R \to \E D , I^D , \iota^D , i^D \R$ 1-cells of $G$-extensions of $\E A , \mathcal F_A \R$ and $\Theta : R \Rightarrow S$ a 2-cell between them, we construct $\theta : \mathcal R \Rightarrow \mathcal S$ via $\theta_a \coloneqq \Theta_a$, as before.  We already know that $\theta$ is a $\mathcal V$-monoidal natural transformation by the work before we introduced $G$-extensions.  The last coherence on $\theta$ follows immediately from taking mates as when we constructed $\Theta$ going the other way in the previous section.  Lastly, since we're defining $\theta = \Theta$ at the component level, this map is a bijection, and thus $P_G^{\mathcal A}$ is fully faithful on 2-cells.
\end{proof}

\bibliographystyle{amsalpha}

\bibliography{bibliography}

\end{document}